\documentclass{amsart}

\usepackage{amscd}
\usepackage{amsmath}
\usepackage{enumerate}
\usepackage[dvips]{graphicx}
\usepackage{amssymb,amsxtra,amstext,amsbsy,amsthm,latexsym,calligra,mathrsfs}    
\usepackage{amsfonts,euscript,delarray,enumitem}
\usepackage{fancyhdr,nameref,amscd}
\usepackage[all,cmtip]{xy}
\usepackage[letter,center]{crop}
\usepackage{geometry}
\usepackage{tikz}
\usetikzlibrary{arrows.meta} 
\usetikzlibrary{matrix,arrows}

\newtheorem{theorem}{Theorem}[section]
\newtheorem{lemma}[theorem]{Lemma}
\newtheorem{proposition}[theorem]{Proposition}
\newtheorem{corollary}[theorem]{Corollary}
\newtheorem{definition}[theorem]{Definition}

\newtheorem{remark}[theorem]{Remark}
\newtheorem{notation}[theorem]{Notation}
\newtheorem{example}[theorem]{Example}

\def\proof{\removelastskip\par\medskip \noindent {\sc Proof.}\enspace}
\def\endproof{\hbox{ }{\qed}\hfill\par\medskip}

\newcommand{\CC}{{\mathbb{C}}}
\newcommand{\NN}{{\mathbb{N}}}

\newcommand{\ZZ}{{\mathbb{Z}}}

\newcommand{\calA}{{\mathcal{A}}}

\newcommand{\comp}{{\circ}}

\newcommand{\smvee}{\raise0.9ex\hbox{$\scriptscriptstyle\vee$}}

\newcommand{\Ss}[1]{\mathcal{O}_{#1}}
\newcommand{\Cs}[1]{\omega_{#1}}
\newcommand{\Rs}{\tilde{X}}
\newcommand{\Rsd}{\tilde{\Sf{X}}}
\newcommand{\Sf}[1]{\mathcal{#1}}
\newcommand{\dimc}[1]{\dim_{\mathbb{C}}(#1)}

\DeclareMathOperator{\im}{Im}

\DeclareMathOperator{\codim}{codim}
\DeclareMathOperator{\corank}{corank}

\DeclareMathOperator{\Exts}{\mathscr{E}\text{\kern -3pt {\calligra\large xt}}\,}
\DeclareMathOperator{\Homs}{\mathscr{H}\text{\kern -3pt {\calligra\large om}}\,}
\DeclareMathOperator{\Hom}{\text{Hom}\,}
\DeclareMathOperator{\Ext}{\text{Ext}\,}

\begin{document}
\title[Reflexive modules on normal Gorenstein Stein surfaces]{Reflexive modules on normal Gorenstein Stein surfaces, their deformations and moduli}

\author{Javier Fern\'andez de Bobadilla}
\address{Javier Fern\'andez de Bobadilla:  
(1) IKERBASQUE, Basque Foundation for Science, Maria Diaz de Haro 3, 48013, 
    Bilbao, Bizkaia, Spain
(2) BCAM  Basque Center for Applied Mathematics, Mazarredo 14, E48009 Bilbao, 
Basque Country, Spain 
(3) Academic Colaborator at UPV/EHU} 
\email{jbobadilla@bcamath.org}
\author{Agust\'in Romano-Vel\'azquez}
\address{Agust\'in Romano-Vel\'azquez: 
(1) CIMAT, Jalisco S/N, Mineral de Valenciana, 36023, Guanajuato, M\'exico} 
\email{faustino.romano@cimat.mx}

\thanks{The first author is partially supported by IAS and by ERCEA 615655 NMST Consolidator Grant, MINECO by the project 
reference MTM2016-76868-C2-1-P (UCM), by the Basque Government through the BERC 2018-2021 program and Gobierno Vasco Grant IT1094-16, by the Spanish Ministry of Science, Innovation and Universities: BCAM Severo Ochoa accreditation SEV-2017-0718 and by Bolsa Pesquisador Visitante Especial (PVE) - Ciencias sem Fronteiras/CNPq Project number:  401947/2013-0. The second author is partially supported by IAS and by ERCEA 615655 NMST Consolidator Grant, Bolsa Pesquisador Visitante Especial (PVE) - Ciencias sem Fronteiras/CNPq Project number: 401947/2013-0, CONACYT 253506 and 286447.}

\date{13-3-2019}
\begin{abstract}
In this paper we generalize Artin-Verdier, Esnault and Wunram construction of McKay correspondence to arbitrary Gorenstein surface singularities. The key idea is the definition and a systematic use of a degeneracy module, which is an enhancement of the first Chern
class construction via a degeneracy locus. We study also deformation and moduli questions. Among our main result we quote: a full
classification of special reflexive MCM modules on normal Gorenstein surface singularities in terms of divisorial valuations centered at
the singularity, a first Chern class determination at an adequate resolution of singularities, construction of moduli spaces of special
reflexive modules, a complete classification of Gorenstein normal surface singularities in representation types, and an study on the
deformation theory of MCM modules and its interaction with their pullbacks at resolutions. For the proof of these theorems it is crucial to establish several isomorphisms between different deformation
functors, that we expect that will be useful in further work as well. 
\end{abstract}


\maketitle

\tableofcontents
\section{Introduction}
McKay correspondence~\cite{McK} is a bijection between the irreducible representations of a finite subgroup of $\mathrm{SL}(2,\CC)$ and the irreducible components of the exceptional divisor of the minimal resolution of the associated quotient surface singularity, such a bijection extends to an isomorphism of the McKay quiver (associated to the structure of representations with respect to direct sums and tensor products), and the dual graph of the exceptional divisor of the minimal resolution of singularities. McKay noticed the correspondence by a case by case study using the classification of finite subgroups of $\mathrm{SL}(2,\CC)$. After its discovery by McKay, a conceptual geometric understanding of the correspondence was achieved by a series of papers by Gonzalez-Springberg and Verdier~\cite{GoVe}, by Artin and Verdier~\cite{AV} for the correspondence at the level of vertices and by Esnault and Kn\"orrer~\cite{EsKn} at the level of edges. 

At the level of vertices the correspondence can be summarized as follows: let $\pi:\Rs\to X$ be the minimal resolution of the quotient singularity. To an irreducible representation $\rho$ one associates an indecomposable reflexive $\Ss{X}$-module $M$. The module $\pi^*M/\mathrm{Torsion}$ was proved to be locally free, and its first Chern class $c_1(\pi^*M/\mathrm{Torsion})$ is the Poincar\'e dual of a curvette hitting transversely an unique irreducible component of the exceptional divisor. Such a component is the image of the representation $\rho$ under McKay correspondence. Artin and Verdier proved that the first Chern class determines the module $M$ along with the representation $\rho$. Conversely, for any irreducible component of the exceptional divisor there is a representation and a module realizing it. We do not spell the correspondence at the level of edges since the main concern of this paper is a wide generalization of the results of Artin and Verdier at the level of vertices, and the subsequent contributions of Esnault, Wunram and Kahn which we describe below.

Esnault~\cite{Es} improved Artin and Verdier correspondence by working on arbitrary rational surface singularity. She discovered that already for quotient singularities by finite subgroups of $\mathrm{GL}(2,\CC)$ not included in $\mathrm{SL}(2,\CC)$ the pair given by the first Chern class $c_1(\pi^*M/\mathrm{Torsion})$ and the rank of the module is not enough to determine the reflexive module. A satisfactory McKay correspondence for arbitrary rational surface singularities was provided by Wunram~\cite{Wu}. For this he defined a reflexive module to be {\em special} if we have the vanishing $R^1\pi_*(\pi^*M)^{\smvee}=0$ (where $(\quad)^{\smvee}$ denotes the dual with respect to the structure sheaf). He proved that the first Chern class construction of Artin and Verdier defines a bijection between: a) the set of special indecomposable reflexive $\Ss{X}$-modules and b) the set of irreducible components of the exceptional divisor of the minimal resolution. Moreover the first Chern class $c_1(\pi^*M/\mathrm{Torsion})$ determines special reflexive $\Ss{X}$-modules. The reader may consult the survey of Riemenschneider~\cite{Rie}, for a more complete account of results described up to now, a summary of other approaches to McKay correspondence and a nice characterization of special reflexive modules.

Beyond the case of rational singularities the only complete study of reflexive modules was provided by Kahn~\cite{Ka} for the case of minimally elliptic singularities. He studied reflexive $\Ss{X}$-modules $M$ via the associated locally free $\Ss{\Rs}$-module $\Sf{M}:=(\pi^*M)^{\smvee \smvee}$ restricted to a cycle supported at the exceptional divisor. In the simply elliptic case he provides a full classification of reflexive modules building on Atiyah's classification of vector bundles on elliptic curves.  

On normal surface singularities, reflexive modules coincide with Maximal Cohen Macaulay (MCM) ones.
A singularity has finite, tame or wild Cohen-Macaulay representation type if the maximal dimension of the families of indecomposable MCM modules is 0, 1 or unbounded respectively. A consequence of Kahn's results is that simply elliptic singularities are of tame representation type. His results were completed by Drodz, Greuel and Kashuba~\cite{DrGrKa}, who proved that cusp singularities, and more generally any log-canonical surface singularity are of tame Cohen-Macaulay representation type, and posed the conjecture that non-canonical surface singularities are of wild representation type (we prove this conjecture in this paper for normal Gorenstein surface singularities). 

Due to a result of Eisenbud~\cite{Ei}, in the hypersurface case MCM modules correspond to matrix factorizations. Therefore any result proven for MCM modules have a translation into matrix factorizations. Using this equivalence, Kn\"orrer~\cite{Kn} and Buchweitz, Greuel, Schreyer~\cite{BuGrSch} proved in arbitrary dimension that the isolated hypersurface singularities of finite Cohen-Macaulay representation type are exactly the simple (ADE type) ones.

Besides the results described above the knowledge on the classification of MCM modules on other singularities available in the literature is quite limited. 

McKay correspondence admits generalizations and extensions in many directions (some of them hinted in Riemenschneider survey~\cite{Rie}). 
An important one is the study of Auslander and Reiten categories of MCM modules over a singularity. The reader may consult the book of Yoshino~\cite{Yos} for a rather complete account of known results. In this paper we focus in the classification of the objects of this categories for arbitrary normal Gorenstein singularities, leaving the structure of the category ready for later work.  

Following the line of the first Chern class correspondence described above it is natural to ask whether there is a similar description of indecomposable MCM modules on other singularities, and on the other on the characterization of the irreducible components of the exceptional divisor in terms of reflexive modules. Once the singularities are not rational, reflexive modules come in families, and deformation and moduli problems are important. We obtain quite complete answers to these questions for the case of Gorenstein normal singularities, including: 
\begin{itemize}
 \item a full classification of special reflexive MCM modules in terms of divisorial valuations centered at the singularity, which can be seen as a generalization of McKay correspondence,
 \item a first Chern class determination at an adequate resolution of singularities,
 \item construction of moduli spaces of special reflexive modules,
 \item a complete classification of Gorenstein normal surface singularities in Cohen-Macaulay representation types, confirming Drodz, Greuel and Kashuba conjecture for this class, 
 \item a study on the deformation theory of MCM modules and its interaction with their pullbacks at resolutions,
\end{itemize}

A detailed, non technical, description of the results of this paper is provided in the next section. 

Although the papers quoted above on McKay correspondence and classification of MCM modules work over singularities or complete local rings, we prove many of our results in the more general setting of reflexive modules over Stein normal surfaces (usually with Gorenstein singularities). We feel that this generalization will have interesting applications. For example it allows to apply our results to affine patches of reflexive modules on projective surfaces, opening the way to obtain applications in the global projective case. 

It should be expected that our results have applications in other areas. MCM modules on Gorenstein singularities have become recently even more important due to the equivalence with matrix factorizations explained above. Following an idea of Kontsevich, the work of Kasputin and Li (see ~\cite{KaLi}), and Orlov (see~\cite{Or1},~\cite{Or2}) showed that matrix factorizations have applications to the study of Landau-Ginzburg models appearing in string theory, and to the study of Kontsevich's homological mirror symmetry. By Khovanov and Rozansky~\cite{KoRo1},~\cite{KoRo2}, MCM modules have interesting applications to link invariants. Matrix factorizations also have   applications to  cohomological field theories~\cite{PV2}. Besides these applications, matrix factorizations have connections with representation theory, Hodge theory and other topics; for more information see for example the references of the paper by Eisenbud and Peeva~\cite{EiPe}, where matrix factorizations are generalized to complete intersections.

\section{Description of results}

In this section we offer a detailed summary of the main results and techniques contained in this paper. The reader should be able to get a picture of the paper by only reading this section. Along this section we cross-reference each of the results so the reader may jump to the appropriate part of the paper for more details. We describe each of the sections of the paper, but first we need to set the terminology.

\subsection{Setting and notation} Throughout this article, unless otherwise stated, we denote by $X$ a Stein normal surface. By $(X,x)$ we denote either a complex analytic normal surface germ, or the spectrum of
a normal complete $\CC$-algebra of dimension $2$. 

As usual $\Ss{X}$ denotes the structure sheaf and $\Ss{X,y}$ is the local ring at a 
(non-necessarily closed) point $y\in X$. In this situation $X$ has a dualizing sheaf $\omega_X$, 
and by abuse of notation we denote also by $\omega_X$ its stalk at $x$, which is the dualizing module of the ring $\Ss{X,x}$. If $X$ has Gorenstein singularities then the dualizing sheaf is an invertible sheaf. This means that if $(X,x)$ is 
Gorenstein then the dualizing module coincides with $\Ss{X,x}$. 

Unless otherwise stated, we denote by $\pi \colon \Rs\to X$ a resolution of singularities (in a few instances it will denote a proper modification from a normal 
origin). 
The exceptional set is denoted by $E:=\pi^{-1}(x)$, and its decomposition into irreducible components is $E=\cup_{i}E_i$. 
By a curvette we will mean a (multi)-germ of a curve centered at the exceptional divisor.

If $(X,x)$ is 
Gorenstein, there is a Gorenstein $2$-form $\Omega_{\Rs}$ which is meromorphic in $\Rs$, and has neither zeros nor poles in $\Rs\setminus E$; it is
called the {\em Gorenstein form}. Let $div(\Omega_{\Rs})=\sum q_iE_i$ be the divisor associated with the Gorenstein form. 
If $X$ is a Stein surface with Gorenstein singularities then there exist holomorphic $2$-forms $\Omega_{\Rs}$ in $\Rs\setminus E$ so that $div(\Omega_{\Rs})=A+\sum q_iE_i$, where $A$ is disjoint from $E$. The coefficients $q_i$ are independent of the form $\Omega_{\Rs}$ having these properties.

Next, we extend to Stein normal surfaces some usual notions for singularities and define a new one:

\begin{definition}
\label{def:smallresgor}
Let $\pi \colon \Rs\to X$ be a resolution either of a Stein normal surface with Gorenstein singularities, or of a Gorenstein surface singularity. The {\em canoninal cycle} is defined by $Z_k:=\sum_i -q_iE_i$, where the $q_i$ are the coefficients defined above. 

We say that $\Rs$ is {\em small with respect to the Gorenstein form} if $Z_k$ is greater than or equal to $0$. 

The {\em geometric genus} of $X$ is defined to be the dimension as a $\mathbb{C}$-vector space of $R^1\pi_*\Ss{\Rs}$ for any resolution.
\end{definition}

\begin{remark}
\label{rem:smallnotmodifiesreg}
If a resolution $\pi \colon \Rs\to X$ of a Stein normal surface with Gorenstein singularities is small with respect to the Gorenstein form, then it is an isomorphism over the regular locus of $X$.
\end{remark}

We will only consider resolutions which are isomorphisms over the regular locus of $X$.
 
Given $X$ a normal Stein surface, or $(X,x)$ a normal surface singularity germ, we denote by $i:U\to X$ the inclusion of $U:=X\setminus Sing(X)$ in $X$. 

\subsection{The degeneracy module correspondences}
\label{sec:introdegmodcorr}

We study reflexive modules via resolutions. Let $X$ be a Stein normal surface and $\pi:\Rs\to X$ be a resolution. Let $M$ be a reflexive $\Ss{X}$-module. Its associated full $\Ss{\Rs}$-module is $\Sf{M}:=(\pi^*M)^{\smvee \smvee}$. A basic proposition of Kahn (Proposition~\ref{fullcondiciones}) characterizes full $\Ss{\Rs}$-modules and establishes a bijection between reflexive $\Ss{X}$-modules and full $\Ss{\Rs}$-modules. 

In Section~\ref{chapter:artin-verdier-esnault} we improve and systematize Artin-Verdier's first Chern class construction by proving that there is a bijective correspondence between reflexive $\Ss{X}$-modules (respectively full $\Ss{\Rs}$-modules) of rank $r$ equipped with a system of $r$ sufficiently generic sections, and Cohen Macaulay $\Ss{X}$-modules of dimension $1$ together with a system of $r$ generators (respectively dimension $1$ Cohen-Macaulay $\Ss{\Rs}$-modules with $r$ global sections generating as a $\Ss{\Rs}$-module). These correspondences are the main tool in establishing the main results of this paper; we describe them in detail right below.

Let $\pi:Y\to X$ denote a proper birational map (which could be the identity or a resolution). Let $\Sf{M}$ be a $\Ss{Y}$-module which is locally free of rank $r$ and generated by global sections except at a finite subset of $Y$. The degeneracy $\Ss{Y}$-module $\Sf{A}$ (see Section~\ref{sec:degeneracymodule}) is defined to be the quotient of $\Sf{M}$ by the submodule generated by the sections. When $\Sf{M}$ is the sheaf of sections of a vector bundle the support of $\Sf{A}$ represents the first Chern class. We prove that $\Sf{A}$ is Cohen-Macaulay of dimension $1$, with support $A$ meeting the exceptional divisor at finitely many points, and such that $\Sf{A}_y\cong\Ss{A,y}$ for generic $y$. We call such a module a {\em generically reduced Cohen-Macaulay modules of dimension $1$}. A set of {\em nearly generic sections} is a set of $r$ sections so that its associated degeneracy module is generically reduced Cohen-Macaulay of dimension $1$. We prove in Section~\ref{sec:CM1} that such modules admit a double inclusion 
$$\Ss{A}\subset \Sf{A}\subset \Ss{\tilde{A}}$$
where $\tilde{A}$ is the normalization of $A$. Based on this inclusion we associate to $\Sf{A}$ a numerical invariant $\mathfrak{A}$ called the {\em set of orders}  to $\Sf{A}$ which consists of a subset of $\NN^l$ and is similar to the semigroup of a curve with $l$ branches. As in the case of curve semigroups the set $\mathfrak{A}$ has a {\em minimal conductor} element $cond(\mathfrak{A})$.

When $Y=\Rs$ is a resolution of singularities, for any pair $(\Sf{A},(\psi_1,...,\psi_r))$ we define the {\em Containment Condition} (see Definition~\ref{def:containment}) which measures the interaction of $(\Sf{A},(\psi_1,...,\psi_r))$ with the canonical sheaf $\omega_{\Rs}$. The following theorem is a precise statement of the correspondences announced above (see Theorems~\ref{th:corrsing} and~\ref{th:corres} in the main body of the paper)  

\begin{theorem}
\label{th:introcorr}
Let $\pi:\Rs\to X$ be a resolution of a normal Stein surface.
\begin{itemize}
 \item If $X$ has Gorenstein singularities there is a bijection between a) the set of pairs $(M,(\phi_1,...,\phi_r))$ of reflexive $\Ss{X}$-modules with a set of nearly generic sections and b) the set of pairs $(\Sf{C},(\psi_1,...,\psi_r))$ of generically reduced Cohen-Macaulay modules of dimension $1$ with $r$ generators. If the sections are generic the module $M$ has free factors if and only if the system of generators $(\psi_1,...,\psi_r)$ is not minimal.
 \item There is a bijection between a) the set of pairs $(\Sf{M},(\phi_1,...,\phi_r))$ of full $\Ss{\Rs}$-modules with a set of nearly generic sections and b) the set of pairs $(\Sf{A},(\psi_1,...,\psi_r))$ of generically reduced Cohen-Macaulay modules of dimension $1$ with $r$ generators as a $\Ss{\Rs}$-module satisfying the Containment Condition. 
\end{itemize}
\end{theorem}

The key to proof of the theorem consists mainly of a cohomological analysis for which the previous study on the structure of the degeneracy modules plays a central role. As we will see later, the degeneracy modules are the crucial new ingredient that allows us to extend McKay correspondence techniques to general surface singularities.     

In Propositions~\ref{prop:dirressing} and~\ref{prop:invressing} we investigate the relation between the correspondences at $X$ and at a resolution, and the relation of the correspondences at different resolutions.

It is not easy to handle the Containment Condition directly. In Section~\ref{sec:valuative} we introduce a numerical condition for $(\Sf{A},(\psi_1,...,\psi_r))$, called the {\em Valuative Condition} which requires the set of orders $\mathfrak{A}$ to be contained in another subset of $\NN^l$, called the {\em Canonical Set of Orders}, which codifies the interaction of $(\Sf{A},(\psi_1,...,\psi_r))$ with the canonical sheaf $\omega_{\Rs}$ (see Definition~\ref{def:canoset2}). In Proposition~\ref{prop:contval} we show that the Containment Condition implies the Valuative Condition. In Proposition~\ref{prop:consecuenciaspracticas} we prove that the converse is true in sufficiently many cases, that cover all the needs of this paper.

\subsection{The structure of reflexive modules via the degeneracy module correspondences}
\label{sec:introstruct}

Each reflexive module over a singular surface has a favorite resolution of singularities (for rational singularities it coincides with the minimal resolution). Let $M$ be a reflexive $\Ss{X}$-module of rank $r$ over a normal Stein surface $X$. The {\em minimal adapted resolution} of $M$ is the minimal resolution of singularities $\pi:\Rs\to X$ such that the full $\Ss{\Rs}$-module $\Sf{M}$ associated with $M$ is generated by global sections. In Proposition~\ref{prop:minadap} we prove its existence and uniqueness. In Proposition~\ref{prop:minadapnumchar} we characterize numerically the minimal adapted resolution of a reflexive module in terms of the minimal conductor of the Canonical Set of Orders of the pair $(\Sf{A},(\psi_1,...,\psi_r))$, obtained by applying the correspondence of Theorem~\ref{th:introcorr} to the pair $(\Sf{M},(\phi_1,...,\phi_r))$ formed by the full $\Ss{\Rs}$-module associated with $M$ and a system of generic sections.  

Similar to Wunram's generalization of McKay correspondence we obtain the most detailed results for special reflexive modules. The right generalization of special module is provided in Definitions~\ref{def:especial} and~\ref{def:espmodule}. Let $\pi:\Rs\to X$ be a resolution of a normal Stein surface, $M$ be a reflexive $\Ss{X}$-module of rank $r$ and $\Sf{M}$ its associated full $\Ss{\Rs}$-module. We say that $\Sf{M}$ is a {\em special full sheaf} if we have the equality $\dimc{R^1\pi_*\left(\Sf{M}^{\smvee}\right)}=rp_g$. The {\em specialty defect} of
$\Sf{M}$ is the number $d(\Sf{M}):=\dimc{R^1\pi_*\left(\Sf{M}^{\smvee}\right)}-rp_g$. We prove that the specialty defect is non-negative. We say $M$ to be a {\em special reflexive} module if for any resolution of singularities of $X$ its associated full sheaf is special. 

We study the behavior of the specialty defect under dominant maps $\sigma:\Rs_2\to\Rs_1$ between resolutions, and various properties of special modules in Sections~\ref{sec:minimaladapted} and \ref{sec:indecomposables}. The main results are: 
\begin{enumerate}
\item the specialty defect of the full $\Ss{\Rs_2}$-module $\Sf{M}_2$ associated to $M$ is greater than or equal to the specialty defect of the full $\Ss{\Rs_1}$-module $\Sf{M}_1$ associated to $M$. We have the equality if $\rho$ is an isomorphism over the locus where $\Sf{M}_1$ is
not generated by global sections (see Proposition~\ref{prop:specialtybehaviour}).
\item Given a reflexive $\Ss{X}$-module, if at the minimal adapted resolution the specialty defect of the full sheaf associated to $M$ vanishes, then it vanishes for any resolution (see Theorem~\ref{th:characspecial}). 
\item Assume that the Stein normal surface $X$ has Gorenstein singularities. Let $(\Sf{C},(\psi_1,...,\psi_r))$ be the pair associated with $(M,(\phi_1,...,\phi_r))$ by the correspondence of Theorem~\ref{th:introcorr}. If $M$ is a special $\Ss{X}$-reflexive module over $X$ and $(\phi_1,...,\phi_r)$ are generic sections, then $\Sf{C}$ is isomorphic to $n_*\Ss{\tilde{C}}$, where $n:\tilde{C}\to C$ is the normalization of the support of $\Sf{C}$. The converse is true (see Proposition~\ref{prop:Aeslanormalizacion}).
\item Assume now that $(X,x)$ is a normal Gorenstein surface singularity and that $M$ is special without free factors. Let $C$ be like in the previous property. There is a bijection between the indecomposable direct summands of $M$ and the irreducible components of $C$.
\end{enumerate}
Property (3) gives a geometric understanding of specialty under the degeneracy module correspondence, and also allows to produce full sheaves with prescribed Chern classes. If $(X,x)$ is a germ, the normalization map $n:\tilde{C}\to C$ is an arc at the singularity if $C$ is irreducible. Hence Property (3) also establishes a link between arc spaces and reflexive modules. Property (4) is important because it shows how the degeneracy module correspondence recovers the decomposition into indecomposables in a very geometric way. Properties (3) and (4) are false for non-special modules.

A crucial tool in the study of reflexive modules is the computation of the first cohomology of full sheaves (see Theorem~\ref{formuladimensionM}):   
\begin{theorem} \label{th:introformuladimensionM}
Let $X$ be a Stein normal surface with Gorenstein singularities. Let $M$ be a reflexive $\Ss{X}$-module of rank $r$. 
Let $\pi:\Rs\to X$ be a small resolution with respect to the Gorenstein form, let $Z_k$ be the canonical cycle at $\Rs$.
Let $\Sf{M}$ be the full $\Ss{\Rs}$-module associated to $M$. Let $d(\Sf{M})$ be the specialty defect of $\Sf{M}$. Then we have the equality
\begin{equation*}
\dimc{R^1 \pi_* \Sf{M}} = rp_g - [c_1(\Sf{M})] \cdot [Z_k]  + d(\Sf{M}).
\end{equation*}
\end{theorem}
The proof occupies the whole Section~\ref{sec:cohomology}. An interesting Corollary is the fact that the minimal adapted resolution $\pi:\Rs\to X$ of a special reflexive $\Ss{X}$-module $M$ over a Stein normal surface with Gorenstein singularities is characterized by the fact that $\dimc{R^1 \pi_* \Sf{M}} = rp_g$, where $\Sf{M}$ is the associated full $\Ss{\Rs}$-module (see Corollary~\ref{cor:dimMadap}). This has the interesting consequence that the cycle representing first Chern class of $\Sf{M}$ does not meet the support of the canonical cycle $Z_K$. It also shows that $\dimc{R^1 \pi_* \Sf{M}} = rp_g$ can be used as an invariant controlling the minimal adapted resolution process.

The above tools allow us to prove some of the main results of the paper. The first is a determination of special reflexive modules in terms of a first Chern class (Theorem~\ref{th:Chernresolucionadapted}):

\begin{theorem}\label{th:introChernresolucionadapted}
Let $X$ be a Stein normal surface with Gorenstein singularities. Let $M$ be a special $\Ss{X}$-module without free factors. Let $\pi:\Rs\to X$ be the minimal
resolution adapted to $M$, and $\Sf{M}$ the full $\Ss{\Rs}$-module associated to $M$. The module $\Sf{M}$ (and equivalently $M$) is determined by its first Chern class in $\text{Pic}(\Rs)$.
\end{theorem}

\subsection{The classification of reflexive modules}
\label{sec:introclass}

In Section~\ref{sec:combclass} we provide a combinatorial classification of reflexive modules. Given a normal surface singularity $(X,x)$ and a reflexive $\Ss{X}$-module $M$, we define its associated graph $\Sf{G}_M$ to be the dual graph of the minimal good resolution dominating the minimal adapted resolution to $M$, decorated adding as many arrows to each of its vertices $v$ as the number $c_1(\Sf{M})(E_v)$, where $E_v$ is the component of the exceptional divisor corresponding to $v$ and $c_1(\Sf{M})$ is the first Chern class of the associated full
$\Sf{M}$-module. 

Let $(\Sf{A},(\psi_1,...,\psi_r))$ be the pair associated with $(\Sf{M},(\phi_1,...,\phi_r))$ under the correspondence of Theorem~\ref{th:introcorr}, where $(\phi_1,...,\phi_r)$ are generic sections. Proposition~\ref{prop:minadapspproperty} shows that the support of $\Sf{A}$ is a disjoint union of as many smooth curvettes as arrows has $\Sf{G}_M$, each of them meeting transversely the irreducible component of the exceptional divisor corresponding to the vertex where the arrow is attached. 

In Theorem~\ref{th:charresgraphsp} we characterize combinatorially the graphs of special modules over Gorenstein surface singularities. We prove that these are precisely the graphs such that
\begin{enumerate}
\item the graph is numerically Gorenstein.
\item if a vertex has genus $0$, self intersection $-1$ and has at most two neighboring vertices, then it supports at least $1$ arrow.
\item if a vertex supports arrows then its coefficient in the canonical cycle equals $0$.
\end{enumerate}

However, a much stronger classification result is the following one (Corollary~\ref{Cor:finalprincipal} in the body of the paper). Given a normal surface singularity a {\em irreducible divisor over} $x$ is the same that a divisorial valuation of the function field of $X$ centered at $x$. An irreducible divisor over $x$ {\em appears at} a model $\pi:\Rs\to X$ if its center at $\Rs$ is a divisor.
\begin{theorem}
\label{th:introfinalprincipal}
Let $(X,x)$ be a Gorenstein surface singularity. Then there exists a bijection between the following sets: 
\begin{enumerate}
\item The set of special indecomposable reflexive $\Ss{X}$-modules up to isomorphism.
\item The set of irreducible divisors $E$ over $x$, such at any resolution of $X$ where $E$ appears, the Gorenstein form has neither 
zeros nor poles along $E$.
\end{enumerate}
\end{theorem}

This theorem specializes to classical Mckay correspondence in the case of rational double points. By taking direct sums one obtains a full classification of special reflexive modules over Gorenstein surface singularities.

\subsection{Deformations and families}
\label{def:introdefs_fam}

Having studied reflexive modules as individual objects we turn to deformations and moduli questions for the rest of the paper. Here we work as generally as possible: we allow deformations of the underlying space when we deform reflexive modules; when we deform full sheaves we allow simultaneous deformation of the space and of the resolution. 

In Section~\ref{sec:deffunctors} we define the relevant deformation functors, morphisms between them and establish the existence of versal deformations. Let $X$ be a Stein normal surface and $M$ be a reflexive $\Ss{X}$-module. 

A deformation of $(X,M)$ over a base $(S,s)$ consists of a flat deformation $\Sf{X}$ of the $X$ over $(S,s)$  together with a $\Ss{\Sf{X}}$-module which is flat over $S$ and whose fibre over $s$ is isomorphic to $M$ (see Definition~\ref{def:deforfam}). Since reflexivity is an open property in flat families this is an adequate notion of deformations of reflexive sheaves. This definition leads to a deformation functor $\mathbf{Def_{X,M}}$. The functor of deformations fixing the base space $X$ is a sub-functor.

On the other hand fullness is not an open property in flat families. Therefore the right definition of the deformation functor of full sheaves needs a cohomological condition: let $X$ and $M$ as before. Let $\pi:\Rs\to X$ be a resolution and $\Sf{M}$ the full $\Ss{\Rs}$ module associated with $M$. A deformation of $(\Rs,X,\Sf{M})$ is formed by a very weak simultaneous resolution $\Pi:\Rsd\to \Sf{X}$ of a flat deformation $\Sf{X}$ of $X$ over $S$, and a $\Ss{\Rsd}$-module $\overline{\Sf{M}}$ which is flat over $S$, whose specialization over $s$ is isomorphic to $\Sf{M}$ and such that $R^1\Pi_*\overline{\Sf{M}}$ is flat over $S$ (see Definition~\ref{def:deformationfull}). This leads to a deformation functor $\mathbf{FullDef_{\Rs,X,\Sf{M}}}$. One has subfunctors fixing the underlying space and/or the resolution. In Proposition~\ref{prop:fullopen} it is shown that fullness is an open property in families defined as above; the proof uses the flatness of $R^1\Pi_*\overline{\Sf{M}}$ in a crucial way. This shows that our definition is the correct notion of deformation within the category of full sheaves.

In Proposition~\ref{prop:naturaltrans} we show that the push-forward functor $\Pi_*$ defines a natural transformation from $\mathbf{FullDef_{\Rs,X,\Sf{M}}}$ to $\mathbf{Def_{X,M}}$. Since Kahn's result (Proposition~\ref{fullcondiciones}) establishes a bijection between reflexive $\Ss{X}$-modules and full $\Ss{\Rs}$-modules, and this bijection is via the push-forward functor, one could naively expect that this $\Pi_*$ is an isomorphism of functors. This is not the case as we will see below.
Analyzing the functors $\mathbf{FullDef_{\Rs,X,\Sf{M}}}$ and $\mathbf{Def_{X,M}}$, and the transformation $\Pi_*$ directly seems a difficult task. The extension to deformation functor isomorphisms of the correspondences of Theorem~\ref{th:introcorr} is the crucial tool in our subsequent analysis. 

\subsection{The correspondences as isomorphisms of deformation functors}
\label{def:isodef}

First we enrich the functors $\mathbf{FullDef_{\Rs,X,\Sf{M}}}$ and $\mathbf{Def_{X,M}}$ and define deformation functors 
$\mathbf{FullDef_{\Rs,X,\Sf{M}}^{(\phi_1,...,\phi_r)}}$ and $\mathbf{Def_{X,M}^{(\phi_1,...,\phi_r)}}$. Given $X$ and $M$ as above, and $(\phi_1,...,\phi_r)$ a set of $r=rank(M)$ nearly generic sections, the deformation functor $\mathbf{Def_{X,M}^{(\phi_1,...,\phi_r)}}$ associates to $(S,s)$ a deformation $(\Sf{X},\overline{M})$ in $\mathbf{Def_{X,M}}(S,s)$ along with a set of sections $(\overline{\phi}_1,...,\overline{\phi}_r)$ extending $(\phi_1,...,\phi_r)$ (see Definition~\ref{def:enhanceddef1}). The definition of $\mathbf{FullDef_{\Rs,X,\Sf{M}}^{(\phi_1,...,\phi_r)}}$ is similar (see Definition~\ref{def:enhanceddef2}). There are obvious forgetful functors from $\mathbf{Def_{X,M}^{(\phi_1,...,\phi_r)}}$ to  $\mathbf{Def_{X,M}}$, and from $\mathbf{FullDef_{\Rs,X,\Sf{M}}^{(\phi_1,...,\phi_r)}}$ to $\mathbf{FullDef_{\Rs,X,\Sf{M}}}$. 

In order to be able to prove an analog of Theorem~\ref{th:introcorr} for deformations we need deformation functors of generically reduced $1$-dimensional Cohen-Macaulay modules together with sets of generators. The relevant definitions are the following (see Definitions~\ref{def:enhanceddef1} and~\ref{def:enhanceddef2}): let $X$ be as above, let $(\Sf{C},(\psi_1,...,\psi_r))$ be a generically reduced $1$-dimensional Cohen-Macaulay $\Ss{X}$-module, together with a system of generators as a $\Ss{X}$-module. A deformation of $(X,\Sf{C},(\psi_1,...,\psi_r))$ over a germ $(S,s)$ consists of a flat deformation $\Sf{X}$ of the space $X$ over $(S,s)$, a $\Ss{\Sf{X}}$-module $\overline{\Sf{C}}$ which is flat over $S$ and specializes to $\Sf{C}$ over $s$, and a set of sections $(\overline{\psi}_1,...,\overline{\psi}_r)$ of $\overline{\Sf{C}}$ which specialize to $(\psi_1,...,\psi_r)$ over $s$. The resulting deformation functor is denoted by $\mathbf{Def_{X,\Sf{C}}^{(\psi_1,...,\psi_r)}}$. 

Like in the case of deformations of full sheaves the straightforward generalization of this functor to the case of resolutions does not work; we need to add a further condition to the families, that in a certain sense is the analog of the flatness condition of $R^1\Pi_*\overline{\Sf{M}}$ in the case of deformations of full sheaves. Let $\pi:\Rs\to X$ be a resolution of singularities and $\Sf{A}$ be a $1$-dimensional generically reduced Cohen-Macaulay $\Ss{\Rs}$-module whose support meets the exceptional divisor at finitely many points. Let $(\psi_1,...,\psi_r)$ be a set of global sections of $\Sf{A}$ generating it as a $\Ss{\Rs}$-module. A {\em specialty defect constant deformation of} $(\Sf{A},(\psi_1,...,\psi_r))$ over a germ $(S,s)$ consists of a very weak simultaneous resolution $\Pi:\Rsd\to \Sf{X}$ of a flat deformation of $X$ over $(S,s)$, a $\Ss{\Rsd}$-module $\overline{\Sf{A}}$ which is flat over $S$ and specializes to $\Sf{A}$ over $s$, and a set of sections $(\overline{\psi}_1,...,\overline{\psi}_r)$ which specialize to $(\psi_1,...,\psi_r)$ over $s$, and which are so that  the cokernel $\overline{\Sf{D}}$ of the natural mapping $\Pi_*\Ss{\Rsd}^r\to \Pi_*\overline{\Sf{A}}$ induced by the sections is flat over $S$. The resulting deformation functor is denoted by $\mathbf{SDCDef_{\Rs,X,\Sf{A}}^{(\psi_1,...,\psi_r)}}$. The functor $\Pi_*$ defines a natural transformation from $\mathbf{SDCDef_{\Rs,X,\Sf{A}}^{(\psi_1,...,\psi_r)}}$ to $\mathbf{Def_{X,\Sf{C}}^{(\psi_1,...,\psi_r)}}$, where 
$\Sf{C}$ is the submodule of $\pi_*\Sf{A}$ generated by $(\psi_1,...,\psi_r)$.

The following is our main tool in studying deformation and moduli functors (see Theorems~\ref{th:dirXdef} and~\ref{th:dirresdef} for a more precise version). 

\begin{theorem}
\label{th:introdefcorr}
Let $\pi:\Rs\to X$ be a resolution of a Stein normal surface with Gorenstein singularities. Let $M$ be a reflexive $\Ss{X}$-module of rank $r$ and $\Sf{M}$ be the associated full $\Ss{\Rs}$-module.
\begin{enumerate}
 \item Let $(\phi_1,...,\phi_r)$ be nearly generic sections of $M$ , let $(\Sf{C},(\psi_1,...,\psi_r))$ be the pair associated with $(M,(\phi_1,...,\phi_r))$ under the correspondence of Theorem~\ref{th:introcorr}. There is an isomorphism between the functors $\mathbf{Def_{X,M}}^{(\phi_1,...,\phi_r)}$ and $\mathbf{Def_{X,\Sf{C}}^{(\psi_1,...,\psi_r)}}$.
 \item Let $(\phi_1,...,\phi_r)$ be nearly generic sections of $\Sf{M}$, let $(\Sf{A},(\psi_1,...,\psi_r))$ be the pair associated with $(\Sf{M},(\phi_1,...,\phi_r))$ under the correspondence of Theorem~\ref{th:introcorr}. There is an isomorphism between the functors $\mathbf{FullDef_{\Rs,X,\Sf{M}}}^{(\phi_1,...,\phi_r)}$ and $\mathbf{SDCDef_{\Rs,X,\Sf{A}}^{(\psi_1,...,\psi_r)}}$.
\end{enumerate}
\end{theorem}

The proof of this Theorem gets quite technical and occupies several pages of the paper, but its subsequent applications makes the effort worthwhile. 

In Propositions~\ref{prop:defdirressing} and~\ref{prop:comparecorrdef} we explain the behavior of the correspondences of Theorem~\ref{th:introdefcorr} under the functor $\Pi_*$.

An important corollary of this theorem is that the specialty defect remains constant in a deformation of $\mathbf{FullDef_{\Rs,X,\Sf{M}}}$ (see Corollary~\ref{cor:specialtydefectconstant}). In Example~\ref{ex:liftsnotlifts} we give an example of deformation of a special reflexive module such that the generic member of the family is not special. As a consequence we produce a deformation which does not lift to the minimal resolution. This shows that $\Pi_*$ does not induce an isomorphism of functors from $\mathbf{FullDef_{\Rs,X,\Sf{M}}}$ to $\mathbf{Def_{X,M}}$.

In the previous example, the reason for which the natural transformation of functors $\Pi_*:\mathbf{FullDef_{\Rs,X,\Sf{M}}}\to\mathbf{Def_{X,M}}$ is not an isomorphism is that, in general, deformations in $\mathbf{Def_{X,M}}(S,s)$ do not lift to $\Rsd$. This motivates Section~\ref{sec:liftingdefs}, in which we study systematically the liftability problem for families using the correspondences of Theorem~\ref{th:introdefcorr}. 

Let $(\Sf{X},\overline{M})$ be an element in $\mathbf{Def_{X,M}}(S,s)$. Let $(\overline{\phi}_1,...,\overline{\phi}_r)$ be $r=rank(M)$ sections of $\overline{M}$ which specialize to
nearly generic sections over $s$. Let $(\Sf{X},\overline{\Sf{C}},(\psi_1,...,\psi_r))$ be the result of applying the correspondence of 
Theorem~\ref{th:introdefcorr} to $(\Sf{X},\overline{M},(\overline{\phi}_1,...,\overline{\phi}_r))$. Let $\overline{C}$ be the support of $\overline{\Sf{C}}$. We say that $C$ {\em lifts to} $\Rsd$ if the fibre over $s$ of the strict transform of $C$ by $\Pi$ coincides with the strict transform by $\pi$ of the fibre of $C$ over $s$. This notion is introduced at Definition~\ref{def:liftlocus}, where also the notion of liftability for $(\Sf{X},\overline{\Sf{C}},(\psi_1,...,\psi_r))$ is defined.

Assume that $(S,s)$ is a reduced base. In Proposition~\ref{prop:necessarylifting} we prove that there is a deformation in $\mathbf{FullDef_{\Rs,X,\Sf{M}}}(S,s)$ which transforms under $\Pi_*$ to $(\Sf{X},\overline{M})$ if and only if $(\Sf{X},\overline{\Sf{C}},(\psi_1,...,\psi_r))$ lifts to $\Rsd$ according with Definition~\ref{def:liftlocus}. Moreover this implies that the support $\overline{C}$ lifts to $\Rsd$. In Example~\ref{ex:nonlifting} we exhibit a deformation in $\mathbf{FullDef_{\Rs,X,\Sf{M}}}(S,s)$ that does not lift to $\Rsd$ because the support $\overline{C}$ does not lift. In Example~\ref{ex:liftsnotlifts} we find a deformation in $\mathbf{FullDef_{\Rs,X,\Sf{M}}}(S,s)$ that does not lift to $\Rsd$ because $(\Sf{X},\overline{\Sf{C}},(\psi_1,...,\psi_r))$ does not lift to $\Rsd$, even if the support $\overline{C}$ does lift. In Proposition~\ref{prop:genericlifting} we prove that for any deformation over a reduced base $(S,s)$ there exists a Zariski dense open subset on $(S,s)$ over which the deformation lifts to a full family. 

For our later applications we need a sufficient condition for lifting of deformations in $\mathbf{FullDef_{\Rs,X,\Sf{M}}}(S,s)$ that is easier to handle than the liftability of $(\Sf{X},\overline{\Sf{C}},(\psi_1,...,\psi_r))$ predicted in Proposition~\ref{prop:necessarylifting}.
This is worked out in Section~\ref{sec:suffcondlift}, where it is proved that under certain conditions the liftability of the support $\overline{C}$ is enough.

In Definition~\ref{def:deltaconstant} we introduce the notion of simultaneously normalizable deformations of reflexive modules: let $X$ be a normal Stein surface, let $\Sf{X}$ be a deformation of $X$ over a reduced base $(S,s)$, let $M$ be a reflexive $\Ss{X}$-module of rank $r$. A deformation $(\Sf{X},\overline{M})$ of $(X,M)$ over a reduced base $(S,s)$ is said to be {\em simultaneously normalizable} if the degeneracy locus $\overline{C}$ of $\overline{M}$ for a generic system of $r$ sections admits a simultaneous normalization over $S$. Using this definition we prove (see Theorem~\ref{th:sufficientlifting}):

\begin{theorem}
\label{th:introsufficientlifting}
Let $X$ be a normal Gorenstein surface singularity. Let $\Sf{X}$ be a deformation of $X$ over a normal base $(S,s)$. Let $\Pi:\Rsd\to \Sf{X}$ be a very weak simultaneous resolution. Let $M$ be a reflexive $\Ss{X}$-module and $(\Sf{X},\overline{M})$ be a simultaneously normalizable  deformation of $(X,M)$ over the base $(S,s)$, so that for each $s'\in S$ the module $\overline{M}|_{s'}$ is special. If the support of the degeneracy module of $\overline{M}$ for a generic system of sections is liftable to $\Rsd$, then the family $(\Sf{X},\overline{M},\iota)$ lifts to a full family on $\Rsd$.
\end{theorem}

\subsection{Moduli spaces of reflexive modules, Cohen-Macaulay representation types}
\label{sec:introapplications}

Now we describe two applications of the machinery developed up to now. 

The first application appears in Section~\ref{sec:fintamewild} and confirms a conjecture of Drodz, Greuel and Kashuba~\cite{DrGrKa} and, together with previous work in~\cite{AV} and  \cite{DrGrKa} completes the classification of Gorenstein normal surface singularities in Cohen-Macaulay representation types. Let us recall that a surface singularity $(X,x)$ is of finite, tame or wild Cohen-Macaulay representation type if there are at most finite, $1$-dimensional or unbounded dimensional families of indecomposable Maximal Cohen-Macaulay $\Ss{X,x}$-modules respectively. Here we prove (see Theorem~\ref{th:reptype}). 

\begin{theorem}
A Gorenstein surface singularity is of finite Cohen-Macaulay representation type if and only if it is a rational double point. Gorenstein surface singularities 
of tame Cohen-Macaulay representation type are precisely the log-canonical ones. The remaining Gorenstein surface singularities are of wild Cohen-Macaulay representation type.
\end{theorem}

The second application is the construction of fine moduli spaces of special modules without free factors of prescribed graph and rank on Gorenstein normal surface singularities. This enhances the classification Theorem~\ref{th:introfinalprincipal}. It is provided in Section~\ref{sec:moduli}. Let $\Sf{G}$ be the graph of a special reflexive $\Ss{X}$-module on a Gorenstein normal surface singularity $X$. Ler $r$ be a positive integer. In Definition~\ref{def:modulifunctor} a moduli functor $\mathbf{Mod_{\Sf{G}}^r}$ is defined in a similar way as the deformation functors above. It parametrizes flat families of special reflexive modules without free factors of rank $r$ and graph $\Sf{G}$, over normal base spaces. The main result is Theorem~\ref{theo:moduli}:
\begin{theorem}
The functor $\mathbf{Mod_{\Sf{G}}^r}$ is represented by an algebraic variety. 
\end{theorem}
Moreover in its proof we see that the variety representing the functor has a very nice geometric description: it parametrizes sequences of infinitely near points to $x$ in the singularity $X$ with a given combinatorial type.

\section{Reflexive modules and full sheaves}

See~\cite{BrHe},~\cite{Har1} and ~\cite{Ne} as basic references on dualizing sheaves, modules and normal surface singularities. 

\subsection{Cohen-Macaulay modules and reflexive modules}
Let $X$ be a normal surface along this section.
Let $\Homs_{\Ss{X}}(\bullet,\bullet)$ and by $\Exts^i_{\Ss{X}}(\bullet,\bullet)$ the sheaf theoretic $Hom$ and $Ext$ functors. 
The dual of a $\Ss{X}$-module $M$ is $M^{\smvee}:=\Homs_{\Ss{X}}(M,\Ss{X})$. The $\omega_X$-dual is $\Homs_{\Ss{X}}(M,\omega_X)$. 
A module $\Ss{X}$-module $M$ is called \emph{reflexive} if the natural homomorphism from $M$ to $M^{\smvee \smvee}$ is an isomorphism. It is called
$\omega_X$-reflexive if the natural map $M\to\Homs_{\Ss{X}}(\Homs_{\Ss{X}}(M,\omega_X),\omega_X)$ is an isomorphism. 

A $\Ss{X}$-module $M$ is called \emph{Cohen-Macaulay} if
the depth of each of its stalks $M_y$ is equal to the dimension of the module. If the depth of $M_y$ is equal to the dimension of $\Ss{X,y}$, 
then the module $M_y$ is called \emph{maximal Cohen-Macaulay}; this definitions extend to sheaves if we ask that they hold for every stalk. 
If $M_x$ is Cohen-Macaulay then $M$ is Cohen-Macaulay at a neighborhood of $x$.
A module is {\em indecomposable} if it can not be written as a direct sum of two non trivial submodules. 

By \cite[Proposition~1.5]{Yos} and \cite[Section~1]{Har3} some basic properties of maximal Cohen-Macaulay modules are:
\begin{enumerate}
\item If $\Ss{X,y}$ is a regular local ring, then any maximal Cohen-Macaulay module over it is free.
\item If $\Ss{X,y}$ is a reduced local ring of dimension one, then an $\Ss{X,y}$-module $M$ is maximal Cohen-Macaulay if and only if it is torsion 
free, that is, when the natural homomorphism $M \to M^{\smvee \smvee}$ is a monomorphism.
\item If $\Ss{X,y}$ is normal of dimension two, then an $\Ss{X,y}$-module $M$ is maximal Cohen-Macaulay if and only if it is reflexive.
\item A consequence of the previous properties is that reflexive sheaves over regular rings $\Ss{X,y}$ of dimension at most $2$ are free. 
\item Let $M$ be a $\Ss{X}$-module. Then $M_x$ is a reflexive $\Ss{X,x}$-module if and only if the adjunction morphism $M\to i_*i^*M$ is an 
isomorphism.
\item If a $\Ss{X}$-module is reflexive at $x$, then it is reflexive at an open neighbourhood of $x$ in $X$.
\end{enumerate}

The canonical module and Cohen-Macaulay modules have the following properties, which are a special case of~\cite[Theorem~3.3.10]{BrHe}.

\begin{theorem}
\label{Th:Herzog}
Let $X$ be a normal surface. For $t=0,1,2$ and all Cohen-Macaulay $\Ss{X}$-modules $M$ of dimension $t$ one has
\begin{enumerate}
\item $\Exts_{\Ss{X}}^{2-t}(M,\omega_X)$ is Cohen-Macaulay of dimension t,
\item $\Exts_{\Ss{X}}^{i}(M,\omega_X)=0$ for all $i \neq 2-t$,
\item there exists an isomorphism $M \to \Exts_{\Ss{X}}^{2-t}\left(\Exts_{\Ss{X}}^{2-t}(M,\omega_X),\omega_X\right)$ which in the case $t=2$ is just 
the natural 
homomorphism from $M$ into the $\omega_X$-bidual of $M$.
\end{enumerate}
\end{theorem}

A consequence of the previous Theorem and Property (3) above is that $\omega_X$-reflexivity is equivalent to reflexivity.

The following proposition will be useful:

\begin{proposition}
 \label{prop:extCM1}
Let $X$ be a normal surface.  Let $\Sf{A}$ be a $1$-dimensional $\Ss{X}$-module.
 Then the $\Ss{X}$-module $\Exts^1_{\Ss{X}}(\Sf{A},\omega_X)$ is Cohen-Macaulay of dimension $1$. 
\end{proposition}
\proof
Consider the exact sequence $0\to\Sf{B}\to\Sf{A}\to\Sf{A}'\to 0$, where $\Sf{B}$ is the submodule of elements with support at $x$. No 
element of $\Sf{A'}$ has support at $x$, and hence $\Sf{A}'$ is Cohen-Macaulay of dimension $1$. Applying $\Exts_{\Ss{X}}(\bullet,\omega_X)$
and considering the associated exact sequence we obtain the isomorphism 
$\Exts^1_{\Ss{X}}(\Sf{A},\omega_X)\cong \Exts^1_{\Ss{X}}(\Sf{A}',\omega_X)$, which implies the result by the first assertion of 
Theorem~\ref{Th:Herzog}.
\endproof

\subsection{Full sheaves}
We are interested in studying the reflexive modules on Stein normal surfaces, or in normal surface singularities,  via a resolution. For this we will use the notion of full sheaves, introduced
by Esnault~\cite{Es} for rational surface singularities and generalized by Kahn~\cite{Ka}. Along this section let $X$ be a Stein normal surface and $\pi:\Rs\to X$ be a resolution. All results of this section are valid replacing $x$ by $(X,x)$, which is either a germ of
normal surface singularity or the spectrum of a normal complete $\CC$-algebra of dimension $2$.

\begin{definition}[{\cite[Definition~1.1]{Ka}}]
A $\Ss{\Rs}$-module $\Sf{M}$ is called \emph{full} if there is a reflexive $\Ss{X}$-module $M$ such that 
$\Sf{M} \cong \left(\pi^* M\right)^{\smvee \smvee}$. A $\Ss{\Rs}$-module $\Sf{M}$ is {\em generically generated by global sections} if it is generated by global sections except in a 
finite set.
\end{definition}

\begin{proposition}[{\cite[Proposition~1.2]{Ka}}]\label{fullcondiciones}
A locally free sheaf $\Sf{M}$ on $\Rs$ is full if and only if
\begin{enumerate}
\item $\Sf{M}$ is generically generated by global sections.
\item The natural map $H^1_E(\Rs,\Sf{M}) \to H^1(\Rs,\Sf{M})$ is injective.
\end{enumerate}
If $\Sf{M}$ is the full sheaf associated to $M$, then $\pi_* \Sf{M}=M$.
\end{proposition}
\proof
Kahn's proof is for singularities. The proof for Stein normal surfaces is the same if one uses that a $\Ss{X}$-module on a Stein space is 
generated by global sections. 
The last assertion is implicit in Kahn's proof, and it gives us a natural bijection between reflexive $\Ss{X}$-modules
and full $\Ss{\Rs}$-modules.
\endproof

The following two lemmas that will be used later.
\begin{lemma} \label{lema:ceroggsg}
If $\Sf{M}$ is a full $\Ss{\Rs}$-module, then $R^1 \pi_* (\Sf{M}  \otimes \Cs{\Rs})=0$.
\end{lemma}
\proof
If $\Sf{M}$ is generated by global sections, Grauert-Riemenschneider Vanishing Theorem implies that 
$R^1 \pi_* \left( \Sf{M} \otimes \Cs{\Rs} \right)$ is equal to zero.

If $\Sf{M}$ is almost generated by global sections, consider $\Sf{M}'$ the subsheaf of $\Sf{M}$ generated by global sections, therefore we get 
the exact sequence
$0  \to \Sf{M}' \to \Sf{M} \to \Sf{G} \to 0$, with $\text{Supp}(\Sf{G})$ zero dimensional.
Applying the functor $-\otimes \Cs{\Rs}$ to the previous exact sequence, we get the desired vanishing via the long exact sequence of the 
the functor $\pi_* -$. 
\endproof

\begin{lemma}\label{lema:dualM}
If $\Sf{M}$ is a full sheaf, then $\pi_* \left(\Sf{M}^{\smvee}\right) = \left( \pi_* \Sf{M} \right)^{\smvee}$. 
\end{lemma}
\proof
Consider the following cohomology exact sequence

\begin{equation*}
\begin{tikzpicture}
  \matrix (m)[matrix of math nodes,
    nodes in empty cells,text height=1.5ex, text depth=0.25ex,
    column sep=2.5em,row sep=2em] {
  0 & H_{E}^{0}\left ( \Sf{M}^{\smvee} \right) & H^{0}\left ( \Sf{M}^{\smvee} \right) & H^{0}\left (U, \Sf{M}^{\smvee} \right) & \\
    & H_{E}^{1}\left ( \Sf{M}^{\smvee} \right) & H^{1}\left ( \Sf{M}^{\smvee} \right) & H^{1}\left (U, \Sf{M}^{\smvee} \right) & \dots \\
};
\foreach \y [remember=\y as \lasty (initially 1)] in {1, 2}
{
\foreach \x [remember=\x as \lastx (initially 2)] in {3,...,4}
{
\draw[-stealth] (m-\y-\lastx) -- (m-\y-\x);
}
}
\draw[-stealth] (m-1-1) -- (m-1-2);
\draw[-stealth] (m-2-4) -- (m-2-5);
\draw[densely dotted,-stealth] (m-1-4) to [out=355, in=175] (m-2-2);
\end{tikzpicture}
\end{equation*}

Since $\Sf{M}$ is locally free we have that
$H_{E}^{0}\left ( \Sf{M}^{\smvee} \right) = 0$,
$H_{E}^{1}\left ( \Sf{M}^{\smvee} \right) \cong H^{1}\left ( \Sf{M} \otimes \Cs{\Rs} \right)$, by Serre duality.
By Lemma~\ref{lema:ceroggsg} we get
$H^{1}\left ( \Sf{M} \otimes \Cs{\Rs} \right) = 0$. Hence $\pi_*(\Sf{M}^{\smvee})=H^{0}\left ( \Sf{M}^{\smvee} \right) \cong H^{0}\left (U, \Sf{M}^{\smvee} \right)$.

Now denote by $M:= \pi_* \Sf{M}$. Since $M^{\smvee}$ is reflexive we get the equalities
$M^{\smvee} = i_{*} i^{*} \left(M^{\smvee}\right) = i_{*}\left(\Sf{M}_{|_U}^{\smvee}\right) = H^{0}\left (U, \Sf{M}^{\smvee}\right)$.
Therefore we have the isomorphism $M^{\smvee} \cong   \pi_* \left(\Sf{M}^{\smvee}\right)$.
\endproof

Another notion that will be important in this work is the concept of specialty. Previously Wunram~\cite{Wu} and Riemenschneider~\cite{Rie} 
defined a special full sheaf as a full sheaf which its dual has the first cohomology group is equal to zero. 
Using this definition Wunram generalized McKay correspondence in the following sense: he proved that in the case of a rational surface singularity and taking the minimal resolution, there is a bijection 
between isomorphism classes of special full sheaves and irreducible components of the exceptional divisor.

Their specialty notion is adapted to the case of rational singularities. For us the definition of special is as follows. 

\begin{definition}\label{def:especial}
A full $\Ss{\Rs}$-module $\Sf{M}$ on $\Rs$ of rank $r$ is called \emph{special} if $\dimc{R^1 \pi_* \left(\Sf{M}^{\smvee}\right)} = rp_g$.
\emph{The defect of specialty of $\Sf{M}$} is the number $\dimc{R^1 \pi_* \left( \Sf{M}^{\smvee}\right)}-rp_g$.
\end{definition}

Notice that this definition is a generalization of the concept given by Wunram and Riemenschneider and both definitions coincide in the case of a
rational singularity. 

Since the definition of being special depends on the resolution, we have a another related notion.

\begin{definition}\label{def:espmodule}
Let $M$ be a reflexive $\Ss{X}$-module. We say that $M$ is \emph{a special module} if for any resolution the full sheaf associated to $M$ is 
special.
\end{definition}

Later it will become clear the importance of these concepts.

\section{Enhancing the Chern class and the degeneracy modules correspondences}\label{chapter:artin-verdier-esnault}

In~\cite{AV} Artin and Verdier interpret geometrically McKay correspondence as follows: given an indecomposable reflexive 
$\Ss{X}$-module, with $X$ a rational double point, 
they assign to it the first Chern class of the bi-dual of its pull-back to the minimal resolution. It turns out that the Chern class determines the module and that
it hits precisely the exceptional divisor that McKay correspondence associates to the module.

If $X$ is not a rational double point the fisrt Chern class does not determine the module~\cite{Es}. 
In this section we refine Artin-Verdier construction in 
the following sense. The first Chern class may be constructed as the degeneracy locus of a set of as many generic sections as the rank of the 
module. Here we, to the same set of sections, we associate a {\em degeneracy module}, which is a Cohen-Macaulay module of dimension $1$ that
whose support is the degeneracy locus. This refined correspondence is one of the main tools in our study of reflexive modules.

\subsection{Degeneracy modules of vector bundles}
\label{sec:degeneracymodule}
In this section we refine the construction of the first Chern class of a vector bundle as explained above.
In order to avoid introducing more notation we only work in the generality needed in this paper, but the construction apply to more situations.

Let $X$ be a normal Stein surface and $\pi:\Rs\to X$ a proper birational map from a normal space $\Rs$ 
({\em not necessarily a resolution, for example $\pi$ may be the identity map}). Let $E:=\pi^{-1}(x)$. Let $\Sf{M}$ be a reflexive $\Ss{\Rs}$-module 
that is generically generated by global sections. Let $S\subset E$ be
the finite subset which is the union of the singular locus $Sing(\Rs)$ and the locus where $\Sf{M}$ is not generated by global sections. Denote by $M:=\pi_*\Sf{M}$ the $\Ss{X}$-module of global sections of
$\Sf{M}$. 

Suppose that $\text{rank}(\Sf{M})=r$ and take $\phi_1, \dots, \phi_r$ generic sections. Consider the exact sequence given by the
sections
\begin{equation}
\label{exctseq:directaM}
0 \to \Ss{\Rs}^r \stackrel{(\phi_1,...,\phi_r)}{\longrightarrow} \Sf{M} \to \Sf{A}' \to 0.
\end{equation}

\begin{definition}
\label{def:degeneracymodule}
Given $\Sf{M}$ a reflexive $\Ss{\Rs}$-module of rank $r$ as above and $(\phi_1, \dots, \phi_r)$ a set of $r$ sections, the $\Ss{\Rs}$-module 
$\Sf{A}'$ defined by the previous exact sequence is called the {\em degeneracy module} of $\Sf{M}$ associated with the given sections. 
The sections are called {\em weakly generic} if the support of the degeneracy module is a proper closed subset.
\end{definition}

\begin{proposition}
\label{prop:generalizationAV0}
Given $\Sf{M}$ a $\Ss{\Rs}$-module of rank $r$ as above and a weakly generic set of $r$ sections, the associated 
degeneracy module is Cohen-Macaulay.
\end{proposition}
\proof
Dualize the exact sequence (\ref{exctseq:directaM}) with respect to a dualizing sheaf $\omega_{\Rs}$ (which exists since the surface $\Rs$ is normal, and hence Cohen-Macaulay) to obtain

\begin{equation}
0 \to \Homs_{\Rs}(\Sf{M},\omega_{\Rs}) \to \omega_{\Rs}^r \to \Exts_{\Ss{\Rs}}^1 \left(\Sf{A}', \omega_{\Rs} \right) \to 0,
\end{equation}
and dualizing again with respect to $\omega_{\Rs}$ we get
\begin{equation}\label{exact:AesCM}
0 \to \Homs_{\Rs}(\omega_{\Rs},\omega_{\Rs})^r \to \Homs_{\Ss{\Rs}} \left(\Homs_{\Ss{\Rs}} \left(\Sf{M}, \omega_{\Rs} \right), \omega_{\Rs} \right) \to 
\Exts_{\Ss{\Rs}}^1 \left(\Exts_{\Ss{\Rs}}^1 \left(\Sf{A}', \omega_{\Rs} \right), \omega_{\Rs} \right) \to 0.
\end{equation}

Theorem~\ref{Th:Herzog} applied to $\Ss{\Rs}$ yields that the first term is isomorphic to $\Ss{\Rs}^r$.
Since $\Sf{M}$ is reflexive we have that the middle term is isomorphic to $\Sf{M}$. Functoriality of the double $\Homs$ implies then that
the first morphism of the last exact sequence coincides with the first morphism of the exact sequence~(\ref{exctseq:directaM}). 
Therefore we get that $\Sf{A}'$ is isomorphic to 
$\Exts_{\Ss{\Rs}}^1 \left(\Exts_{\Ss{\Rs}}^1 \left(\Sf{A}', \omega_{\Rs} \right), \omega_{\Rs} \right)$.
Finally since $\Sf{A}'$ has dimension one we conclude that $\Sf{A}'$ is a Cohen-Macaulay $\Ss{\Rs}$-module of dimension one by 
Proposition~\ref{prop:extCM1}.
\endproof

The following Bertini-Type Proposition is a generalization of Lemma~1.2~of~\cite{AV}.

\begin{proposition}
\label{prop:generalizationAV}
Let $\Sf{M}$ a reflexive $\Ss{\Rs}$-module of rank $r$ that is generically generated by global sections as above and 
$(\phi_1, \dots, \phi_r)$ a set of $r$ generic sections. Let $C$ be the support of
the degeneracy module $\Sf{A}'$ with reduced structure. 
\begin{enumerate} 
 \item For any point $x\in\Rs$ not contained in $S$ we have the isomorphism $\Sf{A}'_x\cong\Ss{C,x}$. 
 \item The support of $C$ meets $E$ in finitely many points. 
 Let $Z\subset\Rs$ be any finite set disjoint with $S$, a sufficiently generic choice of the sections  $(\phi_1, \dots, \phi_r)$ ensures
 that $C$ is smooth outside $S$, that $C$ does not meet $Z$, and that $E$ and $C$ meet in a transversal way at those meeting points not 
 contained in $S$. However $C$ contains $S$ and at these points it may meet $E$ in a non-transversal way. 
\end{enumerate}
\end{proposition}
\proof
The restriction $\Sf{M}|_{\Rs\setminus S}$ of the sheaf $\Sf{M}$ to $\Rs\setminus S$ is locally free and generated by global sections of $\Sf{M}$
over $\Rs$. Let $\psi_1, \dots, \psi_k$ be global sections of $\Sf{M}$ such that they generate it over $\Rs\setminus S$. 
Denote by $E$ the vector bundle over $\Rs\setminus S$ whose sheaf of sections is $\Sf{M}|_{\Rs\setminus S}$ and by 
$(\Rs\setminus S) \times \mathbb{C}^k$ the trivial vector bundle of rank $k$ over $\Rs\setminus S$. 
The generating global sections $(\psi_1,...,\psi_k)$ induce a surjective morphism of vector bundles 
$$\Psi:(\Rs\setminus S) \times \mathbb{C}^k\to E,$$
defined by $\Psi(x,(c_1, \dots, c_k))=\sum_{j=1}^k \psi_j(x)c_j$. 

Since holomorphic vector bundles over Stein spaces are trivial, and $\Rs\setminus S$ admits a finite Stein cover, there exist a finite trivializing covering for $E$. Let $U$ be an trivializing open set.
Consider the local trivialization $E|_{U}\to  U \times \mathbb{C}^r$.
In the open set $U$ the global sections $\psi_1, \dots, \psi_k$ can be written as follows
\begin{equation*}A=
\begin{tikzpicture}[baseline=(m-2-1.base)]
  \matrix (m)[matrix of math nodes,
    left delimiter=(,right delimiter=)]{
    a_{11} &  a_{12} & \dots & a_{1k} \\
		\vdots &  \vdots & \vdots & \vdots \\
		a_{r1} &  a_{12} & \dots & a_{rk} \\};
\end{tikzpicture},
\end{equation*}
where the entries of the $i$-th column
are the coordinates of $\psi_i$. Notice that the matrix $A$ has entries in $\Ss{\Rs}(U)$.

In the trivialization over $U$ the restriction of the map $\Psi$ is 
$\Psi_U:U \times \mathbb{C}^k \to U \times \mathbb{C}^r$, 
where $\Psi_U(x,(c_1, \dots, c_k))=(x, A(x)(c_1, \dots, c_k)^{\intercal})$.

Now for each matrix $B$ in $\text{Mat}\left( k \times r, \mathbb{C} \right)$, we get sections $\phi_1, \dots, \phi_r$ of $\Sf{M}$ by the formula
\begin{equation*}
(\phi_1, \dots, \phi_r) = (\psi_1, \dots, \psi_k)B,
\end{equation*}
and a choice of generic sections amounts to the choice of a generic matrix $B$.

In the trivializing frame over the open set $U$ the coordinates of the sections $\phi_i$ is the $i$-th column 
of the matrix product $AB$: 
\begin{equation*}
(\phi_1, \dots, \phi_r) = AB.
\end{equation*}
So, in the open set $U$ the exact sequence \eqref{exctseq:directaM} is 
\begin{equation*}
\begin{tikzpicture}
  \matrix (m)[matrix of math nodes,
    nodes in empty cells,text height=1.5ex, text depth=0.25ex,
    column sep=3.5em,row sep=2em] {
    0 & \Ss{U}^r & \Ss{U}^r & \Sf{A}'|_U & 0. \\
		  &  \Ss{U}^k &           &         &\\};
\draw[-stealth] (m-1-1) -- (m-1-2);
\draw[-stealth] (m-1-2) edge node[auto]{$(\phi_1, \dots, \phi_r)$} (m-1-3);
\draw[-stealth] (m-1-3) -- (m-1-4);
\draw[-stealth] (m-1-4) -- (m-1-5);
\draw[-stealth] (m-1-2) edge node[auto]{$B$} (m-2-2);
\draw[-stealth] (m-2-2) edge node[below]{$A$} (m-1-3);
\end{tikzpicture}
\end{equation*}
Then we have that
\begin{equation}
\label{eq:A'explicito}
\text{Supp}(\Sf{A}')\cap U = \{x \in U \, | \, \det (AB) = 0\}.
\end{equation}

Consider the stratification by rank in the set $\text{Mat}\left( r \times r, \mathbb{C} \right)$ and  denote by
\begin{equation*}
\text{Mat}\left( r \times r, \mathbb{C} \right)^i := \{ c \in \text{Mat}\left( r \times r, \mathbb{C} \right) \, | \, \corank (c) \geq i \}.
\end{equation*}
We have that
$\codim (\text{Mat}\left( r \times r, \mathbb{C} \right)^i) = i^2$ and 
$\dim(\Rs \setminus S) =2.$

Now consider the map
$\Theta \colon \left( U \setminus S \right) \times \text{Mat}(k \times r, \mathbb{C})  \to \text{Mat}(r \times r, \mathbb{C})$, given by
$(x,B) \mapsto A(x)B$. Since the sections $\{\psi_1 ,\dots, \psi_k\}$ generate $\Sf{M}$ over the set $U \setminus S$, we get that the map 
$\Theta$ is a submersion and 
therefore it is transverse to the rank stratification.

By the Parametric Transversality Theorem, for almost every $B$ in $\text{Mat}(k \times r, \mathbb{C})$, the map
$$\Theta_B: U \setminus S \to Mat(r \times r, \mathbb{C})$$
defined by $x \mapsto A(x)B$
is transverse to the rank stratification and to the sets $Z$ and $E$. By the finiteness of the trivializing cover we can choose a matrix $B$ generic such 
that in each trivialization the map $\Theta_B$ is transverse to the rank stratification and to $Z$. 

By Equation~(\ref{eq:A'explicito}) and transversality we have that $\text{Supp}(\Sf{A}')\cap U$ is smooth of dimension $1$, disjoint to $Z$ and 
transversal to $E$. The tranversality of $\Theta_B$ to the rank stratification also implies that for any $x\in U$ we have the isomorphism
$\Sf{A}'_x\cong \Ss{C,x}$, where $C$ is the support of $\Sf{A}'$. Since the trivializing open sets cover $\Rs\setminus S$ the proposition is
proved.
\endproof

The previous Proposition motivates the following definition.

\begin{definition}
 \label{def:nearlygeneric}
 Let $\Sf{M}$ be a reflexive $\Ss{\Rs}$-module of rank $r$. A collection $(\phi_1, \dots, \phi_r)$ of $r$ sections is called {\em nearly generic} if 
 the following conditions are satisfied.
 Let $C$ be the support of the degeneracy module $\Sf{A}'$. 
\begin{enumerate}
 \item For any point $x\in\Rs$ except in a finite collection we have the isomorphism $\Sf{A}'_x\cong(\Ss{C})_x$.
 \item The support of $C$ meets $E$ in finitely many points (notice that this condition is void in the case $X=\Rs$).  
\end{enumerate}
\end{definition}

\begin{remark}
 Proposition~\ref{prop:generalizationAV} states that a generic set of sections is in particular  nearly generic.
\end{remark}

\subsection{Cohen-Macaulay modules of dimension 1}
\label{sec:CM1}

In this section we study the structure of Cohen-Macaulay modules of dimension $1$ which are of rank 1 and generically reduced. The concrete 
description that we are about to obtain will be very important in our study of reflexive modules.

\begin{definition}
\label{def:genredCM}
Let $Y$ be an analytic space and $\Sf{C}$ a $\Ss{Y}$-module dimension $1$. Let $C$ be the support of $\Sf{C}$, with reduced structure. The module $\Sf{C}$ is {\em  rank 1 generically reduced } if $\Sf{C}$ is isomorphic to $\Ss{C}$ except in finitely many points of the support.
\end{definition}

\begin{remark}
\label{rem:obsss}
Definition~\ref{def:nearlygeneric} and the first assertion of Proposition~\ref{prop:generalizationAV} states that degeneracy modules 
(see Definition~\ref{def:degeneracymodule}) for nearly generic sections are  rank 1 generically reduced  Cohen-Macaulay modules of dimension $1$.
\end{remark}

\begin{proposition}
\label{prop:genredCM}
Let $Y$ be an analytic space, and $\Sf{C}$ be a  rank 1 generically reduced  Cohen-Macaulay $\Ss{Y}$-module of dimension $1$.
Let $C$ be the support of $\Sf{C}$ with reduced structure. Denote by $n \colon \tilde{C} \to C$ the 
normalization.
\begin{enumerate}
 \item The sheaf $\Sf{C}$ is a $\Ss{C}$-module, that is, the ideal of $C$ is contained in the annihilator of $\Sf{C}$.
 \item Let $n:\tilde{C}\to C$ be the normalization. If $C$ is Stein then there exists an inclusion of $\Sf{C}$ as $\Ss{C}$-submodule of $n_*\Ss{\tilde{C}}$ which contains $\Ss{C}$. 
 In other words, we have the chain of inclusions
\begin{equation*}
 \Ss{C} \subset \Sf{C} \subset n_*\Ss{\tilde{C}}. 
\end{equation*}
\end{enumerate}
\end{proposition}

\proof
For the first assertion let $f$ be an element of the ideal of $C$. Assume that there exists a section $c$ of $\Sf{C}$ such that 
$f\cdot c$ is different from zero. Since $\Sf{C}$ is  rank 1 generically reduced  the support of $f\cdot c$ is a finite set of points, but this forces $f\cdot c$ to
vanish, since otherwise $\Sf{C}$ would not have depth $1$. 

Now we prove the second assertion. Consider the following map
\begin{align*}
h\colon \Sf{C} &\to \Sf{C} \otimes_{\Ss{C}} \Ss{\tilde{C}}/(Torsion),\\
c &\mapsto c\otimes 1.
\end{align*}

By hypothesis we have that $\Sf{C}$ is isomorphic to $\Ss{C}$ except in finitely many points. Therefore the support of the kernel of $h$ is finite and since $\Sf{C}$ does not have any finitely supported section  we get that the map $h$ is injective. 

Now notice that $n^* \Sf{C}/(Torsion) = \Sf{C} \otimes_{\Ss{C}} \Ss{\tilde{C}}/(Torsion)$ is a torsion-free $\Ss{\tilde{C}}$-module of rank one,
and hence isomorphic to $\Ss{\tilde{C}}$ (since $\tilde{C}$ is smooth and Stein). So we have an injection
$h\colon \Sf{C} \to \Ss{\tilde{C}}$. 

Consider the (multi)-germ of $\tilde{C}$ at the support of $\Ss{\tilde{C}}/h(\Sf{C})$. Enumerate the branches of the
multi-germ as $(\tilde{C}_j,p_j)$ for $j=1,...,l$. 
Each $\Ss{\tilde{C}_j,p_j}$ is a discrete valuation ring for any $j$. Let $t_j$ be a uniformizing parameter of $\Ss{\tilde{C}_j,p_j}$.
Let us denote by $h(c)_j$ the germ of $h(c)$ in $\Ss{\tilde{C}_j,p_j}$.

For any section $c$ of $\Sf{C}$ we define
$\text{ord}(h(c)) := \left(\dots,\text{ord}_{t_j}\left(h(c)_j \right),\dots \right),$
where $\text{ord}_{t_j}$ denotes the valuation of the ring $\Ss{\tilde{C}_j}$. 
Notice that $\text{ord}(h(c))$ belongs to the set $\mathbb{N}^l$.

Now for any generic $\lambda$ and $\mu$ in $\mathbb{C}$ and $c$ and $c'$ in $\Sf{C}$ we have that
\begin{equation*}
\text{ord}(h(\lambda c+ \mu c' ))= \text{min}\{\text{ord}(h(c)), \text{ord}(h(c')) \},
\end{equation*}
where the minimum is taken componentwise.

As a consequence, for a generic section $c_0$ of $\Sf{C}$ we have that $\text{ord}(h(c_0))$ is the absolute minimum of the image of $ord$. Denote by $(n_1, \dots, n_l) =\text{ord}(h(c_0))$. By genericity and the fact that $h(\Sf{C})$ spans $\Ss{\tilde{C}}$ outside 
$\{p_1,...,p_l\}$ we also know that $h(c_0)$ does not vanish outside $\{p_1,...,p_l\}$.

Consider the commutative diagram
\begin{equation*}
\begin{tikzpicture}
  \matrix (m)[matrix of math nodes,
    nodes in empty cells,text height=1.5ex, text depth=0.25ex,
    column sep=2.5em,row sep=2em] {
    \Sf{C} & \Ss{\tilde{C}}  \\
    & \Ss{\tilde{C}}[t_1^{-1},...,t_l^{-1}] \\};
\draw[-stealth] (m-1-1) edge node[auto]{$h$} (m-1-2);
\draw[-stealth] (m-1-1) edge node[below]{$g$} (m-2-2.north west);
\draw[-stealth] (m-1-2)  edge node[auto]{$\cdot \frac{1}{h(c_0)}$} (m-2-2);
\end{tikzpicture}
\end{equation*}
where $\cdot \frac{1}{h(c_0)}$ is multiplication by $\frac{1}{h(c_0)}$.

By construction the map $g$ is an $\Ss{C}$-monomorphism and the image of $g$ is contained in $\Ss{\tilde{C}}$ because 
$\text{min}\{\text{ord}(g(c)) \, | \, \text{$c$ is a section of $\Sf{C}$}\} = \text{ord}(g(c_0))=(0,\dots,0).$
Moreover we have the equality $g(c_0)=1$. This implies that $g$ provides the needed chain of embeddings.
\endproof

The previous structure result allows to define some invariants of  rank 1 generically reduced  Cohen-Macaulay modules of dimension $1$ which 
will be important for us.
Consider the notations of the previous proposition and its proof. We have defined an embedding of $\Ss{C}$-modules
$$\iota:\Sf{C}\to\Ss{\tilde{C}},$$ 
such that $\Ss{C}$ is included in $\iota(\Sf{C})$. Therefore the support of $\Ss{\tilde{C}}/\iota(\Sf{C})$ is contained in the 
pre-image by $n$ of the singular set of $C$. Let now be $\{p_1,...,p_l\}$ the pre-image by $n$ of the singular set, and denote 
by $(\tilde{C}_j,p_j)$ the germ of $\tilde{C}$ at $p_j$. As in the previous proof we have an order function
$$ord:\Ss{\tilde{C}}\to\NN^l.$$

\begin{lemma}
\label{lem:indeporderset}
The image of the composition $ord\comp\iota$ is independent of the embedding $\iota:\Sf{C}\to\Ss{\tilde{C}}$ as long as 
$\Ss{C}$ is included in $\iota(\Sf{C})$.
\end{lemma}
\proof
Any two embeddings differ by multiplication by a unit in $\Ss{\tilde{C}}$. 
\endproof

\begin{definition}
\label{def:orderset}
The {\em set of orders} $\mathfrak{C}$ of $\Sf{C}$ is the image in $\NN^l$ of the composition $ord\comp\iota$ for an embedding  
$\iota:\Sf{C}\to\Ss{\tilde{C}}$ of $\Ss{C}$-modules such that $\Ss{C}$ is included in $\iota(\Sf{C})$.
\end{definition}

Given a subset $\mathfrak{B}\subset\mathbb{Z}^l$, and a vector $(d_1,...,d_l)$, we denote by $(d_1,...,d_l)+\mathfrak{B}\subset\mathbb{Z}^l$ 
the translation of $\mathfrak{B}$ in the direction of the vector $(d_1,...,d_l)$. 

\begin{remark}
\label{rem:tanslationinv}
Since $\Sf{C}$ is a $\Ss{Y}$-module, the set of orders $\mathfrak{C}$ is stable by translation in the direction given by any vector $ord(f|_C)$ 
for $f\in\Ss{Y}$ (where $f|_C$ is the restriction of $f$ to $C$).
\end{remark}

It is well known that $\Ss{C}$ has a conductor ideal (see for example \cite[(19.21)]{Ab}). 
In our case we define the conductor of $\Sf{C}$ as follows (compare with \cite[(19.21)]{Ab}).

\begin{definition}\label{def:conductorsett}
Let $\Sf{C}$ be an $\Ss{C}$-module such that 
\begin{equation*}
\Ss{C} \subset \Sf{C} \subset n_* \Ss{\tilde{C}}.
\end{equation*}

The $\Ss{C}$-submodule
\begin{equation*}
\left \{s \in \Sf{C} \, | \, s\cdot n_* \Ss{\tilde{C}} \subset \Sf{C} \right \},
\end{equation*}
is called \emph{the conductor of $\Sf{C}$}. The image of the conductor under the order function is called the {\em conductor set of} $\Sf{C}$. 
It is a subset of $\mathfrak{C}$. Any element of the conductor set of $\Sf{C}$ is called {\em a conductor of} $\mathfrak{C}$.
\end{definition} 

\begin{remark}
\label{rem:propertiesconductor}
Since $\Ss{C}$ is contained in $\Sf{C}$, we get that the conductor of $\Sf{C}$ is non-empty. Since the conductor of $\Sf{C}$ is closed by taking 
$\Ss{X}$-linear combinations, the conductor set of $\mathfrak{C}$ has an absolute minimum, which is denoted by $cond(\mathfrak{C})$.

The conductor set of $\mathfrak{C}$ satisfies the following property: if $c=(c_1,\dots,c_l)$ is a conductor of $\mathfrak{C}$ then for any 
vector $w$ in $\mathbb{N}^l$ we have that $c+w$ belongs to the set $\mathfrak{C}$.
\end{remark}

This motivates the following

\begin{definition}
\label{def:conductorset}
Let $\mathfrak{C}$ be a subset of $\ZZ^l$, the {\em conductor set} of $\mathfrak{C}$ is the (perhaps empty) set 
\begin{equation*}
\{v\in \mathfrak{C} \,| \, v+\NN^l\subset\mathfrak{C} \}.
\end{equation*}
\end{definition}

\subsection{The correspondence at the Stein surface}
In this section we let $X$ be a Stein surface with Gorenstein singularities; in many of the cases $X$ will be a Milnor representative of a normal Gorenstein surface singularity. The Gorenstein assumption and Theorem~\ref{Th:Herzog} allow us to understand the relation between reflexive $\Ss{X}$-modules with nearly generic sections and  rank 1 generically reduced  Cohen-Macaulay $\Ss{X}$-modules with a system of generators. 

\label{sec:corrsing}
\begin{theorem}
\label{th:corrsing}
Let $X$ be a Stein surface with Gorenstein singularities. 
There is a bijective correspondence between the set of pairs $(M,(\phi_1,...,\phi_r))$ of rank $r$ reflexive $\Ss{X}$-modules with $r$ nearly
generic sections and the set of pairs $(\Sf{C},(\psi_1,...,\psi_r))$ of  rank 1 generically reduced  Cohen-Macaulay $\Ss{X}$-modules with a system 
of generators of $\Sf{C}$ as $\Ss{X}$-module. 

Under this correspondence, if the system of generators $(\psi_1,...,\psi_r)$ is not minimal then the module $M$ contains free factors. As a partial converse: if $M$ contains free factors and $(\phi_1,...,\phi_r)$ are generic, then the system of generators $(\psi_1,...,\psi_r)$ is not minimal.

The correspondence from the first set to the second is called the {\em direct correspondence at $X$}, its inverse is called the 
{\em inverse correspondence at $X$}. 
\end{theorem}
\proof
Let $M$ be a reflexive $\Ss{X}$-module of rank $r$ and $(\phi_1,...,\phi_r)$ be $r$ nearly generic sections.
We obtain the exact sequence defining the degeneracy module given by the sections
\begin{equation}
\label{exsq1}
0 \to \Ss{X}^r \to M \to \Sf{C}' \to 0.
\end{equation}
Since the sections are nearly generic the module $\Sf{C}'$ is of rank and 1 generically reduced by Remark~\ref{rem:obsss}.  

Dualizing the exact sequence with respect to $\Ss{X}$ we get
\begin{equation}
\label{exsq11}
0 \to N \to \Ss{X}^r \to \Exts_{\Ss{X}}^1 \left(\Sf{C}', \Ss{X} \right) \to 0.
\end{equation}
where $N$ is the dual of $M$.

By Proposition~\ref{prop:extCM1} the module $\Sf{C}:=\Exts_{\Ss{X}}^1 \left(\Sf{C}', \Ss{X} \right)$ is Cohen-Macaulay of dimension one. Let $C$ be the support of $\Sf{C}'$, which coincides with the support of $\Sf{C}$. A direct computation of 
$\Sf{C}$ shows that, at a smooth point $y\in X$,  if $\Sf{C}'_y$ is isomorphic to $\Ss{C,y}$ then $\Sf{C}_y$ is isomorphic to $\Ss{C,y}$
too. This shows that $\Sf{C}$ is  rank 1 generically reduced .
Therefore we associate to the reflexive module $M$ with the given sections, 
the module $\Sf{C}$ with the generators induced by the previous exact sequence.

Conversely, let $(\Sf{C},(\psi_1,...,\psi_r))$ be a  rank 1 generically reduced  $1$-dimensional Cohen-Macaulay module with a system of generators.
Define $N$ to be the kernel of the morphism $\Ss{X}^r \to \Sf{C}$ induced by the generators. We have the exact sequence
\begin{equation}
\label{exsq2}
0 \to N \to \Ss{X}^r \to \Sf{C} \to 0.
\end{equation}

Dualizing the exact sequence we get
\begin{equation}
0 \to \Ss{X}^r \to M \to \Exts_{\Ss{X}}^1\left(\Sf{C}, \Ss{X} \right) \to 0,
\end{equation}
where $M$ is the dual of $N$, and hence it is reflexive.

To the pair $(\Sf{C},(\psi_1,...,\psi_r))$ we associate the pair $(M,(\phi_1,...,\phi_r))$, where the sections are induced by the second morphism
of the previous exact sequence. The sections are nearly generic since the module $\Exts_{\Ss{X}}^1\left(\Sf{C}, \Ss{X} \right)$ is generically 
reduced, for being $\Sf{C}$  rank 1 generically reduced . 

The correspondences are mutually inverse due to part (iii) of Theorem~\ref{Th:Herzog}. 

The assertion relating free factors with minimality of the system of generators follows~\cite{Es}.
Suppose that the system of generators $(\psi_1,...,\psi_r)$ of $\Sf{C}$ is not minimal. Then an obvious computation shows that the module
of relations $N$ has free factors. Since $M$ is the dual of $N$ then $M$ has free factors. Conversely, suppose that $M$ has a free factors, that is $M\cong M_1\oplus\Ss{X}^a$. Then sequence~(\ref{exsq1}) becomes
$$0 \to \Ss{X}^{r-a}\oplus\Ss{X}^a  \to M_1\oplus\Ss{X}^a \to \Sf{C}' \to 0.$$
The genericity of the choice of the system of generators $(\phi_1,...,\phi_r)$ imply that the morphism 
$\Ss{X}^a \to \Ss{X}^a$ obtained by the triple composition of the natural inclusion of $\Ss{X}^a$ into $\Ss{X}^{r-a}\oplus\Ss{X}^a$, the first morphism of the sequence, and the canonical projection of $M_1\oplus\Ss{X}^a$ to $\Ss{X}^a$ is an isomorphism.
Dualizing we obtain 
$$0 \to N_1\oplus\Ss{X}^a \to \Ss{X}^{r-a}\oplus\Ss{X}^a \to \Sf{C} \to 0.$$
Since the corresponding morphism $\Ss{X}^a \to \Ss{X}^a$ is an isomorphism we conclude that the system of generators of $\Sf{C}$ is not minimal. 
\endproof

We may also understand the module of relations of the generators of the Cohen-Macaulay module $\Sf{C}$.

\begin{proposition}\label{inversaabajo}
Let $\Sf{C}$ be an Cohen-Macaulay $\Ss{X}$-module of dimension one, $\{\phi_1, \dots, \phi_n\}$ a set of generators of $\Sf{C}$ as 
$\Ss{X}$-module and consider the exact sequence obtained by the generators

\begin{equation}\label{exctseq:directaNabajo}
0 \to N \to \Ss{X}^r \to \Sf{C} \to 0.
\end{equation}
Then the module, $N$ is reflexive.
\end{proposition}
\proof
Dualizing the exact sequence \eqref{exctseq:directaNabajo} and denoting by $M:= N^{\smvee}$, we obtain the exact sequence
\begin{equation}\label{exctseq:directaMabajo}
0 \to \Ss{X}^r \to M \to \Exts_{\Ss{X}}^1 \left(\Sf{C}, \Ss{X}\right) \to 0.
\end{equation}

Since $\Sf{C}$ is Cohen-Macaulay of dimension one then by Theorem~\ref{Th:Herzog} the module
$\Exts_{\Ss{X}}^1 \left(\Sf{C},\Ss{X}\right)$ is Cohen-Macaulay of dimension one and 
$\Sf{C} \cong \Exts_{\Ss{X}}^1 \left( \Exts_{\Ss{X}}^1 \left(\Sf{C},\Ss{X}\right),\Ss{X}\right)$. 
Now dualizing \eqref{exctseq:directaMabajo} and using the previous identification we obtain the exact sequence
\begin{equation*}
0 \to N^{\smvee \smvee} \to \Ss{X}^r \to \Sf{C} \to 0,
\end{equation*}
hence we have the identification $N=N^{\smvee \smvee}$ and $N$ is reflexive.
\endproof

\subsection{The correspondence at the resolution}
\label{sec:corrres}

In this section we generalize the correspondence given by Artin-Verdier~\cite{AV}, Esnault~\cite{Es} and Wunram~\cite{Wu} 
at the resolution to the general case, that is $(X,x)$ is any normal surface singularity; in fact we go further and allow $X$ to be a Stein normal surface with possibly several singularities. We will obtain a bijection as in previous section. One of the sets is formed by pairs of modules together with systems
of nearly generic sections. The other set is formed by  rank 1 generically reduced  $1$-dimensional Cohen-Macaulay modules together with sets of 
sections that satisfy a certain property. Before we state the main result of the section we need to explain what the property means precisely. 

\subsubsection{The Containment Condition}
\label{sec:containment}
Let $\pi:\Rs\to X$ be a resolution of singularities of a Stein normal surface $X$ which is an isomorphism at the regular locus of $X$. Let $E$ denote the exceptional divisor and $U=\Rs\setminus E$.
Let $\Sf{A}$ be a  rank 1 generically reduced  $1$-dimensional Cohen-Macaulay $\Ss{\Rs}$-module.
Let $(\psi_1,...,\psi_r)$ be $r$ global sections
spanning $\Sf{A}$ as $\Ss{\Rs}$-module. The set of sections $(\psi_1,...,\psi_r)$ define a morphism $\Ss{\tilde{X}}^r\to\Sf{A}$. Tensoring with the dualizing sheaf $\omega_{\tilde{X}}=\Lambda^2\Omega^1_{\Rs}$ and taking sections in $U$ we obtain a morphism 
\begin{equation}
\label{eq:deltaprimeravez}
\delta:H^0(U,\omega_{\tilde{X}}^r)\to H^0(U,\Sf{A}\otimes\omega_{\tilde{X}}).
\end{equation}
We also have a restriction morphism
$$\gamma_1:H^0(\Rs,\Sf{A}\otimes\omega_{\tilde{X}})\to H^0(U,\Sf{A}\otimes\omega_{\tilde{X}}).$$

\begin{definition}
\label{def:containment}
The pair $(\Sf{A},(\psi_1,...,\psi_r))$ satisfies the {\em Containment Condition} if we have the inclusion $\mathrm{Im}\gamma_1\subset \mathrm{Im}\delta$.
\end{definition}

\subsubsection{The correspondence at the resolution}

\begin{theorem}
\label{th:corres}
Let $\pi:\Rs\to X$ be a resolution of singularities of a Stein normal surface which is an isomorphism at the regular locus of $X$. There is a bijective correspondence between the set of pairs $(\Sf{M},(\phi_1,...,\phi_r))$ formed by a locally free $\Ss{\Rs}$-module which is almost generated by global sections, and a set of $r$ nearly generic sections, and the set of pairs $(\Sf{A},(\psi_1,...,\psi_r))$ formed by a  rank 1 generically reduced  $1$-dimensional Cohen-Macaulay $\Ss{\Rs}$-module, whose support meets $E$ in finitely many points, and a set of $r$ global sections spanning $\Sf{A}$ as $\Ss{\Rs}$-module.

Moreover $\Sf{M}$ is full if and only if $(\Sf{A},(\psi_1,...,\psi_r))$ satisfies the Containment Condition (see Definition~\ref{def:containment}). 
\end{theorem}
\proof

We start defining the first bijection.

Given $(\Sf{M},(\phi_1,...,\phi_r))$ we consider the exact sequence induced by the sections:
\begin{equation}
\label{exctseq:da1}
0 \to \Ss{\Rs}^r \to \Sf{M} \to \Sf{A}' \to 0.
\end{equation}

The degeneracy module $\Sf{A}'$ is $1$-dimensional Cohen-Macaulay by Proposition~\ref{prop:generalizationAV0}, and is  rank 1 generically reduced 
with support intersecting the exceptional divisor $E$ in a finite set, by definition of nearly-generic sections (Definition~\ref{def:nearlygeneric}).
Dualizing the sequence we obtain 

\begin{equation}\label{exact:da2}
0 \to \Sf{N} \to \Ss{\Rs}^r  \to \Sf{A} \to 0,
\end{equation}
where $\Sf{A}=\Exts^1_{\Ss{\Rs}}(\Sf{A}',\Ss{\Rs})$ is a  rank 1 generically reduced  $1$-dimensional Cohen-Macaulay module (same arguments than in the 
proof of Theorem~\ref{th:corrsing}), whose support meets $E$ in finitely many points (since it coincides with the support of $\Sf{A'})$. Let 
$(\psi_1,...,\psi_r)$ be the generators of $\Sf{A}$ as $\Ss{\Rs}$-module given by the previous exact sequence. 

We define the direct correspondence as the correspondence sending the pair $(\Sf{M},(\phi_1,...,\phi_r))$ to the pair 
$(\Sf{A},(\psi_1,...,\psi_r))$. 

Conversely, given $(\Sf{A},(\psi_1,...,\psi_r))$ we consider the exact sequence~(\ref{exact:da2}) given by the sections. Dualizing it we obtain 
the exact sequence~(\ref{exctseq:da1}), and we define the inverse correspondence sending $(\Sf{A},(\psi_1,...,\psi_r))$ to $(\Sf{M},(\phi_1,...,\phi_r))$,
where $(\phi_1,...,\phi_r)$ are the sections induce by the sequence~(\ref{exctseq:da1}). 

The direct and inverse correspondences are inverse to each other, for the same reasons appearing in the proof of Theorem~\ref{th:corrsing}.

In order to prove the Theorem we have to show that $\Sf{M}$ is full if and only if $(\Sf{A},(\psi_1,...,\psi_r))$ satisfies the Containment Condition. For this we use the characterization of Proposition~\ref{fullcondiciones}.

We start showing that the inverse correspondence always gives a $\Ss{\Rs}$-module that is generically generated by global sections. This is stated
in a separate lemma.

\begin{lemma}
\label{lem:gengen}
If $(\Sf{A},(\psi_1,...,\psi_r))$ is formed by a  rank 1 generically reduced  $1$-dimensional Cohen-Macaulay 
$\Ss{\Rs}$-module, whose support meets $E$ in finitely many points, and a set of $r$ global sections spanning $\Sf{A}$ as $\Ss{\Rs}$-module, then
the $\Ss{\Rs}$-module $\Sf{M}$ obtained by applying inverse correspondence is generically generated by global sections.
\end{lemma}
\proof
Applying the functor $\pi_* -$ to the exact sequence \eqref{exctseq:da1} we get
\begin{equation*}
0 \to \Ss{X}^r \to \pi_*\Sf{M} \to \pi_*\Sf{A}' \to R^1 \pi_* \Ss{\Rs}^r \to R^1 \pi_*\Sf{M} \to 0.
\end{equation*}

Denote by $\Sf{G}$ the image of $\pi_* \Sf{M}$ in $\pi_*\Sf{A}'$, so we obtain the following two exact sequences of $\Ss{X}$-modules
\begin{equation*}
0 \to \Ss{X}^r \to \pi_*\Sf{M} \to \Sf{G} \to 0,
\end{equation*}
\begin{equation} \label{exctseq:ggsg}
0 \to \Sf{G} \to \pi_* \Sf{A}' \to R^1 \pi_* \Ss{\Rs}^r \to R^1 \pi_*\Sf{M} \to 0.
\end{equation}
Since the support of $\Sf{A}'$ intersects the exceptional divisor in a finite collection of points, then we can identify $\pi_* \Sf{A}'$ with $\Sf{A}'$, viewed as a $\Ss{X}$-module. Then $\Sf{G}$ is a sub $\Ss{X}$-module of $\Sf{A'}$.

Denote by $\Sf{M}'$ the subsheaf of $\Sf{M}$ generated by its global sections, and by $\Sf{G}'$ the sub-$\Ss{\Rs}$-module of $\Sf{A}'$ spanned by $\Sf{G}$. We have the following diagram of coherent 
$\Ss{\Rs}$-modules
\begin{equation*}
\begin{tikzpicture}
  \matrix (m)[matrix of math nodes,
    nodes in empty cells,text height=1.5ex, text depth=0.25ex,
    column sep=2.5em,row sep=1.5em]{
		&  0          & 0       &    0    &   \\
  0 &  \Ss{\Rs}^r & \Sf{M}' & \Sf{G'}  & 0 \\
  0 &  \Ss{\Rs}^r & \Sf{M}  & \Sf{A}' & 0 \\
	  &  0          & \Sf{F}  & \Sf{F}' & 0  \\
		&             & 0       &  0      &   \\
};
\foreach \y [remember=\y as \lasty (initially 1)] in {2,3,4}
{
\foreach \x [remember=\x as \lastx (initially 2)] in {3,4}
{
\draw[-stealth] (m-\y-\lastx) -- (m-\y-\x);
\draw[-stealth] (m-\lasty-\x) -- (m-\y-\x);
}
}
\draw[-stealth] (m-1-2) -- (m-2-2);
\draw[-stealth] (m-2-2) -- (m-3-2);
\draw[-stealth] (m-2-1) -- (m-2-2);
\draw[-stealth] (m-3-1) -- (m-3-2);
\draw[-stealth] (m-2-4) -- (m-2-5);
\draw[-stealth] (m-3-4) -- (m-3-5);
\draw[-stealth] (m-3-2) -- (m-4-2);
\draw[-stealth] (m-4-4) -- (m-4-5);
\draw[-stealth] (m-4-3) -- (m-5-3);
\draw[-stealth] (m-4-4) -- (m-5-4);
\end{tikzpicture}
\end{equation*}
It is enough to prove that the support of $\Sf{F}$, which coincides with the support of $\Sf{F}'$, is a finite set. For this it is enough to show that $\dimc{\Sf{F}'}=\dimc{coker(\Sf{G}'\to\Sf{A}')}<\infty$. But we have the inequalities
$$\dimc{coker(\Sf{G}'\to\Sf{A}')}\leq \dimc{coker(\Sf{G}\to\pi_*\Sf{A}')}\leq rp_g$$
by the Exact Sequence~(\ref{exctseq:ggsg}).
\endproof

In order to finish the proof of Theorem~\ref{th:corres} we have to prove that the map 
$H^1_E(\Sf{M})\to H^1(\Sf{M})$ is injective if and only if $(\Sf{A},(\psi_1,...,\psi_r))$ satisfies the Containment Condition. 

By Serre duality the Containment Condition is equivalent to the surjection of the natural map 
$H^1_E\left(\Sf{M}^{\smvee} \otimes \Cs{\Rs} \right) \to H^1\left(\Sf{M}^{\smvee} \otimes \Cs{\Rs} \right)$, hence we will study this map. As before we denote $\Sf{N}:=\Sf{M}^{\smvee}$.

Apply the functor $- \otimes \Cs{\Rs}$ to the exact sequence~(\ref{exact:da2}),
\begin{equation}\label{exctseq:directaWN}
0 \to \Sf{N}\otimes \Cs{\Rs} \to \Cs{\Rs}^r \to \Sf{A}\otimes \Cs{\Rs} \to 0.
\end{equation}

Taking the long exact sequence of co\-ho\-mo\-lo\-gy and local co\-ho\-mo\-lo\-gy for the previous exact sequence we obtain a diagram of exact sequences:

\begin{tikzpicture}
  \matrix (m)[matrix of math nodes,
    nodes in empty cells,text height=1.5ex, text depth=0.25ex,
    column sep=1.7em,row sep=2em]{
  H^{0}_E\left(\Sf{N} \otimes \Cs{\Rs}\right) & H^{0}_E\left(\Cs{\Rs}^r\right) & H^0_E\left(\Sf{A}\otimes \Cs{\Rs}\right) & H^{1}_E\left(\Sf{N}\otimes \Cs{\Rs}\right)& H^{1}_E\left(\Cs{\Rs}^r\right) \\
  H^{0}\left(\Sf{N}\otimes \Cs{\Rs}\right) & H^{0}\left(\Cs{\Rs}^r\right) & H^0\left(\Sf{A}\otimes \Cs{\Rs}\right) & H^{1}\left(\Sf{N}\otimes \Cs{\Rs}\right)& H^{1}\left(\Cs{\Rs}^r\right)\\
  H^{0}\left(U;\Sf{N}\otimes \Cs{\Rs}\right) & H^{0}\left(U;\Cs{\Rs}^r\right) & H^0\left(U;\Sf{A}\otimes \Cs{\Rs}\right) & H^{1}\left(U;\Sf{N}\otimes \Cs{\Rs}\right)& H^{1}\left(U;\Cs{\Rs}^r\right)\\
};
\foreach \x [remember=\x as \lastx (initially 2)] in {3,...,5}
{
\draw[-stealth] (m-1-\lastx) -- (m-1-\x);
\draw[-stealth] (m-2-\lastx) -- (m-2-\x);
\draw[-stealth] (m-3-\lastx) -- (m-3-\x);
\draw[-stealth] (m-2-\lastx) -- (m-3-\lastx);
}
\draw[-stealth] (m-2-1) -- (m-3-1);
\draw[-stealth] (m-2-5) -- (m-3-5);
\draw[-stealth] (m-1-4) -- (m-2-4);
\draw[-stealth] (m-1-5) -- (m-2-5);
\draw[{Hooks[right]}-{stealth}] (m-1-1) -- (m-2-1);
\draw[{Hooks[right]}-{stealth}] (m-1-2) -- (m-2-2);
\draw[{Hooks[right]}-{stealth}] (m-1-3) -- (m-2-3);
\draw[{Hooks[right]}-{stealth}] (m-1-1) -- (m-1-2);
\draw[{Hooks[right]}-{stealth}] (m-2-1) -- (m-2-2);
\draw[{Hooks[right]}-{stealth}] (m-3-1) -- (m-3-2);
\draw[densely dotted,-stealth] (m-3-1.south) to [out = -90, in = 90, looseness = .7] (m-1-4.north);
\draw[densely dotted,-stealth] (m-3-2.south) to [out = -90, in = 90, looseness = .7] (m-1-5.north);
\end{tikzpicture}

We have $H^1(\Cs{\Rs})=0$ by Grauert-Riemenschneider Vanishing Theorem and $H^0_E(\Sf{A}\otimes \Cs{\Rs})=0$ because $\Sf{A}\otimes \Cs{\Rs}$ has depth one and its support intersects the exceptional divisor in a finite set. Therefore we have the following diagram of exact sequences 
\begin{equation}\label{subdiagram}
\begin{tikzpicture}
  \matrix (m)[matrix of math nodes,
    nodes in empty cells,text height=1.5ex, text depth=0.25ex,
    column sep=2.5em,row sep=2em] {
    \quad & 0 & H_{E}^1\left(\Sf{N} \otimes \Cs{\Rs}\right) & \quad \\
    H^{0}\left(\Cs{\Rs}^r\right) & H^0\left(\Sf{A}\otimes \Cs{\Rs}\right) & H^{1}\left(\Sf{N} \otimes \Cs{\Rs}\right) & 0 \\
    H^{0}\left(U;\Cs{\Rs}^r\right) & H^0\left(U;\Sf{A}\otimes \Cs{\Rs}\right) & H^{1}\left(U;\Sf{N} \otimes \Cs{\Rs}\right) & \quad \\};

\draw[-stealth] (m-1-2) -- (m-2-2);
\draw[-stealth] (m-2-1) -- node[auto]{$\beta$} (m-2-2);
\draw[-stealth] (m-2-2) -- node[auto]{$\alpha$} (m-2-3);
\draw[-stealth] (m-2-3) -- (m-2-4);
\draw[-stealth] (m-3-1) -- node[auto]{$\delta$} (m-3-2);
\draw[-stealth] (m-3-2) -- node[auto]{$\gamma_2$} (m-3-3);

\draw[-stealth] (m-2-3) -- node[auto]{$\varphi$} (m-3-3);
\draw[-stealth] (m-2-1) -- (m-3-1);
\draw[-stealth] (m-2-2) -- node[auto]{$\gamma_1$} (m-3-2);
\draw[-stealth] (m-1-3) -- node[auto]{$\theta$} (m-2-3);
\end{tikzpicture}
\end{equation}

A diagram chase shows that 
$H_{E}^1(\Sf{N} \otimes \Cs{\Rs}) \stackrel{\theta}{\longrightarrow} H^1(\Sf{N} \otimes \Cs{\Rs})$
is an epimorphism if and only if
\begin{equation*}
\im \gamma_1 \subset \im \delta,
\end{equation*}
which is precisely the Containment Condition.
\endproof

\begin{remark}
Working with the Containment Condition is quite difficult. Because of this, in the next section, we introduce a numerical condition which is implied by the Containment Condition, and that, in sufficiently many cases for our applications, is equivalent to it.
\end{remark}

\subsubsection{The Valuative Condition}
\label{sec:valuative}
Let $\pi:\Rs\to X$ be a resolution of singularities of a Stein normal surface $X$ which is an isomorphism over the regular locus of $X$. Let $E$ denote the exceptional divisor and $U:=\Rs\setminus E$.
Let $C\subset X$ be a curve and $\overline{C}$ be its strict transform to $\Rs$. Let $n:\tilde{C}\to \overline{C}$ be the normalization.
Let $\{p_1,...,p_l\}$ be the preimage by $n$ of $E$. Let $(\tilde{C}_j,p_j)$ be the germ of $\tilde{C}$ at $p_j$. The ring $\Ss{\tilde{C}_j,p_j}$ is a discrete valuation ring, and its valuation is denoted by $ord_{\tilde{C_j}}$.

Let $\beta$ be a meromorphic differential $2$-form in  $H^0(U,\omega_{\Rs})$. We define a $l$-uple $ord(\beta)$ in $\left(\ZZ\cup\{+\infty\}\right)^l$ as follows. 
For any $j$ we let $q_j:=n(p_j)$ and choose a non-vanishing holomorphic differential $2$-form 
germ $\omega_{q_j}$ at $q_j$. Then $\beta=h_j\omega_{q_j}$ where $h_j$ is a meromorphic function. Define 
\begin{equation}
\label{eq:orddiffform}
ord(\beta):=(ord_{\tilde{C_1}}(h_1),...,ord_{\tilde{C_l}}(h_l)).
\end{equation}
It is clear that the definition does not depend on the choice of the forms $\omega_{q_j}$. This defines an order function
\begin{equation}
\label{eq:ordercanonical}
ord:H^0(U,\omega_{\Rs})\to(\ZZ\cup\{\infty\})^l.
\end{equation}

\begin{definition}
\label{def:canoset}
The {\em canonical set of orders of the curve} $C$ {\em at the resolution} $\pi$ is the set 
$$\mathfrak{K}_\pi:=ord(H^0(U,\omega_{\Rs}))\subset\ZZ^l.$$
\end{definition}

Given two subsets $\mathfrak{A},\mathfrak{B}\subset\ZZ^l$ we denote by $\mathfrak{A}+\mathfrak{B}$ the subset of sums $a+b$ where 
$a\in\mathfrak{A}$ and $b\in\mathfrak{B}$. 

Let $\mathfrak{S}$ be the semigroup of orders of $\Ss{C}$; since $H^0(U,\omega_{\Rs})$ is a 
$\Ss{X}$-module, we have the equality $\mathfrak{K}_\pi+\mathfrak{S}=\mathfrak{K}_\pi$.

Given $\alpha,\beta\in H^0(U,\omega_{\Rs})$, for generic $\lambda,\mu\in\CC$ we have the equality
\begin{equation}
\label{eq:ordenminimoforma}
ord(\lambda\alpha+\mu\beta)=min(ord(\alpha),ord(\beta));
\end{equation}
hence the canonical set of orders of the curve $C$ has an absolute minimum.

\begin{definition}
\label{def:tranlation}
We define the {\em canonical vector} $(d_1(\overline{C}),...,d_l(\overline{C}))$ of the curve $\overline{C}\in\Rs$ to be the absolute minimum
of the canonical set of orders. 
\end{definition}

\begin{remark}
\label{rem:canonicalcondset}
Since we have the equality $\mathfrak{K}_\pi+\mathfrak{S}=\mathfrak{K}_\pi$, the set $(d_1(\overline{C}),...,d_l(\overline{C}))+\mathfrak{S}$ is
included in $\mathfrak{K}_\pi$. Then $\mathfrak{K}_\pi$ has a non-empty conductor set, which contains the conductor set of $\mathfrak{S}$ translated
by the vector $(d_1(\overline{C}),...,d_l(\overline{C}))$. Since $\mathfrak{K}_\pi$ is closed by taking minima, the conductor set of 
$\mathfrak{K}_\pi$ has a unique absolute minimun, which we denote by $cond(\mathfrak{K}_\pi)$.  
\end{remark}

Now assume that $X$ has Gorenstein singularities. Then there exist a holomorphic 2-form $\Omega\in H^0(U,\omega_{\Rs})$ whose associated divisor is 
\begin{equation*}
\text{div}(\Omega) =A+ \sum q_iE_i,
\end{equation*}
where each $q_i$ is a integer, the $E_i$'s are the irreducible components of the exceptional divisor and $A$ is a divisor disjoint with 
$E$. The integers $q_i$ are independent on the choice of $\Omega$. We call $\Omega$ a {\em Gorenstein form}.

\begin{remark}
\label{rem:translationgor}
If $X$ has Gorenstein singularities, and adopting the previous notation, the canonical vector $(d_1(\overline{C}),...,d_l(\overline{C}))$ of the curve $\overline{C}$
is given by the formulae
$$d_i(\overline{C}):=\sum_{j}q_j\overline{C}_i\cdot E_j,$$
where $\overline{C}_i\cdot E_j$ denotes intersection multiplicity.
\end{remark}

\begin{remark}
\label{rem:gorordcan}
If $X$ has Gorenstein singularities then the canonical set of orders of the curve $C$ at the resolution $\pi$ is equal to 
$(d_1(\overline{C}),...,d_l(\overline{C}))+\mathfrak{S}$, and its conductor set is the conductor set 
of $\mathfrak{S}$ translated by the vector $(d_1(\overline{C}),...,d_l(\overline{C}))$.
\end{remark}

Let $\Sf{A}$ be a  rank 1 generically reduced  $1$-dimensional Cohen-Macaulay $\Ss{\Rs}$-module whose support equals $\overline{C}$. Let $\mathfrak{A}$ be its set of orders 
(see Definition~\ref{def:orderset}).
Let $(\psi_1,...,\psi_r)$ be $r$ global sections
spanning $\Sf{A}$ as $\Ss{\Rs}$-module. The $\Ss{X}$-module $\Sf{C}$ spanned by $(\psi_1,...,\psi_r)$ is  rank 1 generically reduced  
$1$-dimensional Cohen-Macaulay. Let $\mathfrak{C}$ be the set of orders of $\Sf{C}$, which is the subset of $\mathfrak{A}$ obtained following Definition~\ref{def:orderset}.

The set of sections $(\psi_1,...,\psi_r)$ define a epimorphism $\Ss{\tilde{X}}^r\to\Sf{A}$. Tensoring with $\omega_{\tilde{X}}$ we obtain an 
epimorphism $\omega_{\tilde{X}}^r\to\Sf{A}\otimes\omega_{\tilde{X}}$. Taking global sections in $U$ we obtain a morphism 
\begin{equation}
\label{eq:deltaprimeravez2}
\delta:H^0(U,\omega_{\tilde{X}}^r)\to H^0(U,\Sf{A}\otimes\omega_{\tilde{X}}).
\end{equation}

Since $\omega_{\tilde{X}}$ is locally free of rank $1$ and $\Sf{A}$ has Stein support, the restriction of $\omega_{\tilde{X}}$ 
to the support of $\Sf{A}$ is isomorphic to the structure sheaf of the support. Hence there is an isomorphism
$\Sf{A}\cong \Sf{A}\otimes\omega_{\tilde{X}}$. We have fixed an embedding of $\Sf{A}$ into $\Ss{\tilde{C}}$, which naturally embeds
$H^0(U,\Sf{A})$ into the 
total fraction ring $K(\tilde{C})$. Such total fractions ring maps into the direct sum ring $\oplus_{i=1}^l K(\tilde{C}_i,p_i)$ (here $K(\tilde{C}_i,p_i)$ is the quotient field of $\Ss{\tilde{C}_i,p_i}$).

\begin{definition}
\label{def:canoset2}
We define the canonical set of orders of $(\Sf{A},(\psi_1,...,\psi_r))$ to be the image $\mathfrak{K}_{(\Sf{A},(\psi_1,...,\psi_r))}$ of the composition
\begin{equation}
H^0(U,\omega_{\tilde{X}}^r)\stackrel{\delta}{\longrightarrow}H^0(U,\Sf{A}\otimes\omega_{\tilde{X}})\cong H^0(U,\Sf{A})\subset \oplus_{i=1}^l K(\tilde{C}_i,p_i)\stackrel{ord}{\longrightarrow} (\ZZ\cup\{+\infty\})^l.
\end{equation}
\end{definition}

It is easy to verify that the previous definition does not depend on the choice of the isomorphism $\Sf{A}\cong\Sf{A}\otimes\omega_{\tilde{X}}$.

\begin{definition}
\label{def:valuative}
With the notations introduced above, the pair $(\Sf{A},(\psi_1,...,\psi_r))$ satisfy the {\em Valuative Condition} if the inclusion
$$\mathfrak{A}\subset\mathfrak{K}_{(\Sf{A},(\psi_1,...,\psi_r))}$$
is satisfied.
\end{definition}

\begin{remark}
\label{rem:valuative}
We have the inclusion 
$$\mathfrak{C}+\mathfrak{K}_\pi\subset\mathfrak{K}_{(\Sf{A},(\psi_1,...,\psi_r))}.$$

If $X$ has Gorenstein singularities and $(d_1(\overline{C}),...,d_l(\overline{C}))$ is the Gorenstein vector at the resolution $\pi$, then we have the 
equalities
$$\mathfrak{K}_{(\Sf{A},(\psi_1,...,\psi_r))}=\mathfrak{C}+\mathfrak{K}_\pi=(d_1(\overline{C}),...,d_l(\overline{C}))+ \mathfrak{C}.$$

Therefore, in the Gorenstein case the valuative condition holds for $(\Sf{A},\psi_1,...,\psi_r))$ if and only if the inclusion
$$\mathfrak{A}\subset (d_1(\overline{C}),...,d_l(\overline{C}))+ \mathfrak{C}$$
is satisfied.
\end{remark}
\proof
The first inclusion follows easily from the definition of $\mathfrak{K}_{(\Sf{A},(\psi_1,...,\psi_r))}$. The equalities follow because in the Gorenstein case any element
of $H^0(U,\omega_{\tilde{X}}^r)$ is equal to a Gorenstein form multiplied by a global regular function. 
\endproof

\begin{remark}
\label{rem:conductorsuma}
Since both $\mathfrak{C}$ and $\mathfrak{K}_\pi$ have non empty conductor sets and $\mathfrak{K}_{(\Sf{A},(\psi_1,...,\psi_r))}$ is closed by taking 
minima, the conductor set of $\mathfrak{K}_{(\Sf{A},(\psi_1,...,\psi_r))}$
is non-empty and has an absolute minimum denoted by $cond(\mathfrak{K}_{(\Sf{A},(\psi_1,...,\psi_r))})$. Moreover we have the 
inequalities
\begin{equation}
\label{eq:ineqcond}
 cond(\mathfrak{K}_{(\Sf{A},(\psi_1,...,\psi_r))})\leq cond(\mathfrak{C})+cond(\mathfrak{K}_\pi)\leq cond(\mathfrak{C})+(d_1(\overline{C}),...,d_l(\overline{C})).
\end{equation}
In the Gorenstein case this inequality becomes the equality
\begin{equation}
\label{eq:eqcondgor}
cond(\mathfrak{K}_{(\Sf{A},(\psi_1,...,\psi_r))})=cond(\mathfrak{C})+(d_1(\overline{C}),...,d_l(\overline{C})).
\end{equation}

Likewise, the conductor set of $\mathfrak{A}$ has an absolute minimum denoted by $cond(\mathfrak{A})$. 
\end{remark}

\begin{remark}
\label{rem:valuativeparticular}
The following set of easy observations will have very useful consequences later.
\begin{enumerate}
 \item If $cond(\mathfrak{K}_{(\Sf{A},(\psi_1,...,\psi_r))})\leq (0,...,0)$, then the Valuative Condition is satisfied.
 \item If $\mathfrak{A}$ equals $\NN^l$, then the valuative condition is equivalent to the inequality
 $$cond(\mathfrak{K}_{(\Sf{A},(\psi_1,...,\psi_r))})\leq (0,...,0).$$
 \item If $X$ has Gorenstein singularities, then the condition $cond(\mathfrak{K}_{(\Sf{A},(\psi_1,...,\psi_r))})\leq (0,...,0)$ holds if and only if 
 $cond(\mathfrak{C})\leq -(d_1(\overline{C}),...,d_l(\overline{C}))$.
\end{enumerate}
\end{remark}

\begin{proposition}
\label{prop:contval}
The Containment Condition implies the Valuative Condition.
\end{proposition}
\proof
We have to translate the Containment Condition $\mathrm{Im}\gamma_1\subset \mathrm{Im}\delta$ into the inclusion of sets $$\mathfrak{A}\subset\mathfrak{K}_{(\Sf{A},(\psi_1,...,\psi_r))}.$$ We need the concrete description of the sheaf $\Sf{A}$ obtained in Proposition~\ref{prop:genredCM}: let $\overline{C}$ be the 
support of $\Sf{A}$, $n:\tilde{C}\to\overline{C}$ be its normalization. We have a chain of inclusions 
\begin{equation}
\label{eq:chainlocal18}
\Ss{\overline{C}}\subset\Sf{A}\subset n_*\Ss{\tilde{C}},
\end{equation}
which is not necessarily unique. We fix one of such chains of inclusions. 

Since $\Cs{\Rs}$ is an invertible sheaf and $\Sf{A}$ has support contained in a Stein open subset of $\tilde{X}$ we have the 
isomorphism $\Sf{A} \otimes \Cs{\Rs}\cong \Sf{A}$. Hence 
\begin{equation*}
H^0(\Rs,\Sf{A} \otimes \Cs{\Rs}) \cong H^0(\Rs,\Sf{A}), 
\end{equation*}
and 
\begin{equation*}
H^0(U, \Sf{A} \otimes \Cs{\Rs}) \cong H^0(U, \Sf{A})\subset K(\tilde{C}),
\end{equation*} 
where $K(\tilde{C})$ denotes the total fraction ring of the ring $H^0(\tilde{C},\Ss{\tilde{C}})$.

The image $\im \gamma_1$ is then identified with the inclusion 
$$H^0(\Rs,\Sf{A})\subset H^0(\tilde{C},\Ss{\tilde{C}})\subset K(\tilde{C}).$$
As a consequence, if we consider the order function
$$ord:K(\tilde{C})\to\ZZ^l,$$
we obtain the equality of sets
\begin{equation}
\label{eq:1}
ord(\im \gamma_1)=\mathfrak{A}.
\end{equation}

The morphism $\delta$ is induced by the sections $\{\psi_1, \dots, \psi_r\}$ of $\Sf{A}$. The definition of $\mathfrak{K}_{(\Sf{A},(\psi_1,...,\psi_r))}$
gives the equality
\begin{equation}
\label{eq:2}
ord(\im \delta)=\mathfrak{K}_{(\Sf{A},(\psi_1,...,\psi_r))}.
\end{equation}

Now the result follows, because of equalities~(\ref{eq:1}) and~(\ref{eq:2}).
\endproof

Before we extract some useful consequences of Remark~\ref{rem:valuativeparticular} and Theorem~\ref{th:corres} we need a further lemma:

\begin{lemma}\label{lema:formagorenstein}
Let $(X,x)$ be a complex analytic germ of a normal two-dimensional Gorenstein singularity and  $\pi \colon \Rs \to X$ be a resolution with  exceptional divisor 
$E=\bigcup_{i=1}^n E_i$. Then for any component $E_j$ where the Gorenstein form $\Omega$ has a pole and for any component $E_k$ where the Gorenstein form has a zero, 
we have that $E_j \cap E_k = \emptyset$.
\end{lemma}
\proof

If the singularity is rational the Gorenstein form does not have poles at any resolution. The result follows for that case.

If the singularity is non-rational the Gorenstein form have strict poles at any component of the exceptional divisor of the minimal resolution. If $p$ is a point 
at a resolution $\tilde{X}$, $E=\cup_{i=1}^r E_i$ is the decomposition in irreducible components of the exceptional divisor at $\tilde{X}$, 
and $E_p$ is the exceptional divisor at the blow up at $p$, then the order of $\Omega$ at $E_p$ equals
$$ord_{E_p}(\Omega):=1+\sum_{i=1}^r ord_{E_i}(\Omega) mult_{p}(E_i).$$

Using induction on the number of blows ups that are necessary in order to obtain the resolution $\pi \colon \Rs \to X$ from the minimal resolution the proposition is
proved easily in the non-rational case.
\endproof

The consequences announced above are:

\begin{proposition}
\label{prop:consecuenciaspracticas}
Let $\pi:\Rs\to X$ be a resolution.
Let $\Sf{A}$ be a  rank 1 generically reduced  $1$-dimensional Cohen-Macaulay $\Ss{\Rs}$-module. Let $\mathfrak{A}$ be its set of orders 
(see Definition~\ref{def:orderset}). Let $\overline{C}$ be the support of $\Sf{A}$ and $n:\tilde{C}\to \overline{C}$ be the normalization. Let 
$(d_1(\overline{C}),...,d_l(\overline{C}))$ be the canonical vector of $\overline{C}$ and $\mathfrak{K}_\pi$ be the canonical set of orders of
$\overline{C}$.
Let $(\psi_1,...,\psi_r)$ be $r$ global sections
spanning $\Sf{A}$ as $\Ss{\Rs}$-module. The $\Ss{X}$-module $\Sf{C}$ spanned by $(\psi_1,...,\psi_r)$ is  rank 1 generically reduced  
$1$-dimensional Cohen-Macaulay. Let $\mathfrak{C}$ be the set of orders of $\Sf{C}$. Let $\mathfrak{K}_{(\Sf{A},(\psi_1,...,\psi_r))}$ be the 
canonical set of orders of $(\Sf{A},(\psi_1,...,\psi_r))$. 

Let $(\Sf{M},(\phi_1,...,\phi_r))$ be the pair associated with $(\Sf{A},(\psi_1,...,\psi_r))$ in the proof of Theorem~\ref{th:corres}.
\begin{enumerate}
 \item If $cond(\mathfrak{K}_{(\Sf{A},(\psi_1,...,\psi_r))})\leq 0$, then $\Sf{M}$ is full.
 \item For the previous condition it is enough to have the inequality $cond(\mathfrak{C})\leq -(d_1(\overline{C}),...,d_l(\overline{C}))$. 
 \item Suppose that the curve $\overline{C}$ is smooth and meets the exceptional divisor transversely at smooth points. Then $\Sf{M}$ is full if and only if we
 have the inequality $cond(\mathfrak{K}_{(\Sf{A},(\psi_1,...,\psi_r))})\leq 0$. In the Gorenstein case the inequality becomes 
 $cond(\mathfrak{C})\leq -(d_1(\overline{C}),...,d_l(\overline{C}))$.
 \item If there is an  index $i$ such that the strict inequality $d_i(\overline{C})>0$ holds, then $\Sf{M}$ is not full. In the Gorenstein case, if
 $\overline{C}$ meets a component of the exceptional divisor where the Gorenstein form has a zero, then $\Sf{M}$ is not full. 
 \item If $X$ has Gorenstein singularities, $\pi$ is small with respect to the Gorenstein form and $\Sf{C}=\Sf{A}$ (that is, if the module $\Sf{M}$ is special), then $\Sf{M}$ is full.
 \item Suppose that $(d_1(\overline{C}),...,d_l(\overline{C}))\leq 0$ and that $\Sf{C}=\pi_*n_*\Ss{\tilde{C}}$. Then $\Sf{M}$ is a special full
 $\Ss{\tilde{X}}$-module. In the Gorenstein case the first inequality holds when $\overline{C}$ does not meet an exceptional divisor where the 
 Gorenstein form has a $0$.
 
\end{enumerate}
\end{proposition}
\proof
If $cond(\mathfrak{K}_{(\Sf{A},(\psi_1,...,\psi_r))})\leq 0$ then $n_*\Ss{\tilde{C}}$ is contained in $\mathrm{Im}\delta$. Since $\mathrm{Im} \gamma_1$ is contained in $n_*\Ss{\tilde{C}}$ the Containment Condition holds and Assertion (1) follows by Theorem~\ref{th:corres}.

Assertion (2) is a consequence of Assertion (1) and Inequality~(\ref{eq:ineqcond}). 

In order to prove Assertion (3) notice that if $\overline{C}$ is smooth and meets the exceptional divisor transversely at smooth points then $\Sf{A}$ is equal
to $\Ss{\overline{D}}=n_*\Ss{\tilde{D}}$ and then we have $\mathfrak{A}=\NN^l$. In this situation the Valuative Condition is equivalent to the inequality 
$cond(\mathfrak{K}_{(\Sf{A},(\psi_1,...,\psi_r))})\leq 0$.  In the Gorenstein case Equality~(\ref{eq:eqcondgor}) transforms the 
previous inequality into $cond(\mathfrak{C})\leq -(d_1(\overline{C}),...,d_l(\overline{C}))$. This proves Assertion~(3).

In order to prove Assertion~(4) notice that
the vector $(0,...,0)$ is always contained in $\mathfrak{A}$, since $\Ss{\overline{C}}$ is contained in $\Sf{A}$. On the other hand, 
if there is an  index $i$ such that the strict inequality $d_i(\overline{C})>0$ then $(0,...,0)$ is not contained in 
$\mathfrak{K}_{(\Sf{A},(\psi_1,...,\psi_r))}$, and the Valuative Condition does not hold. By Lemma~\ref{lema:formagorenstein}, in the Gorenstein
case there is an  index $i$ such that the strict inequality $d_i(\overline{C})>0$ holds if and only if $\overline{C}$ meets a component of the 
exceptional divisor where the Gorenstein form has a zero. This proves Assertion~(4).

For Assertion~(6) notice that if $\Sf{C}=\pi_*n_*\Ss{\tilde{C}}$ then automatically we have the equality $\Sf{C}=\Sf{A}$, and $\Sf{M}$ is special. If $(d_1(\overline{C}),...,d_l(\overline{C}))\leq 0$, it is easy to show using the equality $\Sf{C}=\pi_*n_*\Ss{\tilde{C}}$, that we have the inequality $cond(\mathfrak{K}_{(\Sf{A},(\psi_1,...,\psi_r))})\leq 0$. So by Assertion (1), if $(d_1(\overline{C}),...,d_l(\overline{C}))\leq 0$ then $\Sf{M}$ is full. This proves the assertion, except the addendum about Gorenstein singularities, which follows from Lemma~\ref{lema:formagorenstein}.

Assertion~(5) is a bit harder. We will proof the Containment Condition directly by a cohomological argument: we have to prove that image of $\gamma_1$ is contained in the image of $\delta$ in Diagram~(\ref{subdiagram}). 

Since we have a small resolution with respect to the Gorenstein form (see Definition~\ref{def:smallresgor}), we can consider the exact sequence
\begin{equation}\label{eq:exactseqCs}
0 \to \Cs{\Rs} \to \Ss{\Rs} \to \Ss{Z_K} \to 0.
\end{equation}

Now apply the functor $-\otimes-$ to the sequences \eqref{exact:da2} and \eqref{eq:exactseqCs}, 

\begin{equation*}
\begin{tikzpicture}
  \matrix (m)[matrix of math nodes,
    nodes in empty cells,text height=1.5ex, text depth=0.25ex,
    column sep=2.5em,row sep=2em]{
		  & 0                       & 0          &                        &   \\
    0 & \Sf{N} \otimes \Cs{\Rs} & \Cs{\Rs}^r & \Sf{A}\otimes \Cs{\Rs} & 0 \\
    0 & \Sf{N}                 & \Ss{\Rs}^r & \Sf{A}                  & 0 \\
      & \Sf{N}\otimes \Ss{Z_k} & \Ss{Z_k}^r & \Sf{A} \otimes \Ss{Z_k} & 0 \\
			&  0                     &  0         &   0                     &   \\};
\foreach \y [remember=\y as \lasty (initially 2)] in {3, 4}
{
\foreach \x [remember=\x as \lastx (initially 2)] in {3,4}
{
\draw[-stealth] (m-\y-\lastx) -- (m-\y-\x);
\draw[-stealth] (m-\lasty-\x) -- (m-\y-\x);
}
}
\draw[-stealth] (m-1-2) -- (m-2-2);
\draw[-stealth] (m-1-3) -- (m-2-3);
\draw[-stealth] (m-2-2) -- (m-2-3);
\draw[-stealth] (m-2-3) -- (m-2-4);
\draw[-stealth] (m-2-2) -- (m-3-2);
\draw[-stealth] (m-2-1) -- (m-2-2);
\draw[-stealth] (m-3-1) -- (m-3-2);
\draw[-stealth] (m-2-4) -- (m-2-5);
\draw[-stealth] (m-3-4) -- (m-3-5);
\draw[-stealth] (m-3-2) -- (m-4-2);
\draw[-stealth] (m-4-4) -- (m-4-5);
\draw[-stealth] (m-4-2) -- (m-5-2);
\draw[-stealth] (m-4-3) -- (m-5-3);
\draw[-stealth] (m-4-4) -- (m-5-4);
\end{tikzpicture}
\end{equation*}

By the last diagram we get the following commutative diagram
\begin{equation*}
\begin{tikzpicture}
  \matrix (m) [matrix of math nodes,
    nodes in empty cells,text height=1.5ex, text depth=0.25ex,
    column sep=2.5em,row sep=2em]{
    & H^0\left(\Rs,\Cs{\Rs}^r\right)& & \vphantom{H^0}H^0\left(\Rs,\Sf{A} \otimes \Cs{\Rs}\right) & \\
    H^0\left(\Rs,\Ss{\Rs}^r\right) & & \vphantom{H^0}H^0\left(\Rs,\Sf{A}\right) & \\
    & H^0\left(U,\Cs{\Rs}^r\right) & & H^0\left(U,\Sf{A} \otimes \Cs{\Rs}\right) \\
    H^0\left(U,\Ss{\Rs}^r\right) & & H^0\left(U,\Sf{A}\right) & \\};
  \path[-stealth]
    (m-1-2) edge  (m-1-4) edge (m-2-1)
            edge [densely dotted] (m-3-2)
    (m-1-4) edge node[auto]{$\gamma_1$} (m-3-4) edge node[auto]{$\beta$} (m-2-3)
    (m-2-1) edge [-,line width=6pt,draw=white] (m-2-3)
            edge node[left=17pt, below=1pt]{$\rho$} (m-2-3) edge node[auto]{$\nu$}(m-4-1)
    (m-3-2) edge [densely dotted] node[left=30pt, below=1pt]{$\delta$} (m-3-4)
            edge [densely dotted] node[auto]{$\theta$} (m-4-1)
    (m-4-1) edge node[auto]{$\delta '$} (m-4-3)
    (m-3-4) edge node[auto]{$\alpha$} (m-4-3)
    (m-2-3) edge [-,line width=6pt,draw=white] (m-4-3)
            edge node[left=5pt,below=-25pt]{$\gamma_1'$} (m-4-3);
\end{tikzpicture}
\end{equation*}

We need to prove that 
\begin{equation}\label{conditionauxiliar}
\text{im}(\gamma_1) \subset \text{im}(\delta).
\end{equation}

Notice that the maps $\alpha$ and $\theta$ are isomorphisms because the support of $\Ss{Z_K}$ does not intersect $U$. Since $\alpha$ is injective, the condition \eqref{conditionauxiliar} is equivalent to 
\begin{equation*}
\im (\alpha \gamma_1)  \subset \im (\alpha \delta).
\end{equation*}

Since the diagram is commutative and $\theta$ is onto we get
\begin{equation*}
\text{im}(\alpha \delta) = \text{im}(\delta ' \theta) = \text{im}(\delta ').
\end{equation*}

Hence it is enough to prove that the image of $(\alpha \gamma_1)$ is contained in the image of $\delta '$. Using again that the diagram is commutative and $\rho$ is onto because $\Sf{M}$ is special, we get
\begin{equation*}
\text{im}(\alpha \gamma_1) = \text{im}(\gamma_1 ' \beta) \subset \text{im}(\gamma_1') = \text{im}(\gamma_1 ' \rho) = \text{im}(\delta ' \nu) \subset \text{im}(\delta '),
\end{equation*}
as we wish.
\endproof

\subsection{A comparison of correspondences}
\label{sec:comparacion}
In order to compare the two correspondences we need to impose that $X$ is a normal Stein surface with Gorenstein singularities. The following propositions compare the correspondence at the Stein surface (Theorem~\ref{th:corrsing}) and the correspondences at various 
resolutions (Theorem~\ref{th:corres}).

\begin{proposition}
\label{prop:dirressing}
Let $X$ be a normal Stein surface with Gorenstein singularities.
Let $M$ be a reflexive $\Ss{X}$-module of rank $r$. Let $\pi_1:\Rs_1\to X$ be a resolution and $\rho:\Rs_2\to\Rs_1$ be the blow up at a point $p$. 
Denote by $\pi_2:\Rs_2\to X$ the composition $\pi_2=\pi_1\comp\rho$. Denote by $\Sf{M}_1$ and $\Sf{M}_2$ the full sheaves associated with $M$ 
at each of the resolutions. Let $(\phi_1,...,\phi_r)$ be $r$ generic sections of $M$. Let $(\Sf{A}_1,(\psi^1_1,...,\psi^1_r))$ and  
$(\Sf{A}_2,(\psi^2_1,...,\psi^2_r))$  be 
the pairs associated with $(\Sf{M}_1,(\phi_1,...,\phi_r))$ and $(\Sf{M}_2,(\phi_1,...,\phi_r))$ under Theorem~\ref{th:corres}. 
Let $(\Sf{C},(\psi^0_1,...,\psi^0_r))$ be the pair associated with $(M,(\phi_1,...,\phi_r))$ under Theorem~\ref{th:corrsing}.
\begin{enumerate}
 \item There are inclusions $\Sf{C}\subset (\pi_1)_*\Sf{A}_1\subset (\pi_2)_*\Sf{A}_2$. Under this inclusion the sections $(\psi^i_1,...,\psi^i_r)$ are identified for $i=0,1,2$. The dimension of the quotient 
 $(\pi_i)_*\Sf{A}_i/\Sf{C}$ as a $\CC$-vector space equals $\dim_{\CC}(R^1(\pi_i)_*\Sf{M}^{\smvee}_i)-rp_g$ for $i=1,2$.
 \item  We have the inclusion $\Sf{A}_1\subset \rho_*\Sf{A}_2$ and the dimension of the quotient 
 $\rho_*\Sf{A}_2/\Sf{A}_1$ as a $\CC$-vector space equals $\dim_{\CC}(R^1\rho_*\Sf{M}^{\smvee}_2)$.
\end{enumerate}
\end{proposition}
\proof
For $i=1,2$ consider the exact sequence 
$$0\to\Ss{\Rs_i}^r\to \Sf{M}_i\to \calA'_i\to 0,$$
whose first morphism is induced by the sections $(\phi_1,...,\phi_r)$. Dualizing we obtain 
\begin{equation}
\label{eq:otravez1}
0\to\Sf{N}_i\to\Ss{\Rs_i}^r\to \calA_i\to 0,
\end{equation}
where the second morphism is induced by the sections $(\psi^i_1,...,\psi^i_r)$. 

The sections $(\phi_1,...,\phi_r)$ also induce the exact sequence 
$$0\to \Ss{X}^r\to M\to\Sf{C}'\to 0,$$
and dualizing it we obtain
\begin{equation}
\label{eq:otravez2}
0\to N\to\Ss{X}^r\to\Sf{C}\to 0,
\end{equation}
where the second morphism is induced by $(\psi^0_1,...,\psi^0_r)$. 

Applying $R(\pi_i)_*$ to (\ref{eq:otravez1}), we obtain the exact sequence
\begin{equation}
\label{exseq:918}
0\to (\pi_i)_*\Sf{N}_i\to\Ss{X}^r\to (\pi_i)_*\Sf{A}_i\to R^1(\pi_i)_*\Sf{N}_i\to R^1(\pi_i)_*\Ss{X}^r\to 0.
\end{equation}
By Lemma~\ref{lema:dualM} and its proof the first morphism of the previous sequence (for $i=1,2$) coincides with the first morphism of 
sequence~(\ref{eq:otravez2}). This implies that the image of the second morphism of the previous sequence (for $i=1,2$) coincides with $\Sf{C}$
and that, under this identification the systems of sections $(\psi^i_i,...,\psi^i_r)$ coincide for $i=0,1,2$. This proves the inclusions
$\Sf{C}\subset (\pi_i)_*\Sf{A}_i$ for $i=1,2$ and the identification of the sections. Since $\Sf{A}_i$ is generated by $(\psi^i_1,...,\psi^i_r)$ as a $\Ss{\Rs_i}$-module and  $\Ss{\Rs_2}$ contains $\Ss{\Rs_1}$, the inclusion $(\pi_1)_*\Sf{A}_1\subset (\pi_2)_*\Sf{A}_2$ also holds. The equality
$$\dimc{\pi_i)_*\Sf{A}_i/\Sf{C}}=\dim_{\CC}(R^1(\pi_i)_*\Sf{M}^{\smvee}_i)-rp_g$$
follows from Exact Sequence~(\ref{exseq:918}). This shows the first assertion.

The proof of the second assertion follows similarly, by applying $R\rho_*$ to Sequence~(\ref{eq:otravez1}), and taking into account
the identification of the sections $(\psi^i_i,...,\psi^i_r)$ for $i=1,2$.
\endproof

For later use we need to compare the sets $\mathfrak{K}_{(\Sf{A}_1,(\psi^1_1,...,\psi^1_r))}$ and $\mathfrak{K}_{(\Sf{A}_2,(\psi^2_1,...,\psi^2_r))}$ (see Section~\ref{sec:valuative} for the corresponding definition).

\begin{proposition}
\label{prop:compcanord}
Consider the same situation than in the previous proposition, but allow non-Gorenstein normal singularities.
There exist a non-negative integer vector $(d_1,...,d_l)$ such that we have the equality
\begin{equation}
\label{eq:transorder}
\mathfrak{K}_{(\Sf{A}_2,(\psi^2_1,...,\psi^2_r))}=(d_1,...,d_{l})+\mathfrak{K}_{(\Sf{A}_1,(\psi^1_1,...,\psi^1_r))}.
\end{equation}
As a consequence we have also the equality
\begin{equation}
\label{eq:transcond}
cond(\mathfrak{K}_{(\Sf{A}_2,(\psi^2_1,...,\psi^2_r))})=(d_1,...,d_{l})+cond(\mathfrak{K}_{(\Sf{A}_1,(\psi^1_1,...,\psi^1_r))}).
\end{equation}
The vector is strictly positive if and only if the blowing up center of $\rho$ meets the support of $\Sf{A}_1$. 
\end{proposition}
\proof
Denote by $E^i$ the exceptional divisor of $\pi_i$. Let $C$ be the support of $\Sf{C}$, let $\overline{C}^i$ be the support of $\Sf{A}_i$. We have a birational morphisms $\rho|_{\overline{C}^2}:\overline{C}^2\to\overline{C}^1$ and $\pi_1|_{\overline{C}^1}: \overline{C}^1\to C$. Consider the normalization $n:\tilde{C}\to \overline{C}^2$. Let $p_j$ for $j=1,...,l$ be the points of $\tilde{C}$ which map via $n$ to a point of the exceptional divisor. Let $(\tilde{C}_j,p_j)$ be the 
germ at $p_j$. Denote $p^2_j:=n(p_j)$ and $p^1_j:=\rho(p^2_j)$ for $j=1,...,l$.

The proposition is 
trivial if the center of the blowing up $\rho$ does not coincide with $p^1_j$ for any $j$. By notational convenience
we assume that $p^1_1$ is the blowing-up center. We denote by $E^2_1$ the exceptional divisor of $\rho$. Choose local coordinates
$(x^i_j,y^i_j)$ of $\Rs^i$ around each point $p^i_j$. The choice is made so that if two points $p^i_j$ coincide for different $j$, then the corresponding coordinates are also 
the same, and so that, if $\rho(p^2_j)=p^1_1$, then $\rho$ expresses in local coordinates as $\rho(x^2_j,y^2_j)=(x^2_j,x^2_jy^2_j)$.

Let $\beta_1,...,\beta_n$ be a system of generators of $H^0(U,\omega_{\Rs})$ as a $\Ss{X}$-module. The differential form $\beta_k$ expresses in each of
the local chart around $p^i_j$ as $\beta=h_{k,j}^idx^i_j\wedge dy^i_j$, where $h^i_{k,j}$ is a germ of meromorphic function at $p^i_j$. 
If $\rho(p^2_j)=p^1_1$ then $h^2_{k,j}=x^2_j\rho^*h^1_{k,1}$. If $\rho(p^2_j)=p^1_m$ for $m\neq 1$ then $h^2_{k,j}=\rho^*h^1_{k,m}$. 

In order to compare $\mathfrak{K}_{(\Sf{A}_1,(\psi^1_1,...,\psi^1_r))}$ and $\mathfrak{K}_{(\Sf{A}_2,(\psi^2_1,...,\psi^2_r))}$ we compare the images
of $H^0(U,\omega_{\Rs})$ in the ring $\bigoplus_{i=1}^lK(\tilde{C}_i,p_i)$ according with Definition~\ref{def:canoset2}; we denote each of the images by
$Im^i$ for $i=1,2$. Since, by 
Proposition~\ref{prop:dirressing} the sections $(\psi^1_1,...,\psi^1_r)$ and $(\psi^2_1,...,\psi^2_r)$ are identified, each of them define the 
same $l$-uple in $\bigoplus_{i=1}^lK(\tilde{C}_i,p_i)$, which we denote by $(\psi_v|_{\tilde{C}_1},...,\psi_v|_{\tilde{C}_1})$ for $v\in\{1,...,r\}$. 
Then $Im^i$ is the $\Ss{X}$-module spanned by 
$$\left \{(h^i_{1}|_{\tilde{C}_1}\psi_v|_{\tilde{C}_1},...,h^i_{l}|_{\tilde{C}_l}\psi_v|_{\tilde{C}_l}):v\in\{1,...,r\} \right \}.$$

Enumerate the points $p^2_1,...,p^2_l$ so that those whose image by $\rho$ equals $p^1_1$ are $p^2_1,...,p^2_{l_1}$ for $l_1\leq l$. We have the 
equality 
$$Im^2=(x^2_1|_{\overline{C}^2_1},...,x^2_{l_1}|_{\overline{C}^2_{l_1}},1,...,1)Im^1.$$
As a consequence, if for any $w\leq l_1$ we define the intersection multiplicity $d_w:=I_{p^2_w}(E^2_1,\overline{C}^2_w)$ we have the equality
$$\mathfrak{K}_{(\Sf{A}_2,(\psi^2_1,...,\psi^2_r))}=(d_1,...,d_{l_1},0,...,0)+\mathfrak{K}_{(\Sf{A}_1,(\psi^1_1,...,\psi^1_r))}.$$
\endproof

\begin{proposition}
\label{prop:invressing}
 Let $\pi:\Rs\to X$ be a resolution of a normal Stein surface with Gorenstein singularities, which is an isomorphism over the regular locus of $X$.
Let $(\Sf{A},(\psi_1,...,\psi_r))$ be a pair formed by a  rank 1 generically reduced  $1$-dimensional Cohen-Macaulay $\Ss{\Rs}$-module, 
whose support meets $E$ in finitely many points, and a set of $r$ global sections spanning $\Sf{A}$ as $\Ss{\Rs}$-module and satisfying the Containment Condition. Let $\Sf{C}$ be the $\Ss{X}$-module spanned by $\psi_1,...,\psi_r$. Then $\Sf{C}$ is  a  rank 1 generically reduced  $1$-dimensional Cohen-Macaulay $\Ss{X}$-module. Let $(\Sf{M},(\phi_1,...,\phi_r))$ and $(M,(\phi'_1,...,\phi'_r))$ be the results of applying the 
correspondences of Theorems~\ref{th:corres} and~\ref{th:corrsing} at $\Rs$ and at the $X$ to $(\Sf{A},(\psi_1,...,\psi_r))$ and $(\Sf{C},(\psi_1,...,\psi_r))$ respectively.
Then we have the equalities $\pi_*\Sf{M}=M$ and $(\phi_1,...,\phi_r)=(\phi'_1,...,\phi'_r)$.
\end{proposition}
\proof
According with the proof of Theorem~\ref{th:corres} and its proof the module $\Sf{N}$ in the sequence 
$0\to\Sf{N}\to\Ss{\Rs}^r\to\Sf{A}\to 0$
is the dual of $\Sf{M}$. 
Pushing down by $\pi_*$ we obtain
$$0\to\pi_*\Sf{N}\to\Ss{X}^r\to\pi_*\Sf{A}\to R^1\pi_*\Sf{N}\to R^1\pi_*\Ss{\Rs}^r\to 0,$$
and the image of the map $\Ss{X}^r\to\pi_*\Sf{A}$ is the $\Ss{X}$-module spanned by $\psi_1,...,\psi_r$, that is, the module $\Sf{C}$.
So we obtain the sequence 
$$0\to\pi_*\Sf{N}\to\Ss{X}^r\to\Sf{C}\to 0.$$
According with Theorem~\ref{th:corrsing} and its proof the module $\pi_*\Sf{N}$ is isomorphic to the dual of $M$. By Lemma~\ref{lema:dualM} the module $\pi_*\Sf{N}$ is isomorphic to the dual of $\pi_*\Sf{M}$. This concludes the proof of the 
equality $\pi_*\Sf{M}=M$. Under the equality, the coincidence of the sections is straightforward.
\endproof

\section{The minimal adapted resolution}

In this section we show that, given a Stein normal surface $X$ and a reflexive $\Ss{X}$-module, there is a minimal resolution for which the associated full sheaf is generated by global 
sections. This resolution will be crucial later. 

\begin{proposition}\label{prop:minadap}

Let $X$ be a Stein normal surface. 
If $M$ is a reflexive $\Ss{X}$-module, then there exists a unique minimal resolution $\rho \colon \Rs' \to X$ such that the associated full 
$\Ss{\Rs}$-module $\Sf{M}:= \left ( \rho^* M \right)^{\smvee \smvee}$ is generated by global sections.
\end{proposition}
\proof
Let $M$ be a reflexive $\Ss{X}$-module, $\pi \colon \Rs \to X$ be the minimal resolution with exceptional divisor $E$ and denote by $\Sf{M}= \left ( \pi^* M \right)^{\smvee \smvee}$. If $\Sf{M}$ is generated by global sections, then we are done.

If $\Sf{M}$ is not generated by global sections, then there exists a finite set of points $S=\{p_1, \dots, p_n\} \subset E$ where $\Sf{M}$ fails to be generated by global sections. 

Assume that the rank of $\Sf{M}$ is $r$. Take $r$ generic sections of $\Sf{M}$ and consider the exact sequence given by the sections 
\begin{equation}
0 \to \Ss{\Rs}^r \to \Sf{M} \to \Sf{A}' \to 0.
\end{equation}
By the degeneracy module definition (Definition~\ref{def:degeneracymodule}) we have the inclusion $S \subset \text{Supp}(\Sf{A}')$.

Let $\overline{C}$ be the support of $\Sf{A}'$ and $(d_1(\overline{C}),...,d_r(\overline{C}))$ be the associated canonical vector (see
Definition~\ref{def:tranlation}). By Proposition~\ref{prop:consecuenciaspracticas},~(4) we have the inequality 
\begin{equation}
\label{eq:recordatorio}
(d_1(\overline{C}),...,d_r(\overline{C})) \leq (0,...,0). 
\end{equation}

Denote by $\sigma_{S} \colon \Rs' \to X$ the blow up at the set of centers $S$. Therefore we have the following commutative diagram
\begin{equation*}
\begin{tikzpicture}
  \matrix (m)[matrix of math nodes,
    nodes in empty cells,text height=1.5ex, text depth=0.25ex,
    column sep=2.5em,row sep=2em]{
  \Rs' & \Rs \\
  & X \\};
\draw[-stealth] (m-1-1) edge node [above] {$\sigma_{S}$} (m-1-2);
\draw[-stealth] (m-1-1) edge node [right] {$\pi'$} (m-2-2);
\draw[-stealth] (m-1-2) edge node [right] {$\pi$} (m-2-2);
\end{tikzpicture}
\end{equation*}

Denote by $\Sf{M}'$ the full $\Ss{\tilde{X}'}$-module associated to $M$. 
If $\Sf{M}'$ is generated by global sections, then we are done, otherwise we repeat the procedure. 

In order to prove that this process eventually ends we use $(d_1(\overline{C}),...,d_r(\overline{C}))$ as a resolution invariant. 

We take the same generic global sections for $\Sf{M}'$ and $\Sf{M}$ (in both cases the set of global sections is $M$). The support $\overline{C}'$ 
of the degeneracy module of $\Sf{M}'$ for this sections is the strict transform of $\overline{C}$ by $\sigma_S$. The normalization 
$\tilde{C}$ of $\overline{C}$ and $\overline{C}'$ is the same. Let $\{(\tilde{C}_i,p_i)\}_{i=1}^l$ be the branches of $\tilde{C}$ considered at the 
beginning of Section~\ref{sec:valuative}. Let $\beta$ be meromorphic  differential form in $\tilde{X}$. By the behavior of the poles of the pullback of a meromorphic differential form by
a blow up at a point in a smooth surface, we obtain that the order of $\sigma_S^*\beta$ at the different branches $(\tilde{C}_i,p_i)$ for $i=1,...,l$ is
greater or equal that the order of $\beta$. Moreover the order is strictly greater if the 
blowing up center meets the component that we are dealing with. This implies the strict inequality
$$(d_1(\overline{C}),...,d_r(\overline{C}))< (d_1(\overline{C}'),...,d_r(\overline{C}')).$$

This together with Inequality~(\ref{eq:recordatorio}) shows that the process terminates after finitely many steps.
\endproof

\begin{definition}
\label{def:minadap}
Let $M$ be a reflexive $\Ss{X}$-module. The minimal resolution $\rho:\Rs\to X$ where the associated full $\Ss{\Rs}$-module is generated by global 
sections is called the {\em minimal adapted resolution} to $M$.  
\end{definition}

\begin{remark}
\label{rem:minadapgor}
Let $X$ be a Stein surface with Gorenstein singularities. Let $M$ be a reflexive $\Ss{X}$-module. Then the minimal adapted resolution to $M$ is small with respect 
to the canonical form.
\end{remark}
\proof
This is a consequence of Inequality~(\ref{eq:recordatorio}) and Proposition~\ref{prop:consecuenciaspracticas},~(4).
\endproof

At the minimal adapted resolution the degeneracy module of the full sheaf for a generic set of sections has special properties, and Lemma~1.2 
of~\cite{AV} holds as stated there. 

\begin{lemma}
\label{lem:ArtinVerdierminadap}
Let $X$ be a Stein normal surface. Let $M$ be a reflexive $\Ss{X}$-module of rank $r$. Let $\rho:\Rs\to X$ be a resolution which dominates the minimal adapted resolution to $M$. 
The degeneracy module 
$\Sf{A'}$ of a set of $r$ generic global sections is isomorphic to $\Ss{D}$, where $D\subset\Rs$ is a smooth curve meeting the exceptional divisor 
transversely at its smooth locus. Moreover, by changing the sections the meeting points of $D$ with the exceptional divisor also change.
\end{lemma}
\proof
It is a simplification of the proof of Proposition~\ref{prop:generalizationAV}.
\endproof

Our aim now is to characterize numerically the minimal adapted resolution. 

\begin{proposition}
\label{prop:minadapnumchar}
Let $X$ be a Stein normal surface. Let $M$ be a reflexive $\Ss{X}$-module of rank $r$. Let $(\phi_1,...,\phi_r)$ be $r$ generic sections. 
Let $\pi:\Rs\to X$ be a resolution, denote by $\Sf{M}$ the full $\Ss{\Rs}$-module associated with $M$. Let $(\Sf{A},(\psi_1,...,\psi_r))$ be the 
pair associated with $(\Sf{M},(\phi_1,...,\phi_r))$ under Theorem~\ref{th:corres}. Genericity of the sections imply that 
the associated canonical set of orders $\mathfrak{K}_{(\Sf{A},(\psi_1,...,\psi_r))}$ (see Definition~\ref{def:canoset2}), and its minimal conductor 
$cond(\mathfrak{K}_{(\Sf{A},(\psi_1,...,\psi_r))})$ are independent of the chosen sections. Then the following properties hold:
\begin{enumerate}
 \item At the minimal adapted resolution we have the equality $cond(\mathfrak{K}_{(\Sf{A},(\psi_1,...,\psi_r))})=(0,...,0)$.
 \item At any resolution we have the inequality $cond(\mathfrak{K}_{(\Sf{A},(\psi_1,...,\psi_r))})\leq (0,...,0)$.
 \item A resolution dominates the minimal adapted resolution if and only if we have the equality 
 $$cond(\mathfrak{K}_{(\Sf{A},(\psi_1,...,\psi_r))})=(0,...,0).$$
\end{enumerate}
\end{proposition}
\proof
Let $\pi:\Rs_1\to X$ be the minimal adapted resolution to $M$ and $\Sf{M}_1$ the full $\Ss{\Rs_1}$-module associated with $M$. Let $E_1$ be the exceptional divisor. Consider the decomposition in irreducible components
$E_1=\cup_{j=1}^{m}E_{1,j}$.

Let $(\Sf{A}_1,(\psi_1,...,\psi_r))$ be the pair associated with $(\Sf{M}_1,(\phi_1,...,\phi_r))$ by Theorem~\ref{th:corres}.
Let $\overline{C}^1$ be the support of $\Sf{A}_ 1$.
By Lemma~\ref{lem:ArtinVerdierminadap},  $\overline{C}^1$ is a smooth curve meeting the exceptional divisor $E_1$ transversely at 
smooth points and we have the isomorphism $\Sf{A}_1\cong\Ss{\overline{C}^1}$. 

Then, by Proposition~\ref{prop:consecuenciaspracticas} (3), at the minimal adapted resolution $\pi:\Rs_1\to X$ we have the inequality 
$cond(\mathfrak{K}_{(\Sf{A}_1,(\psi^1_1,...,\psi^1_r))})\leq (0,...,0)$. Suppose that the inequality is strict. We will derive a contradiction, which 
will prove Assertion (1). 

Up to reindexing we may assume that 
the first coordinate of $cond(\mathfrak{K}_{(\Sf{A}_1,(\psi_1,...,\psi_r))})$ is negative. Since $\overline{C}^1$ is smooth it coincides with its normalization. Let $p_1$ be the point of $\overline{C}^1\cap E_1$ so that the order at the branch $(\overline{C}^1,p_1)$ gives 
the first coordinate of the order function (see Section~\ref{sec:valuative} to recall the corresponding definitions). 
Let $\rho:\Rs_2\to\Rs_1$ be the blow up at $p_1$; define $\pi_2:=\pi_1\comp\rho$. Let $\overline{C}^2$ be the strict transform of $\overline{C}^1$ by $\rho$. Define
$\Sf{A}_2:=\Ss{\overline{C}^2}$. Since $\rho$ induces an isomorphism between $\Sf{A}_2=\Ss{\overline{C}^2}$ and $\Sf{A}_1=\Ss{\overline{C}^1}$ the 
sections $(\psi_1,...,\psi_r)$ of $\Sf{A}_1$ may also be regarded as sections of $\Sf{A}_2$. A computation like in the proof of 
Proposition~\ref{prop:compcanord} shows the equality 
$cond(\mathfrak{K}_{(\Sf{A}_2,(\psi_1,...,\psi_r))})=(1,0,...,0)+cond(\mathfrak{K}_{(\Sf{A}_1,(\psi_1,...,\psi_r))})$. Therefore we have the inequality
$cond(\mathfrak{K}_{(\Sf{A}_2,(\psi_1,...,\psi_r))})\leq (0,...,0)$, and by Proposition~\ref{prop:consecuenciaspracticas} (1), the correspondence of 
Theorem~\ref{th:corres} assigns to $(\Sf{A}_2,(\psi_1,...,\psi_r))$ a pair $(\Sf{M}_2,(\phi^2_1,...,\phi^2_r))$, where $\Sf{M}_2$ is a full $\Ss{\Rs_2}$-module and
$(\phi^2_1,...,\phi^2_r)$ is a system of nearly generic global sections. An application of Proposition~\ref{prop:invressing} shows the equalities
$(\pi_2)_*\Sf{M}_2=(\pi_1)_*\Sf{M}_1$ and $(\phi^2_1,...,\phi^2_r)=(\phi_1,...,\phi_r)$.

By Lemma~\ref{lem:ArtinVerdierminadap}, 
there is a slight perturbation $(\phi'_1,...,\phi'_r)$ of the sections $(\phi_1,...,\phi_r)$ such that if we denote by $\Sf{A}_1'$ the degeneracy 
module of $(\Sf{M}_1,(\phi'_1,...,\phi'_r))$, then its support $(\overline{C}^1)'$ satisfies
\begin{itemize}
 \item it does not meet the blowing up center $\overline{C}_1\cap E$;
 \item we have the equality of intersection numbers $(\overline{C}^1)'\cdot E_{1,j}=\overline{C}^1\cdot E_{1,j}$ for any irreducible component 
 $E_{1,j}$.
\end{itemize}
 
The support of the degeneracy module of $(\Sf{M}_2,(\phi_1,...,\phi_r))$ is equal to $\overline{C}^2$. On the other hand, since $(\overline{C}^1)'$ does
not meet the blowing up center $\overline{C}_1\cap E$, the support of the degeneracy
module of $(\Sf{M}_2,(\phi'_1,...,\phi'_r))$ is the strict transform of $(\overline{C}^1)'$ to $\Rs_2$. Let $F$ be the exceptional divisor of $\rho$. 
Observe that $\overline{C}^2\cdot F=1$, and by property (2) above $(\overline{C}^1)'\cdot F=0$. Since the Poincare dual of the support of the 
degeneracy locus is the first Chern class of the module $\Sf{M}_2$, we have two different Chern class representations intersecting differently the 
cycle $F$. This is a contradiction which proves Assertion (1).

Assertions (2) and (3) are simple Corollaries of Assertion (1) and Proposition~\ref{prop:compcanord}.
\endproof

It is convenient to specialize the previous Proposition to the Gorenstein case:

\begin{corollary}
\label{cor:minadapnumchargor}
With the notation of the previous Proposition, and of Proposition~\ref{prop:consecuenciaspracticas}, if $X$ has Gorenstein singularities we have:

\begin{enumerate}
 \item At the minimal adapted resolution we have the equality $cond(\mathfrak{C})=(-d_1(\overline{C}),...,-d_l(\overline{C}))$.
 \item At any resolution we have the inequality $cond(\mathfrak{C})\leq (-d_1(\overline{C}),...,-d_l(\overline{C}))$.
 \item A resolution dominates the minimal adapted resolution if and only if we have the equality 
 $$cond(\mathfrak{C})=(-d_1(\overline{C}),...,-d_l(\overline{C})).$$
\end{enumerate}
\end{corollary}

\subsection{The behaviour of the speciality defect and minimal adapted resolutions of special modules}
\label{sec:minimaladapted}
We study the behaviour of the specialty defect under blow up, and show that if the specialty defect vanishes at a given resolution, it vanishes at any
resolution. As a corollary we establish the existence of special reflexive modules and show an interesting property of their generic degeneracy
modules, which links them with arcs in the singularity.

The following two propositions control the behavior of the specialty defect under blow up.

\begin{proposition}\label{prop:specialtybehaviour}
Let $X$ be a Stein normal surface. Let $M$ be a reflexive $\Ss{X}$-module. Let $\pi \colon \Rs \to X$ be a resolution and $p$ be a point in $\Rs$. Denote by 
$\sigma \colon \Rs' \to \Rs$ the blow up of the point $p$. We have the following diagram
\begin{equation*}
\begin{tikzpicture}
  \matrix (m)[matrix of math nodes,
    nodes in empty cells,text height=1.5ex, text depth=0.25ex,
    column sep=2.5em,row sep=2em]{
  \Rs' & \Rs \\
  & X \\};
\draw[-stealth] (m-1-1) edge node [above] {$\sigma$} (m-1-2);
\draw[-stealth] (m-1-1) edge node [right] {$\rho$} (m-2-2);
\draw[-stealth] (m-1-2) edge node [right] {$\pi$} (m-2-2);
\end{tikzpicture}
\end{equation*}
where $\rho:= \pi \circ{} \sigma$.

Denote by $\Sf{M}= \left( \pi^* M \right)^{\smvee \smvee}$ and $\Sf{M}'= \left( \rho^* M \right)^{\smvee \smvee}$. 
Then the specialty defect of $\Sf{M}$ is less or equal to the specialty defect of $\Sf{M}'$. Moreover if $\Sf{M}$ is generated by global sections 
at $p$, then the specialty defect of $\Sf{M}$ equals the specialty defect of $\Sf{M}'$.
\end{proposition}
\proof
Denote by
\begin{align*}
\Sf{N}' &= \Sf{M}'^{\smvee},\\
\Sf{N} &= \Sf{M}^{\smvee}.
\end{align*}

Since $\rho= \pi \circ{} \sigma$, in order to compute $R^1 \rho_*\Sf{N}'$ we use the Leray spectral sequence. In this case the page $E_2^{(p,q)}$ of the spectral sequence is given by

\begin{equation*}
\begin{tikzpicture}
  \matrix (m)[matrix of math nodes,
  nodes in empty cells,
  nodes={minimum width=5ex,minimum height=3ex,outer sep=-5pt},
  column sep=1ex,
  row sep=1ex]{
		  &   &   &   & \\ 
	2   & 0 & 0 & 0 & \\
  1   & \pi_*\left(R^1 \sigma_* \Sf{N}'\right ) & R^1 \pi_*\left(R^1 \sigma_* \Sf{N}'\right ) & 0 & \\
	0   & \pi_*\left(\sigma_* \Sf{N}' \right ) & R^1 \pi_*\left(\sigma_* \Sf{N}' \right ) & 0  & \\
\quad\strut	& 0 & 1 & 2 & \quad\strut \\};
\draw[thick] (m-1-1.east) -- (m-5-1.east) ;
\draw[thick] (m-5-1.north) -- (m-5-5.north) ;
\end{tikzpicture}
\end{equation*}

The spectral sequence degenerates, therefore we obtain the following exact sequence
\begin{equation}\label{exact:cambioN}
0 \to  R^1 \pi_*\left(\sigma_* \Sf{N}' \right ) \to  R^1 \rho_* \Sf{N}' \to \pi_*\left(R^1 \sigma_* \Sf{N}'\right ) \to 0.
\end{equation}

Now by adjuction we have the following identification
\begin{align} \begin{split}\label{eq:cambioN}
R^1 \pi_*\left(\sigma_* \Sf{N}' \right ) &= R^1 \pi_* \left(\sigma_* \Homs_{\Ss{\Rs'}}\left(\sigma^* \pi^* M, \Ss{\Rs'}\right )\right )\\
&=R^1 \pi_* \Homs_{\Ss{\Rs}}\left(\pi^* M, \sigma_* \Ss{\Rs'}\right )\\
&=R^1 \pi_* \Homs_{\Ss{\Rs}}\left(\pi^* M, \Ss{\Rs}\right )\\
&=R^1 \pi_* \Sf{N}. \end{split}
\end{align}

By \eqref{exact:cambioN} and \eqref{eq:cambioN} we get
\begin{equation}\label{exact:cambioN2}
0 \to  R^1 \pi_* \Sf{N} \to  R^1 \rho_* \Sf{N}' \to \pi_*\left(R^1 \sigma_* \Sf{N}'\right ) \to 0.
\end{equation}

Therefore the specialty defect of $\Sf{M}$ is less or equal to the specialty defect of $\Sf{M}'$.

Assume that $\Sf{M}$ is generated by global sections at $p$. By the previous exact sequence we only need to prove that $\pi_{*}\left(R^1 \sigma_* \Sf{N}'\right )=0$.

Now consider the exact sequence given by the natural map from $\pi^* M$ to its double dual 
\begin{equation*}
\begin{tikzpicture}
  \matrix (m)[matrix of math nodes,
    nodes in empty cells,text height=1.5ex, text depth=0.25ex,
    column sep=2em,row sep=0.5em] {
  0 &   T & \pi^* M & \Sf{M}  & \Sf{S} & 0, \\};

\foreach \x [remember=\x as \lastx (initially 1)] in {2,...,6}
{
\draw[-stealth] (m-1-\lastx) -- (m-1-\x);
}
\end{tikzpicture}
\end{equation*}
where $T$ is the kernel and $\Sf{S}$ is the cokernel. Notice that the support of $\Sf{S}$ is the set $S$ (the points where $\Sf{M}$ fails to be generated by global sections, see Section~\ref{chapter:artin-verdier-esnault}).

The last exact sequence can be split as follows
\begin{equation*}
\begin{tikzpicture}
  \matrix (m)[matrix of math nodes,
    nodes in empty cells,text height=1.5ex, text depth=0.25ex,
    column sep=2em,row sep=0.5em] {
  0 &   T & \pi^* M & \pi^* M / T & 0, \\
	0 & \pi^* M / T & \Sf{M} & \Sf{S} & 0. \\};
\foreach \y in {1,2}
{
\foreach \x [remember=\x as \lastx (initially 1)] in {2,...,5}
{
\draw[-stealth] (m-\y-\lastx) -- (m-\y-\x);
}
}
\end{tikzpicture}
\end{equation*}

Applying the functor $\sigma^* -$ to the last two exact sequences we obtain
\begin{equation*}
\begin{tikzpicture}
  \matrix (m)[matrix of math nodes,
    nodes in empty cells,text height=1.5ex, text depth=0.25ex,
    column sep=2em,row sep=0.5em] {
    0 & K_1 & \sigma^*T & \rho^* M  & \sigma^* \pi^* M / T & 0, \\
	  0 & K_2 & \sigma^* \pi^* M / T & \sigma^*\Sf{M} & \sigma^*\Sf{S} & 0, \\};
\foreach \y in {1,2}
{
\foreach \x [remember=\x as \lastx (initially 1)] in {2,...,6}
{
\draw[-stealth] (m-\y-\lastx) -- (m-\y-\x);
}
}
\end{tikzpicture}
\end{equation*}
where $K_1$ and $K_2$ are the modules that make the last sequences exact (recall that $\sigma^* -$ is just a right exact functor).

Hence we split the previous exact sequences as follows
\begin{equation*}
\begin{tikzpicture}
  \matrix (m)[matrix of math nodes,
    nodes in empty cells,text height=1.5ex, text depth=0.25ex, column sep=2em,row sep=0.5em]{
 0 & K_1 & \sigma^*T & H_1 & 0, \\
0 & H_1 & \rho^* M & \sigma^*\pi^* M / T & 0, \\
0 & K_2 & \sigma^* \pi^* M / T & H_2 & 0, \\
0 & H_2 & \sigma^*\Sf{M} & \sigma^*\Sf{S} & 0.\\};
\foreach \y in {1,2,3,4}
{
\foreach \x [remember=\x as \lastx (initially 1)] in {2,...,5}
{
\draw[-stealth] (m-\y-\lastx) -- (m-\y-\x);
}
}
\end{tikzpicture}
\end{equation*}

Dualizing the first, second and third exact sequences we get
\begin{align*}
H_1^{\smvee} &\cong 0, \quad \text{because $\sigma^* T$ is supported in the exceptional divisor},\\
\left(\sigma^*\pi^* M / T\right )^{\smvee} &\cong (\rho^* M)^{\smvee}, \quad \text{by the previous identification}, \\
H_2^{\smvee} &\cong  \left(\sigma^*\pi^* M / T\right )^{\smvee}, \quad \text{because $K_2$ is supported in the exceptional divisor}.
\end{align*}

Hence, as we have $\Sf{N}' = \Sf{M}'^{\smvee} \cong (\rho^* M)^{\smvee}$ we get $\Sf{N}' \cong \left(\sigma^*\pi^* M / T\right )^{\smvee}\cong H_2^{\smvee}$.

Finally dualizing the fourth exact sequence and using the previous identifications we get the exact sequence
$$ 0 \to \left(\sigma^*\Sf{M}\right)^{\smvee} \to \Sf{N}' \to \Exts_{\Ss{\Rs'}}^1\left(\sigma^* \Sf{S}, \Ss{\Rs'}\right ).$$

Since the point $p$ does not belong to the support of $\Sf{S}$, we get that the support of $\sigma^* \Sf{S}$ is zero dimensional, therefore $\Exts_{\Ss{\Rs'}}^1\left(\sigma^* \Sf{S}, \Ss{\Rs'}\right )$ is equal to zero. Hence we get
\begin{equation*}
R^1 \sigma_* \left(\sigma^*\Sf{M}\right)^{\smvee} \cong R^1 \sigma_* \Sf{N}'.
\end{equation*}

Since $\Sf{M}$ is locally free and we obtain $\Rs'$ taking the blow up in the point $p$ we get
\begin{equation*}
R^1 \sigma_* \left ( \left (\sigma^*\Sf{M}\right)^{\smvee} \right)= R^1 \sigma_* \left ( \left(\sigma^*\Ss{\Rs}^r\right )^{\smvee} \right)= R^1 \sigma_* \Ss{\Rs'}^r = 0.
\end{equation*}

Hence $R^1 \sigma_* \Sf{N}'$ is equal to zero. 
\endproof

\begin{corollary}\label{cor:especialminimal}
Let $X$ be a Stein normal surface. Let $M$ be a reflexive $\Ss{X}$-module. If the full sheaf associated to $M$ at the minimal adapted resolution is special, the the full sheaf associated to 
$M$ at the minimal resolution of $X$ is also special. 
\end{corollary}

\begin{theorem}\label{th:characspecial}
Let $X$ be a Stein normal surface. Let $M$ be a reflexive $\Ss{X}$-module. The module $M$ is special if and only if the full sheaf associated with $M$ at its minimal adapted 
resolution is special. 
\end{theorem}
\proof
Denote by $\pi:\Rs\to X$ the minimal resolution adapted to $M$ and by 
$\pi_{\text{min}} \colon \Rs_{\text{min}} \to X$ be the minimal resolution of $X$.

We need to prove that for any resolution $\rho \colon \hat{X} \to X$, the full sheaf $\hat{\Sf{M}}= \left(\rho^* M\right)^{\smvee \smvee}$ is special
(Definition~\ref{def:espmodule}). If the minimal resolution coincides with the minimal 
resolution adapted to $M$, then by Proposition~\ref{prop:specialtybehaviour} we are done.
Suppose that the minimal resolution and the minimal resolution adapted to $M$ do not coincide.
By taking a finite succession of blowing ups in different points we obtain a resolution $\breve{\rho} \colon \breve{X} \to X$ such that it satisfies 
the following diagram
\begin{equation*}
\begin{tikzpicture}
  \matrix (m)[matrix of math nodes,
    nodes in empty cells,text height=1.5ex, text depth=0.25ex,
    column sep=2.5em,row sep=2em]{
  \breve{X} & \Rs \\
  \hat{X} & \Rs_{min} \\};
\draw[-stealth] (m-1-1) edge node[auto]{$o$} (m-1-2);
\draw[-stealth] (m-1-1) edge node[auto]{$\nu$} (m-2-1);
\draw[-stealth] (m-1-2) edge node[auto]{$\pi$} (m-2-2);
\draw[-stealth] (m-2-1) edge node[auto]{$\rho$} (m-2-2);
\end{tikzpicture}
\end{equation*}
where $\nu$ and  $o$ are a composition of blowings up in points and $\breve{\rho}= \rho \circ{}\nu$.

By Proposition~\ref{prop:specialtybehaviour} the full sheaf associated to $M$ in the resolution $\breve{X}$ is special. 
Again by Proposition~\ref{prop:specialtybehaviour} the specialty defect of the full sheaf associated to $M$ in the resolution $\hat{X}$ is less or 
equal to zero, and hence equal to $0$. Therefore $M$ is special.
\endproof

\begin{definition}
\label{def:mincanord}
Let $(X,x)$ be a normal surface singularity, $\pi:\Rs\to X$ a resolution, and $F$ an irreducible component of the exceptional divisor. 
The {\em minimal canonical order at} $F$ is defined to be
\begin{equation}
ord_F(K_X):=\min\left(\{ord_F(div(\beta):\beta\in H^0(U,\omega_{\Rs})\}\right).
\end{equation}
\end{definition}

\begin{remark}
\label{rem:obsmco}
The following easy observations hold:
\begin{enumerate}
 \item The minimal canonical order at $F$ does not depend on the resolution where $F$ appears.
 \item In the Gorenstein case the minimal canonical order at $F$ is the order at $F$ of the Gorenstein form.
 \item In the Gorenstein case the minimal canonical order of a divisor appearing at the minimal resolution is non-positive.
 \item If $F$ is a divisor and $F'$ is another divisor obtained by blowing up $F$ at a generic point then we have $ord_{F'}(K_X)=ord_F(K_X)+1$.
\end{enumerate}
\end{remark}

\begin{corollary}\label{cor:existenmodespeciales}
Let $(X,x)$ be a normal Gorenstein surface singularity. Then there exist non-free special reflexives modules.
\end{corollary}
\proof
A consequence of the third and fourth assertions of Remark~\ref{rem:obsmco} is the existence of a resolution $\pi:\Rs \to (X,x)$ 
such that a component $F$ of the exceptional divisor has minimal canonical order equal to $0$. Let $D$ be a irreducible smooth curve transverse to the exceptional divisor $F$ at a generic point. We choose $\Sf{A}=\Ss{D}$ and consider 
$(\psi_1,...,\psi_r)$ generators of $\Sf{A}$ as a $\Ss{X}$-module. In that case $\mathfrak{K}_{(\Sf{A},(\psi_1,...,\psi_r))}$ equals $\NN$, and hence 
we have the equality $cond(\mathfrak{K}_{(\Sf{A},(\psi_1,...,\psi_r))})=0$. By Proposition~\ref{prop:consecuenciaspracticas} (1), if 
$(\Sf{M},(\phi_1,...,\phi_r))$ is the pair associated with $(\Sf{A},(\psi_1,...,\psi_r))$ by Theorem~\ref{th:corres}, the sheaf $\Sf{M}$ is full. 

The sheaf $\Sf{M}$ is special. Indeed, since the sections $(\psi_1,...,\psi_r)$ generate $\Sf{A}$ as a $\Ss{X}$-module,
the third map in the exact sequence 
\begin{equation}
\label{seq:largaotravez}
0\to\pi_*\Sf{N}\to\Ss{X}^r\stackrel{(\psi_1,...,\psi_r)}{\longrightarrow}\pi_*\Sf{A}\to R^1\pi_*\Sf{N}\to R^1\pi_*\Ss{\Rs}^r\to 0
\end{equation}
is surjective. 

By Proposition~\ref{prop:minadapnumchar}
the resolution $\pi:\Rs \to (X,x)$ is the minimal adapted resolution to the module $\pi_*\Sf{M}$. Finally, by Theorem~\ref{th:characspecial} the 
reflexive $\Ss{X}$-module $\pi_*\Sf{M}$ is special. 

Non-freeness holds because by construction $\Sf{M}$ has non-trivial first Chern class.
\endproof

The next proposition explains the structure of degeneracy modules of special reflexive modules for sets of generic sections. .  

\begin{proposition}\label{prop:Aeslanormalizacion}
Let $X$ be a Stein normal surface with Gorenstein singularities. Let $M$ be a special reflexive $\Ss{X}$-module of rank $r$. Let $(\phi_1,...,\phi_r)$ be $r$ generic sections. 
Let $\Sf{C}$ be the degeneracy module of $(M,(\phi_1,...,\phi_r))$, and let $C$ be its support. Let $n:\tilde{C}\to C$ be the normalization. 
Then we have the isomorphism $\Sf{C}\cong n_*\Ss{\tilde{C}}$. 

Conversely, let $C\subset X$ be a reduced curve and $n:\tilde{C}\to C$ be its normalization. Let $(\psi_1,...,\psi_r)$ be a set of generators of 
$n_*\Ss{\tilde{C}}$ as a $\Ss{X}$-module. Let $(M,(\phi_1,...,\phi_r))$ be the pair associated with $(n_*\Ss{\tilde{C}},(\psi_1,...,\psi_r))$ under
the correspondence of Theorem~\ref{th:corrsing}. Then $M$ is a special reflexive module.
\end{proposition}
\proof
Let $\pi:\Rs\to X$ be the minimal adapted resolution to $M$ and $\Sf{M}$ be the associated full sheaf. Let $(\Sf{A},(\psi_1,...,\psi_r))$ be the pair 
associated with $(\Sf{M},(\phi_1,...,\phi_r))$ by Theorem~\ref{th:corres}. By Lemma~\ref{lem:ArtinVerdierminadap} we have the isomorphism
$\Sf{A}\cong\Ss{D}$ for a smooth curve meeting the exceptional divisor transversely at smooth points. 
By specialty and the exact sequence~(\ref{seq:largaotravez}) we have the equality $\Sf{C}=\pi_*\Sf{A}$.

For the converse let $(M,(\phi_1,...,\phi_r))$ be the pair associated with $(n_*\Ss{\tilde{C}},(\psi_1,...,\psi_r))$ under the correspondence of Theorem~\ref{th:corrsing}. Consider the minimal resolution $\pi:\Rs\to X$ adapted to $M$. Let $\overline{C}$ be the strict
transform of $C$ and let $\Sf{A}:=n_* \Ss{\tilde{C}}$. By Proposition~\ref{prop:consecuenciaspracticas} $(6)$, if $(\Sf{M},(\phi_1,...,\phi_r))$ is the pair 
associated with $(\Sf{A},(\psi_1,...,\psi_r))$ under Theorem~\ref{th:corres}, then the $\Ss{\Rs}$-module $\Sf{M}$ is full and special. 
By Proposition~\ref{prop:invressing} we have that $\pi_*\Sf{M}$ is isomorphic to $M$. Since $\Sf{M}$ is special, the module $M$ is special 
by Theorem~\ref{th:characspecial}.
\endproof

Let us record an interesting property of minimal adapted resolutions of special reflexive modules:

\begin{proposition}
\label{prop:minadapspproperty}
Let $X$ be a Stein normal surface with Gorenstein singularities.
Let $M$ be a special reflexive $\Ss{X}$-module of rank $r$. Let $(\phi_1,...,\phi_r)$ be $r$ generic sections. Let  
$\pi:\Rs\to X$ be the minimal adapted resolution to $M$ and $\Sf{M}$ be the associated full sheaf. Let $(\Sf{A},(\psi_1,...,\psi_r))$ be the pair 
associated with $(\Sf{M},(\phi_1,...,\phi_r))$ by Theorem~\ref{th:corres}. The support of $\Sf{A}$ is a smooth curve meeting the exceptional divisor 
transversely at smooth points which are located at divisors where the minimal canonical order vanishes.
\end{proposition}
\proof
We only need to prove that $D$ only meets at components where the minimal canonical order vanishes. Since $\pi:\Rs\to X$ is the minimal adapted 
resolution we have the equality $cond(\mathfrak{K}_{(\Sf{A},(\psi_1,...,\psi_r))})=(0,...,0)$ by Proposition~\ref{prop:minadapnumchar}. This, 
together with the equality $\Sf{C}=n_*\Ss{D}$ (which follows from the previous proposition), implies that $D$ only meets at components where the minimal 
canonical order vanishes.
\endproof

\subsection{The decomposition in indecomposables and the irreducible components of the degeneracy locus}
\label{sec:indecomposables}

Here we determine the relation between the decomposition of a special reflexive module into indecomposables, and the decomposition of its generic degeneracy locus into irreducible components. Here we work in the local case instead of the more general Stein surface case: we consider $(X,x)$ to be a normal Gorenstein surface singularity. 

\begin{proposition}
\label{prop:decompesp}
Let $(X,x)$ be a normal Gorenstein surface singularity. Let $M$ be a special $\Ss{X}$ module of rank $r$ and $(\phi_1,...,\phi_r)$ be a set of $r$ generic sections. Let $\pi:\Rs\to X$ be any resolution.
Let $\Sf{M}$ be the full $\Ss{\Rs}$-module associated to $M$. 
Let $(\Sf{C},(\psi^0_1,...,\psi^0_r))$ and $(\Sf{A},(\psi^1_1,...,\psi^1_r))$ be the results of applying Theorems~\ref{th:corrsing} and ~\ref{th:corres}
to $(M,(\phi_1,...,\phi_r))$ and $(\Sf{M},(\phi_1,...,\phi_r))$ respectively. Suppose that $M$ has no free direct factors. Then are natural 
bijections between the following sets
\begin{enumerate}
 \item The indecomposable direct summands of $M$.
 \item The indecomposable direct summands of $\Sf{M}$.
 \item The irreducible components of the support of $\Sf{C}$.
 \item The irreducible components of the support of $\Sf{A}$.
 \end{enumerate}
\end{proposition}
\proof
The first and second sets are in a bijection via $\pi_*$. The third and fourth sets are in bijection via $\pi$. Now we show a bijection between 
the first and third set. For this is convenient to choose $\pi:\Rs\to X$ to be the minimal resolution adapted to $M$ (or at least dominating it).

Let $C$ be the support of $\Sf{C}$. By Lemma~\ref{lem:ArtinVerdierminadap} the support $\overline{C}$ of $\Sf{A}$ is the strict transform of $C$ by
$\pi$, and decomposes as a disjoint union $\overline{C}=\coprod_{j=1}^k\overline{C}_j$ of $k$ smooth curves meeting the exceptional divisor 
transversely at smooth points. By Proposition~\ref{prop:Aeslanormalizacion}, we have the isomorphism  $\Sf{C}\cong \pi_*\Ss{\tilde{C}}$. 

We have the isomorphism $\Ss{\overline{C}}=\oplus_{j=1}^k\Ss{\overline{C}_j}$. For each $j$ let $(\psi_{j,1},...,\psi_{j,r_j})$ be a minimal system 
of generators of $n_*\Ss{\tilde{C_j}}$ as a $\Ss{X}$-module. Applying Theorem~\ref{th:corrsing} to the pair 
$(n_*\Ss{\overline{C}},(\psi_{1,1},...,\psi_{k,r_k}))$ we obtain a reflexive $\Ss{X}$-module $M'$, which has no free factors by the minimality of the 
set of generators of $n_*\Ss{\overline{C}}$. If we denote by $M'_j$ the reflexive $\Ss{X}$-module obtained
by applying Theorem~\ref{th:corrsing} to the pair $(n_*\Ss{\overline{C_j}},(\psi_{j,1},...,\psi_{j,r_j}))$ then we have the direct sum decomposition
$M'=\oplus_{j=1}^kM'_j$.

By the proof of Theorem~\ref{th:corrsing} we have that each of $M$ and $M'$ are isomorphic to the dual of the module of relations of a minimal set 
of generators of $n_*\Ss{\overline{C}}$ as $\Ss{X}$-module. Hence $M$ and $M'$ are isomorphic. 

In order to finish the proof we have to show that each $M'_j$ is indecomposable. Let $M'_j=\oplus_{v=1}^{m_j} M'_{j,v}$ be the
decomposition in indecomposable reflexive modules. Since $M$ does not have free factors no one of the factors is free. For each $v=1,...,m_j$ choose
a generic system of sections $(\phi_{j,v,1},...,\phi_{j,v,n_{j,v}})$. Since the minimal resolution adapted to a reflexive module dominates the minimal 
resolution adapted to each of its direct factors we conclude that $\pi:\Rs\to X$ dominates the minimal resolution adapted to $M_j$. Denote by 
$\Sf{M}'_{j,v}$ be the full $\Ss{\Rs}$-module associated to $M'_{j,v}$. Let $(\Sf{A}_{j,v},(\psi_{j,v,1},...,\psi_{j,v,n_{j,v}}))$ by the pair 
associated with $(\Sf{M}'_{j,v},(\phi_{j,v,1},...,\phi_{j,v,n_{j,v}}))$ by Theorem~\ref{th:corres}. By Lemma~\ref{lem:ArtinVerdierminadap}, the support
of $\Sf{A}_{j,v}$ is a disjoint union of smooth curves meeting the exceptional divisor transversely at smooth points. Such collection of curves 
is non-empty since the $\Ss{\Rs}$-module $\Sf{M}'_{j,v}$ is non-free. 

The full $\Ss{\Rs}$-module $\Sf{M}$ associated to $M$ is the direct sum of the modules $\Sf{M}'_{j,v}$ when $j$ and $v$ vary. 
Taking the union of the system of sections $(\phi_{j,v,1},...,\phi_{j,v,n_{j,v}})$ letting $j$ and $v$ vary we find a system of sections of 
$\Sf{M}$, whose degeneracy locus is a disjoint union of smooth curves in $\Rs$ meeting the exceptional divisor transversely at smooth points. Moreover
there is at least a meting point for each pair of indexes $(j,v)$, since the full sheaf $\Sf{M}'_{j,v}$ is not free. Consequently, if at least one $M'_j$ is not indecomposable the number of meeting points is strictly larger than $k$. 

On the other hand we know that the degeneracy locus of $\Sf{M}$ for a system of generic sections is a disjoint union of smooth curves in $\Rs$ meeting 
the exceptional divisor transversely at $k$ smooth points. The system of sections obtained as the union of $(\phi_{j,v,1},...,\phi_{j,v,m_v})$ letting 
$j$ and $v$ vary, may be deformed continuously into a system of generic sections, and the corresponding degeneracy loci deform flatly. Since 
it is impossible to deform a curve meeting transversely the exceptional divisor in strictly more than $k$ points into a curve 
meeting transversely the exceptional divisor in precisely $k$ points, we deduce that each $M'_j$ is indecomposable.
\endproof

\section{The cohomology of full sheaves on normal Stein surfaces with Gorenstein singularities}
\label{sec:cohomology}
The main objective of this section is to prove to following theorem.

\begin{theorem}\label{formuladimensionM}
Let $X$ be a Stein normal surface with Gorenstein singularities. Let $M$ be a reflexive $\Ss{X}$-module of rank $r$. 
Let $\pi:\Rs\to X$ be a small resolution with respect to the Gorenstein form, let $Z_k$ be the canocical cycle at $\Rs$, (see Defintion~\ref{def:smallresgor}).
Let $\Sf{M}$ be the full $\Ss{\Rs}$-module associated to $M$. Let $d$ be the specialty defect of $\Sf{M}$. Then we have the equality
\begin{equation*}
\dimc{R^1 \pi_* \Sf{M}} = rp_g - [c_1(\Sf{M})] \cdot [Z_k]  + d.
\end{equation*}
\end{theorem}

This theorem will be very important in the following section. It will allow us to prove that a full special sheaf on a Gorenstein surface  is determined by its first Chern class in the minimal adapted resolution. An inmediate Corollary of the theorem and Proposition~\ref{prop:minadapspproperty} 
shows how to use the cohomology of full sheaves as a resolution invariant for reflexive modules.

\begin{corollary}
\label{cor:dimMadap}
Let $X$ be a Stein normal surface with Gorenstein singularities.
Let $M$ be a special reflexive $\Ss{X}$-module of rank $r$. Let $\pi:\Rs\to X$ be a small resolution with respect to the 
Gorenstein form. Let $\Sf{M}$ be the full $\Ss{\Rs}$-module associated to $M$.
Then the resolution is the minimal adapted resolution if and only if we have the equality $\dimc{R^1\pi_*\Sf{M}}=rp_g$.
\end{corollary}

\begin{proof}[Proof of Theorem~\ref{formuladimensionM}]
The Theorem is local in the base $X$. Therefore we may assume that $(X,x)$ is a normal Gorenstein germ. Then we have dualizing modules 
$\omega_X=\Ss{X}$ and $\omega_{\Rs}=\Lambda^2\Omega^1_{\Rs}=\pi^{!}\Ss{X}$. 

The proof occupies the rest of the section, including some intermediate results, which we will single out as separate Lemmata and Propositions.
Let us start with some preliminary work. 

Let $M$, $\pi:\Rs\to X$, $Z_k$, $\Sf{M}$, $r$ and $d$  be as in the statement of the Theorem. 
Take $r$ generic sections and consider the exact sequence obtained by the sections
\begin{equation}
\label{exctseq:directaM9}
0 \to  \Ss{\Rs}^r \to \Sf{M} \to \Sf{A}' \to 0,
\end{equation}
and its dual
\begin{equation}
\label{exctseq:directaN9}
0 \to \Sf{N} \to \Ss{\Rs}^r \to \Sf{A} \to 0,
\end{equation}
where $\Sf{A}= \Exts_{\Ss{\Rs}}^1 \left(\Sf{A}', \Ss{\Rs} \right)$.

Since dualizing the first morphism of~(\ref{exctseq:directaN9}) we recover the first morphism of~(\ref{exctseq:directaM9}) back, we deduce that
\begin{equation}\label{Adobleext}
\Sf{A}' \cong  \Exts_{\Ss{\Rs}}^1 \left(\Sf{A}, \Ss{\Rs} \right).
\end{equation}

Applying the functor $\pi_* -$ to the exact sequence \eqref{exctseq:directaN9} we get
\begin{equation*}
\begin{tikzpicture}
  \matrix (m)[matrix of math nodes,
    nodes in empty cells,text height=1.5ex, text depth=0.25ex,
    column sep=2em,row sep=0.5em] {
 0 & N & \Ss{X}^r & \pi_* \Sf{A} & R^1 \pi_* \Sf{N} & R^1 \pi_* \Ss{\Rs}^r & 0, \\};
\foreach \x [remember=\x as \lastx (initially 1)] in {2,...,7}
{
\draw[-stealth] (m-1-\lastx) -- (m-1-\x);
}
\end{tikzpicture}
\end{equation*}
where $N$ is $\pi_* \Sf{N}$ and, by Lemma~\ref{lema:dualM}, the module $N$ is equal to the module $M^{\smvee}$.

The last exact sequence can be split as follows
\begin{equation*}
\begin{tikzpicture}
  \matrix (m)[matrix of math nodes,
    nodes in empty cells,text height=1.5ex, text depth=0.25ex,
    column sep=2em,row sep=0.5em] {
 0 & N & \Ss{X}^r & \Sf{C} & 0, \\
0 & \Sf{C} & \pi_* \Sf{A} & \Sf{D} & 0, \\
0 & \Sf{D} & R^1 \pi_* \Sf{N} & R^1 \pi_* \Ss{\Rs}^r & 0, \\};
\foreach \x [remember=\x as \lastx (initially 1)] in {2,...,5}
{
\draw[-stealth] (m-1-\lastx) -- (m-1-\x);
\draw[-stealth] (m-2-\lastx) -- (m-2-\x);
\draw[-stealth] (m-3-\lastx) -- (m-3-\x);
}
\end{tikzpicture}
\end{equation*}
where $\text{lenght}(\Sf{D})=d$.

Dualizing the first and second exact sequence we obtain
\begin{equation}\label{exact:cortaM}
\begin{tikzpicture}
  \matrix (m)[matrix of math nodes,
    nodes in empty cells,text height=1.5ex, text depth=0.25ex,
    column sep=2em,row sep=0.5em] {
0& \Ss{X}^r & M & \Ext_{\Ss{X}}^1 \left(\Sf{C}, \Ss{X} \right) & 0,	\\};
\foreach \x [remember=\x as \lastx (initially 1)] in {2,...,5}
{
\draw[-stealth] (m-1-\lastx) -- (m-1-\x);
}
\draw[-stealth] (m-1-3) edge node [auto] {$h$} (m-1-4);
\end{tikzpicture}
\end{equation}
\begin{equation} \label{exact:CyD}
\begin{tikzpicture}
  \matrix (n)[matrix of math nodes,
    nodes in empty cells,text height=1.5ex, text depth=0.25ex,
    column sep=2em,row sep=0.5em] {
0&\Ext_{\Ss{X}}^1 \left(\pi_* \Sf{A}, \Ss{X} \right) & \Ext_{\Ss{X}}^1 \left(\Sf{C}, \Ss{X} \right) & \Ext_{\Ss{X}}^2 \left(\Sf{D}, \Ss{X} \right) & 0. \\};

\foreach \x [remember=\x as \lastx (initially 1)] in {2,...,5}
{
\draw[-stealth] (n-1-\lastx) -- (n-1-\x);
}
\draw[-stealth] (n-1-2) edge node [auto] {$i$} (n-1-3);
\end{tikzpicture}
\end{equation}

Applying the functor $\pi_* -$ to the exact sequence \eqref{exctseq:directaM9}, using the identification \eqref{Adobleext} 
and comparing with the exact sequence \eqref{exact:cortaM} we obtain the diagram
\begin{equation}\label{diagram:comparacionM}
\begin{tikzpicture}
  \matrix (m)[matrix of math nodes,
    nodes in empty cells,text height=1.5ex, text depth=0.25ex,
    column sep=2em,row sep=2em] {
		0 & \Ss{X}^r & M & \Ext_{\Ss{X}}^1 \left(\Sf{C}, \Ss{X} \right) & 0\\
		0 & \Ss{X}^r & M & \pi_* \Exts_{\Ss{\Rs}}^1 \left(\Sf{A}, \Ss{\Rs} \right) & R^1 \pi_* \Ss{\Rs}^r & R^1 \pi_* \Sf{M} & 0 \\};
\draw[-stealth] (m-1-1) -- (m-1-2);
\draw[-stealth] (m-1-2) -- (m-1-3);
\draw[-stealth] (m-1-3) edge node [auto] {$h$} (m-1-4);
\draw[-stealth] (m-1-4) -- (m-1-5);
\draw[-stealth] (m-2-1) -- (m-2-2);
\draw[-stealth] (m-2-2) -- (m-2-3);
\draw[-stealth] (m-2-3) edge node [auto] {$h$} (m-2-4);
\draw[-stealth] (m-2-4) -- (m-2-5);
\draw[-stealth] (m-2-5) -- (m-2-6);
\draw[-stealth] (m-2-6) -- (m-2-7);
\draw[-stealth] (m-1-2) edge node [auto] {$Id$} (m-2-2);
\draw[-stealth] (m-1-3) edge node [auto] {$Id$} (m-2-3);
\draw[-stealth] (m-1-4)  edge node [auto] {$\theta$} (m-2-4);
\end{tikzpicture}
\end{equation}
where $Id$ is the identity and $\theta$ is the map that makes the diagram commute.

Since $(X,x)$ is Gorenstein and $\pi:\Rs\to X$ is small with respect to the Gorenstein form we have the exact sequence
\begin{equation}
\label{eq:exactseqCs9}
\begin{tikzpicture}
  \matrix (m)[matrix of math nodes,
    nodes in empty cells,text height=1.5ex, text depth=0.25ex,
    column sep=2em,row sep=0.5em] {
0& \Cs{\Rs}  & \Ss{\Rs} & \Ss{Z_K} & 0.	\\};
\foreach \x [remember=\x as \lastx (initially 1)] in {2,...,5}
{
\draw[-stealth] (m-1-\lastx) -- (m-1-\x);
}
\draw[-stealth] (m-1-2) edge node [auto] {$c$} (m-1-3);
\end{tikzpicture}
\end{equation}

Applying the functor $\pi \left( \Exts_{\Ss{\Rs}} \left( \Sf{A}, - \right) \right)$ to the map $c$ we obtain the map
\begin{equation}\label{mapeoC}
\pi \left( \Exts_{\Ss{\Rs}} \left( \Sf{A}, - \right) \right)(c) \colon \pi \left( \Exts_{\Ss{\Rs}} \left( \Sf{A}, \Cs{\Rs} \right) \right)  \to \pi \left( \Exts_{\Ss{\Rs}} \left( \Sf{A}, \Ss{\Rs} \right) \right). 
\end{equation}

Abusing notation, let us denote the previous map by $c$.

Since the singularity $(X,x)$ is Gorenstein, the ring $\Ss{X}$ is the dualizing module for the singularity \cite[Theorem 3.3.7]{BrHe}.
In this case the Grothendieck duality for the map $\pi$ \cite[Ch.~VII]{Har1} establish the isomorphism 
\begin{equation*}
R \pi_* R \Homs \left(-, \Cs{\Rs} \right) \cong R \Hom_{\Ss{X}} \left(R \pi_* -, \Ss{X} \right).
\end{equation*}
Applying this for $\Sf{A}$, and using Grothendieck spectral sequence for the composition of two functors we obtain
an isomorphism
\begin{equation}
\label{eq:isofrakg9}
\mathfrak{g}:\pi_* \Exts_{\Ss{\Rs}}^1 \left(\Sf{A}, \Cs{\Rs} \right) \cong\Ext_{\Ss{X}}^1 \left(\pi_* \Sf{A}, \Ss{X} \right).
\end{equation}

Now using \eqref{exact:CyD}, \eqref{diagram:comparacionM}, \eqref{mapeoC} and \eqref{eq:isofrakg9}, we have the diagram
\begin{equation}\label{diagramacc}
\begin{tikzpicture}
  \matrix (m)[matrix of math nodes,
    nodes in empty cells,text height=1.5ex, text depth=0.25ex,
    column sep=2em,row sep=2em] {
		\pi_* \Exts_{\Ss{\Rs}}^1 \left(\Sf{A}, \Cs{\Rs} \right) & \Ext_{\Ss{X}}^1 \left(\pi_* \Sf{A}, \Ss{X} \right) & \Ext_{\Ss{X}}^1 \left(\Sf{C}, \Ss{X} \right) & \pi_* \Exts_{\Ss{\Rs}}^1 \left(\Sf{A}, \Ss{\Rs} \right) \\
		\pi_* \Exts_{\Ss{\Rs}}^1 \left(\Sf{A}, \Cs{\Rs} \right) & &  & \pi_* \Exts_{\Ss{\Rs}}^1 \left(\Sf{A}, \Ss{\Rs} \right)\\};
\draw[-stealth] (m-1-1) edge node [auto] {$\mathfrak{g}$} (m-1-2);
\draw[-stealth] (m-1-2) edge node [auto] {$i$} (m-1-3);
\draw[-stealth] (m-1-3) edge node [auto] {$\theta$} (m-1-4);
\draw[-stealth] (m-1-1) edge node [auto] {$Id$} (m-2-1);
\draw[-stealth] (m-1-4) edge node [auto] {$Id$} (m-2-4);
\draw[-stealth] (m-2-1) edge node [auto] {$c$} (m-2-4);
\end{tikzpicture}
\end{equation}

\begin{lemma}\label{lema:tecnico}
The diagram commutes.
\end{lemma}
\proof
Denote by $\mathfrak{c}$ the map given by the composition $\theta \circ{} i \circ{} \mathfrak{g}$.
Consider the map $f:= (c - \mathfrak{c})$.
Since the map $\pi \colon \Rs \to X$ is an isomorphism outside the exceptional divisor, we have that for any section $s$ of 
$\pi_* \Exts_{\Ss{\Rs}}^1 \left(\Sf{A}, \Cs{\Rs} \right)$, the section $f(s)$ is supported in the exceptional divisor, 
hence $f(s) \in H^0_E\left(\Exts_{\Ss{\Rs}}^1 \left(\Sf{A}, \Ss{\Rs} \right) \right)$ but this cohomology group is zero.

Therefore for any section $s$ of $\pi_* \Exts_{\Ss{\Rs}}^1 \left(\Sf{A}, \Cs{\Rs} \right)$ we have that $f(s)=0$ which it is equivalent to say 
that the maps $c$ and $\mathfrak{c}$ coincide. 
\endproof

\begin{proposition}\label{proposition:dimM}
We have the equality
\begin{equation*}
\dimc{R^1 \pi_* \Sf{M}} = rp_g - \dimc{\pi_* \Exts_{\Ss{\Rs}}^1 \left(\Sf{A}, \Ss{Z_k} \right)}  + \dimc{\Ext_{\Ss{X}}^2 \left(\Sf{D}, \Ss{X} \right)}.
\end{equation*}
\end{proposition}
\proof
Applying the functor $\Homs(-,-)$ to the exact sequences \eqref{exctseq:directaN9} and \eqref{eq:exactseqCs9} we get
\begin{equation*}
\begin{tikzpicture}
  \matrix (m)[matrix of math nodes,
    nodes in empty cells,text height=1.5ex, text depth=0.25ex,
    column sep=2.5em,row sep=2em] {
		  & 0 & 0 & 0 & \\
		0 & \Cs{\Rs}^r & \Sf{M} \otimes \Cs{\Rs} & \Exts_{\Ss{\Rs}}^1 \left(\Sf{A}, \Cs{\Rs} \right) & 0 \\
		0 & \Ss{\Rs}^r & \Sf{M} & \Exts_{\Ss{\Rs}}^1 \left(\Sf{A}, \Ss{\Rs} \right) & 0 \\
		0 & \Ss{Z_k}^r & \Sf{M} \otimes \Ss{Z_k} & \Exts_{\Ss{\Rs}}^1 \left(\Sf{A}, \Ss{Z_k} \right) & 0 \\
		  & 0 & 0 & 0 & \\};
\foreach \y [remember=\y as \lasty (initially 2)] in {3,4}
{
\foreach \x [remember=\x as \lastx (initially 2)] in {3,4}
{
\draw[-stealth] (m-\y-\lastx) -- (m-\y-\x);
\draw[-stealth] (m-\lasty-\lastx) -- (m-\y-\lastx);
}
}
\draw[-stealth] (m-2-4) edge node [auto] {$c$} (m-3-4);
\draw[-stealth] (m-1-2) -- (m-2-2);
\draw[-stealth] (m-1-4) -- (m-2-4);
\draw[-stealth] (m-4-1) -- (m-4-2);
\draw[-stealth] (m-1-3) -- (m-2-3);
\draw[-stealth] (m-2-1) -- (m-2-2);
\draw[-stealth] (m-2-2) -- (m-2-3);
\draw[-stealth] (m-2-3) -- (m-2-4);
\draw[-stealth] (m-2-4) -- (m-2-5);
\draw[-stealth] (m-3-1) -- (m-3-2);
\draw[-stealth] (m-3-4) -- (m-3-5);
\draw[-stealth] (m-4-4) -- (m-4-5);
\draw[-stealth] (m-4-2) -- (m-5-2);
\draw[-stealth] (m-4-3) -- (m-5-3);
\draw[-stealth] (m-4-4) -- (m-5-4);
\draw[-stealth] (m-2-4) -- (m-3-4);
\draw[-stealth] (m-3-4) -- (m-4-4);
\end{tikzpicture}
\end{equation*}

Applying the functor $\pi_* -$ to the last commutative diagram we get
\begin{equation*}
\begin{tikzpicture}
  \matrix (m)[matrix of math nodes,
    nodes in empty cells,text height=1.5ex, text depth=0.25ex,
    column sep=2em,row sep=2em] {
		  & 0 & 0 & 0 & \\
		0 & \pi_* \Cs{\Rs}^r & \pi_* \left(\Sf{M} \otimes \Cs{\Rs} \right) & \pi_*\Exts_{\Ss{\Rs}}^1 \left(\Sf{A}, \Cs{\Rs} \right) & 0 & 0 & 0\\
		0 & \Ss{X}^r & M & \pi_* \Exts_{\Ss{\Rs}}^1 \left(\Sf{A}, \Ss{\Rs} \right) & R^1 \pi_* \Ss{\Rs}^r & R^1 \pi_* \Sf{M} & 0 \\
		0 & \pi_* \Ss{Z_k}^r & \pi_* \left( \Sf{M} \otimes \Ss{Z_k} \right) & \pi_*\Exts_{\Ss{\Rs}}^1 \left(\Sf{A}, \Ss{Z_k} \right) &  0 & 0 &\\};
\foreach \y [remember=\y as \lasty (initially 2)] in {3}
{
\foreach \x [remember=\x as \lastx (initially 2)] in {3,...,7}
{
\draw[-stealth] (m-\y-\lastx) -- (m-\y-\x);
\draw[-stealth] (m-\lasty-\lastx) -- (m-\y-\lastx);
}
}
\draw[-stealth] (m-3-3) edge node [auto] {$h$} (m-3-4);
\draw[-stealth] (m-3-4) edge node [auto] {$\alpha$} (m-3-5);
\draw[-stealth] (m-2-4) edge node [auto] {$c$} (m-3-4);
\draw[-stealth] (m-1-4) -- (m-2-4);
\draw[-stealth] (m-4-1) -- (m-4-2);
\draw[-stealth] (m-2-7) -- (m-3-7);
\draw[-stealth] (m-3-5) -- (m-4-5);
\draw[-stealth] (m-3-6) -- (m-4-6);
\draw[-stealth] (m-1-2) -- (m-2-2);
\draw[-stealth] (m-1-3) -- (m-2-3);
\draw[-stealth] (m-2-1) -- (m-2-2);
\draw[-stealth] (m-2-2) -- (m-2-3);
\draw[-stealth] (m-2-3) -- (m-2-4);
\draw[-stealth] (m-2-4) -- (m-2-5);
\draw[-stealth] (m-2-5) -- (m-2-6);
\draw[-stealth] (m-2-6) -- (m-2-7);
\draw[-stealth] (m-3-1) -- (m-3-2);
\draw[-stealth] (m-4-2) -- (m-4-3);
\draw[-stealth] (m-4-3) -- (m-4-4);
\draw[-stealth] (m-3-2) -- (m-4-2);
\draw[-stealth] (m-3-3) -- (m-4-3);
\draw[-stealth] (m-3-4) -- (m-4-4);
\draw[-stealth] (m-4-4) -- (m-4-5);
\draw[densely dotted,-stealth] (m-4-2) to [out = -90, in = 90, looseness = .7] (m-2-5);
\draw[densely dotted,-stealth] (m-4-3) to [out = -90, in = 90, looseness = .7] (m-2-6);
\draw[densely dotted,-stealth] (m-4-4) to [out = -90, in = 90, looseness = .7] (m-2-7);
\end{tikzpicture}
\end{equation*}

By this diagram we get
\begin{align}
\begin{split}\label{align:cuentas1}
\dimc{R^1 \pi_* \Sf{M}} &= rp_g - \dimc{\text{Im}(\alpha)},\\
\dimc{\pi_*\Exts_{\Ss{\Rs}}^1 \left(\Sf{A}, \Ss{Z_k} \right)} &= \dimc{\pi_* \Exts_{\Ss{\Rs}}^1 \left(\Sf{A}, \Ss{\Rs} \right) / \pi_* \Exts_{\Ss{\Rs}}^1 \left(\Sf{A}, \Cs{\Rs} \right)},
\end{split}
\end{align}
and
\begin{equation}\label{alphaauxiliar}
\dimc{\im \alpha} =\dimc{ \pi_* \Exts_{\Ss{\Rs}}^1 \left(\Sf{A}, \Ss{\Rs} \right) / \ker \alpha} = \dimc{ \pi_* \Exts_{\Ss{\Rs}}^1 \left(\Sf{A}, \Ss{\Rs} \right) / \im h}.
\end{equation}

Now by \eqref{diagram:comparacionM} we have $\im h = \Ext_{\Ss{X}}^1 \left(\Sf{C}, \Ss{X} \right)$.

Hence by the previous equality, \eqref{alphaauxiliar}, \eqref{exact:CyD} and \eqref{eq:isofrakg9} we get
\begin{align}
\begin{split}\label{imagenalfa1}
\dimc{\im \alpha} &=\dimc{ \pi_* \Exts_{\Ss{\Rs}}^1 \left(\Sf{A}, \Ss{\Rs} \right) / \im h}\\
&=\dimc{ \pi_* \Exts_{\Ss{\Rs}}^1 \left(\Sf{A}, \Ss{\Rs} \right) / \Ext_{\Ss{X}}^1 \left(\Sf{C}, \Ss{X} \right)} \\
&= \dimc{\pi_* \Exts_{\Ss{\Rs}}^1 \left(\Sf{A}, \Ss{\Rs} \right) / \Ext_{\Ss{X}}^1 \left(\pi_* \Sf{A}, \Ss{X} \right)} -\dimc{\Ext_{\Ss{X}}^2 \left(\Sf{D}, \Ss{X} \right)}     \\
&= \dimc{\pi_* \Exts_{\Ss{\Rs}}^1 \left(\Sf{A}, \Ss{\Rs} \right) / \pi_* \Exts_{\Ss{\Rs}}^1 \left(\Sf{A}, \Cs{\Rs} \right)} -\dimc{\Ext_{\Ss{X}}^2 \left(\Sf{D}, \Ss{X} \right)}     .
\end{split}
\end{align}

Let $c$ and $\mathfrak{c}$ be the morphisms given in the diagram \eqref{diagramacc}. Now by \eqref{imagenalfa1} and by Lemma \ref{lema:tecnico} we have
\begin{align}
\begin{split}\label{align:cuentas2}
\dimc{\im \alpha} &= \dimc{\pi_* \Exts_{\Ss{\Rs}}^1 \left(\Sf{A}, \Ss{\Rs} \right) / \pi_* \Exts_{\Ss{\Rs}}^1 \left(\Sf{A}, \Cs{\Rs} \right)} -\dimc{\Ext_{\Ss{X}}^2 \left(\Sf{D}, \Ss{X} \right)}\\
&= \dimc{\pi_* \Exts_{\Ss{\Rs}}^1 \left(\Sf{A}, \Ss{\Rs} \right) / \im \mathfrak{c}} -\dimc{\Ext_{\Ss{X}}^2 \left(\Sf{D}, \Ss{X} \right)}\\
&= \dimc{\pi_* \Exts_{\Ss{\Rs}}^1 \left(\Sf{A}, \Ss{\Rs} \right) / \im c} -\dimc{\Ext_{\Ss{X}}^2 \left(\Sf{D}, \Ss{X} \right)}\\
&= \dimc{\pi_* \Exts_{\Ss{\Rs}}^1 \left(\Sf{A}, \Ss{\Rs} \right) / \pi_* \Exts_{\Ss{\Rs}}^1 \left(\Sf{A}, \Cs{\Rs} \right)} - \dimc{\Ext_{\Ss{X}}^2 \left(\Sf{D}, \Ss{X} \right)}.
\end{split}
\end{align}

Therefore by \eqref{align:cuentas1} and \eqref{align:cuentas2} we get the desired equality.
\endproof

By the Proposition~\ref{proposition:dimM} to prove Theorem~\ref{formuladimensionM} it is enough to prove the equalities 
\begin{equation*}  
\dimc{\pi_* \Exts_{\Ss{\Rs}}^1 \left(\Sf{A}, \Ss{Z_k} \right)} = [c_1(\Sf{M})] \cdot [Z_k].
\end{equation*}
\begin{equation*}  
\dimc{\Ext_{\Ss{X}}^2 \left(\Sf{D}, \Ss{X} \right)}=d.
\end{equation*}

The following lemma gives us the first equality.

\begin{lemma}\label{lemma:A1A2}
Let $\Sf{A}_1$ and $\Sf{A}_2$ be two Cohen-Macaulay $\Ss{\Rs}$-modules of dimension one such that $\Sf{A}_1$ is contained in  $\Sf{A}_2$, the quotient $\Sf{A}_2 / \Sf{A}_1$ is finitely supported and the support of each sheaf intersects the exceptional divisor in finitely many points. Then we have the equality
\begin{equation*}
\dimc{\pi_* \Exts_{\Ss{\Rs}}^1\left(\Sf{A}_1, \Ss{Z_K} \right)} = \dimc{\pi_* \Exts_{\Ss{\Rs}}^1\left(\Sf{A}_2, \Ss{Z_K} \right)}.
\end{equation*}
\end{lemma}
\proof
Let $\Sf{A}_1$ and $\Sf{A}_2$ as in the statement. We have the exact sequence
\begin{equation*}
0 \to \Sf{A}_1 \to \Sf{A}_2 \to \Sf{A}_2/\Sf{A}_1 \to 0.
\end{equation*}

Applying the functor $\Homs_{\Ss{\Rs}}(-, \Ss{Z_K})$ to the last exact sequence we get
\begin{equation}\label{diagram:A1A2}
\begin{tikzpicture}
  \matrix (m)[matrix of math nodes,
    nodes in empty cells,text height=1.5ex, text depth=0.25ex,
    column sep=2.5em,row sep=2em] { 
		0 & \Homs_{\Ss{\Rs}}\left(\Sf{A}_2/\Sf{A}_1, \Ss{Z_K} \right) & \Homs_{\Ss{\Rs}}\left(\Sf{A}_2, \Ss{Z_K} \right) & \Homs_{\Ss{\Rs}}\left(\Sf{A}_1, \Ss{Z_K} \right) &  \\
		& \Exts_{\Ss{\Rs}}^1\left(\Sf{A}_2/\Sf{A}_1, \Ss{Z_K} \right) & \Exts_{\Ss{\Rs}}^1\left(\Sf{A}_2, \Ss{Z_K} \right) & \Exts_{\Ss{\Rs}}^1\left(\Sf{A}_1, \Ss{Z_K} \right) & \\
		& \Exts_{\Ss{\Rs}}^2\left(\Sf{A}_2/\Sf{A}_1, \Ss{Z_K} \right) & \Exts_{\Ss{\Rs}}^2\left(\Sf{A}_2, \Ss{Z_K} \right) & \Exts_{\Ss{\Rs}}^2\left(\Sf{A}_1, \Ss{Z_K} \right) & 0\\};
\foreach \y [remember=\y as \lasty (initially 1)] in {1, 2,3}
{
\foreach \x [remember=\x as \lastx (initially 2)] in {3,4}
{
\draw[-stealth] (m-\y-\lastx) -- (m-\y-\x);
}
}
\draw[-stealth] (m-1-1) -- (m-1-2);
\draw[-stealth] (m-3-4) -- (m-3-5);
\draw[densely dotted,-stealth] (m-1-4) to [out=355, in=175] (m-2-2);
\draw[densely dotted,-stealth] (m-2-4) to [out=355, in=175] (m-3-2);
\end{tikzpicture}
\end{equation}

Since $\Sf{A}_1$ and $\Sf{A}_2$ are Cohen-Macaulay sheaves of dimension one and the support of each sheaf intersects the exceptional divisor finitely we have
\begin{align}
\begin{split}\label{cerosA1A2}
\Homs_{\Ss{\Rs}}\left(\Sf{A}_2, \Ss{Z_K} \right) &= \Homs_{\Ss{\Rs}}\left(\Sf{A}_1, \Ss{Z_K} \right) = 0, \\
\Exts_{\Ss{\Rs}}^2\left(\Sf{A}_2, \Ss{Z_K} \right) &= \Exts_{\Ss{\Rs}}^2\left(\Sf{A}_1, \Ss{Z_K} \right) = 0. \end{split}
\end{align}
The second equality uses Auslander-Buchbaum formula and the fact that each $\Sf{A}_i$ has depth $1$.

Since all the sheaves in \eqref{diagram:A1A2} are supported in a finite set we can work locally, therefore we assume that $\Ss{\Rs}$ is  $\mathbb{C}[x,y]$ and $Z_K$ is $\Ss{\Rs}/(f)$ for some function $f$.

Now by \eqref{diagram:A1A2} and \eqref{cerosA1A2} we just need to prove the following equality
\begin{equation}\label{A1A2reduccion}
\dimc{\Exts_{\Ss{\Rs}}^1\left(\Sf{A}_2/\Sf{A}_1, \Ss{Z_K} \right)} = \dimc{\Exts_{\Ss{\Rs}}^2\left(\Sf{A}_2/\Sf{A}_1, \Ss{Z_K} \right)}.
\end{equation}

Consider the following resolution of $\Ss{Z_K}$
\begin{equation}\label{resolucionC}
\begin{tikzpicture}
  \matrix (m)[matrix of math nodes,
    nodes in empty cells,text height=1.5ex, text depth=0.25ex,
    column sep=2em,row sep=0.5em] {
 0 & \Ss{\Rs} & \Ss{\Rs} & \Ss{Z_K} & 0. \\};
\draw[-stealth] (m-1-1) -- (m-1-2);
\draw[-stealth] (m-1-2) edge node [auto] {$\cdot f$} (m-1-3);
\draw[-stealth] (m-1-3) -- (m-1-4);
\draw[-stealth] (m-1-4) -- (m-1-5);
\end{tikzpicture}
\end{equation}

Applying the functor $\Homs_{\Ss{\Rs}}\left(\Sf{A}_2 / \Sf{A}_1, - \right)$ to the last exact sequence we get
\begin{equation}\label{diagramaA1A2}
\begin{tikzpicture}
  \matrix (m)[matrix of math nodes,
    nodes in empty cells,text height=1.5ex, text depth=0.25ex,
    column sep=2em,row sep=2em] {
  0 & \Homs_{\Ss{\Rs}}\left(\Sf{A}_2 / \Sf{A}_1, \Ss{\Rs} \right) & \Homs_{\Ss{\Rs}}\left(\Sf{A}_2 / \Sf{A}_1, \Ss{\Rs} \right) &  \Homs_{\Ss{\Rs}}\left(\Sf{A}_2 / \Sf{A}_1, \Ss{Z_K} \right) &\\
    & \Exts_{\Ss{\Rs}}^1\left(\Sf{A}_2 / \Sf{A}_1, \Ss{\Rs} \right) & \Exts_{\Ss{\Rs}}^1\left(\Sf{A}_2 / \Sf{A}_1, \Ss{\Rs} \right) & \Exts_{\Ss{\Rs}}^1\left(\Sf{A}_2 / \Sf{A}_1, \Ss{Z_K} \right) &\\
		& \Exts_{\Ss{\Rs}}^2\left(\Sf{A}_2 / \Sf{A}_1, \Ss{\Rs} \right) & \Exts_{\Ss{\Rs}}^2\left(\Sf{A}_2 / \Sf{A}_1, \Ss{\Rs} \right) & \Exts_{\Ss{\Rs}}^2\left(\Sf{A}_2 / \Sf{A}_1, \Ss{Z_K} \right) & 0\\
};
\foreach \y [remember=\y as \lasty (initially 1)] in {1, 2,3}
{
\foreach \x [remember=\x as \lastx (initially 2)] in {3,...,4}
{
\draw[-stealth] (m-\y-\lastx) -- (m-\y-\x);
}
}
\draw[-stealth] (m-1-1) -- (m-1-2);
\draw[-stealth] (m-3-4) -- (m-3-5);
\draw[densely dotted,-stealth] (m-1-4) to [out=355, in=175] (m-2-2);
\draw[densely dotted,-stealth] (m-2-4) to [out=355, in=175] (m-3-2);
\end{tikzpicture}
\end{equation}

Now since the support of $\Sf{A}_2 / \Sf{A}_1$ is zero dimensional, we have by Theorem~\ref{Th:Herzog} 
\begin{align*}
\Exts_{\Ss{\Rs}}^1\left(\Sf{A}_2 / \Sf{A}_1, \Ss{\Rs} \right) &=0.
\end{align*}

By the previous equality and the exact sequence \eqref{diagramaA1A2} we get 
\begin{equation*}
0 \to \Exts_{\Ss{\Rs}}^1\left(\Sf{A}_2 / \Sf{A}_1, \Ss{Z_K} \right)  \to \Exts_{\Ss{\Rs}}^2\left(\Sf{A}_2 / \Sf{A}_1, \Ss{\Rs} \right) \to \Exts_{\Ss{\Rs}}^2\left(\Sf{A}_2 / \Sf{A}_1, \Ss{\Rs} \right) \to \Exts_{\Ss{\Rs}}^2\left(\Sf{A}_2 / \Sf{A}_1, \Ss{Z_K} \right) \to 0.
\end{equation*}

Taking $\CC$-dimensions we immediately obtain:
\begin{equation*}
\dimc{\Exts_{\Ss{\Rs}}^1\left(\Sf{A}_2/\Sf{A}_1, \Ss{Z_K} \right)} = \dimc{\Exts_{\Ss{\Rs}}^2\left(\Sf{A}_2/\Sf{A}_1, \Ss{Z_K} \right)}.
\end{equation*}
\endproof

\begin{proposition}
The equality $\dimc{\pi_* \Exts_{\Ss{\Rs}}^1 \left(\Sf{A}, \Ss{Z_k} \right)} = [c_1(\Sf{M})] \cdot [Z_k]$ holds.
\end{proposition}
\proof
By the previous Lemma it is enough to assume that $\Sf{A}$ is isomorphic to $\Ss{\overline{C}}$, where $\overline{C}$ is the support of $\Sf{A}$. 
A direct computation of $\Exts_{\Ss{\Rs}}^1 \left(\Ss{\overline{C}}, \Ss{Z_k} \right)$ gives the result.
\endproof

\begin{proposition}
The equality $\dimc{\Ext_{\Ss{X}}^2 \left(\Sf{D}, \Ss{X} \right)}=\dimc{\Sf{D}}=d$ holds.
\end{proposition}
\proof
Using Theorem~\ref{Th:Herzog} one easily reduces by induction on $\dimc{\Sf{D}}$ to the case $\Sf{D}=\CC_p$, where $\CC_p$ is the skyscraper sheaf at 
a point $p$ with stalk $\CC$. A direct computation shows that case. 
\endproof

The proof of Theorem~\ref{formuladimensionM} is complete now. 
\end{proof}

\section{The classification and structure of special reflexive modules}
\label{sec:classsprefl}

\subsection{The combinatorial classification}
\label{sec:combclass}
Let $(X,x)$ be a normal surface singularity. 
Lemma~\ref{lem:ArtinVerdierminadap} allows to define the resolution graph of a reflexive module.
Denote by $M$ a reflexive $\Ss{X}$-module, $\pi \colon (\tilde{X},E) \to (X,x)$ the minimal adapted resolution  to $M$, 
$\Sf{M}$ the full sheaf associated to $M$ and $r$ the rank of $\Sf{M}$.
Take $r$ generic sections of $\Sf{M}$ and consider the exact sequence given by the sections 
\begin{equation*}
0 \to \Ss{\Rs}^r \to \Sf{M} \to \Sf{A}' \to 0.
\end{equation*}

By Lemma~\ref{lem:ArtinVerdierminadap} the sheaf $\Sf{A}'$ is isomorphic to $\Ss{D}$, where $D$ is a smooth curve meeting the exceptional divisor transversely at
smooth points. We construct a graph as follows:
\begin{enumerate}
\item Let $\Sf{G}^o_M$ be the dual graph of $\Rs$ of the minimal good resolution that dominates $\pi$,
weighted with the self-intersection and the genus of each component (see~\cite{Ne}).
\item In each vertex $v_i$, add as many arrows as the first Chern class of $\Sf{M}$ intersects the exceptional divisor $E_i$. Call the 
resulting decorated graph $\Sf{G}_M$.
\end{enumerate}

\begin{definition}
\label{def:resgraphM}
The resolution graph $\Sf{G}_M$ of the module $M$ is the graph described in the previous construction. 
\end{definition}

In the next theorem we characterize combinatorially resolution graphs of special modules over Gorenstein surface singularities. By a negative plumbing 
graph we mean the dual graph of a good resolution of a surface singularity. The property of being numerically Gorenstein only depends on the plumbing 
graph. In this case there is a canonical cycle with integral coefficients (see~\cite{Ne}).

\begin{theorem}
\label{th:charresgraphsp}
Let $\Sf{G}$ be a negative definite plumbing graph, such that to some of its vertices there are a finite number of arrows attached. 
There is a Gorenstein surface singularity $(X,x)$ and a special reflexive module whose resolution graph is isomorphic to $\Sf{G}$ if and only if each
of the following properties is satisfied:
\begin{enumerate}
 \item the graph is numerically Gorenstein.
 \item if a vertex has genus $0$, self intersection $-1$ and has at most two neighboring vertices, then it supports at least $1$ arrow.
 \item if a vertex supports arrows then its coefficient in the canonical cycle equals $0$.
\end{enumerate}
\end{theorem}
\proof
Property $(1)$ is necessary because Gorenstein implies numerically Gorenstein. Property $(2)$ holds by the minimality of the good resolution 
dominating the minimal adapted resolution. Property $(3)$ is a direct consequence of Proposition~\ref{prop:minadapspproperty}. 

Conversely, let $\Sf{G}$ be a graph satisfying all the properties. By~\cite{PPP} there is a Gorenstein normal surface singularity $(X,x)$ who has a 
resolution with plumbing graph equal to the result of deleting the arrows of $\Sf{G}$. Let $\pi:\Rs\to X$ be such a resolution. Let $D\subset\Rs$ be a 
smooth curvette meeting the exceptional divisor $E$ transversely at smooth points, and so that for each vertex $v$ of $\Sf{V}$ the number of components 
of $D$ meeting the irreducible component $E_v$ of $E$ corresponding to $v$, is exactly the number of arrows attached to $v$. 

Define $\Sf{A}:=\Ss{D}$ and let $\psi_1,...,\psi_r$ be a set of generators of $\pi_*\Sf{A}$ as a $\Ss{X}$-module. Since $\Sf{A}$ is equal to $\Ss{D}$ and 
the canonical order at the components $E_v$ met by $D$ is $0$, choosing $D$ generic enough we conclude that 
$cond(\mathfrak{K}_{(\Sf{A},(\psi_1,...,\psi_r))})=(0,...,0)$. Hence, by Proposition~\ref{prop:consecuenciaspracticas} (1), we deduce that if 
$(\Sf{M},(\phi_1,...,\phi_r))$ is the result of applying the bijection of Theorem~\ref{th:corres} to $(\Sf{A},(\psi_1,...,\psi_r))$ then 
$\Sf{M}$ is full. It is also special since $\psi_1,...,\psi_r$ generate $\Sf{A}$ as a $\Ss{X}$-module (same argument than at the proof of Corollary~\ref{cor:existenmodespeciales}). 

By Proposition~\ref{prop:minadapnumchar} the resolution $\pi:\Rs\to X$ is the minimal resolution adapted to $\pi_*\Sf{M}$, and by 
Theorem~\ref{th:characspecial} the module $\pi_*\Sf{M}$ is special. It is clear by construction that the resolution graph of $\pi_*\Sf{M}$ equals
$\Sf{G}$.
\endproof

A consequence of the previous Theorem and Proposition~\ref{prop:decompesp} is the following corollary

\begin{corollary}
\label{cor:charresgraphspindec}
Let $\Sf{G}$ be a negative definite plumbing graph, such that to some of its vertices there are a finite number of arrows attached. 
There is a Gorenstein surface singularity $(X,x)$ and an indecomposable special reflexive module  whose resolution graph is isomorphic to $\Sf{G}$ if and only if the 
conditions of the previous Theorem hold and in addition $\Sf{G}$ has only one arrow.
\end{corollary}

\subsection{The first Chern class of a module at its minimal adapted resolution}
\label{sec:1stchernspecial}

Let $X$ be a Stein normal surface with Gorenstein singularities. 
Here we study the relation of a reflexive $\Ss{X}$-module and its first Chern class in the Picard group of its minimal adapted resolution. We show that if the module is special then the first Chern class determines the module, 
providing a vast generalization of the corresponding result of Artin and Verdier for rational double points~\cite{AV}. 

Let $M$ be a reflexive $\Ss{X}$-module of rank $r$. 
Let $\pi:\Rs \to X$ be the minimal resolution adapted to $M$, denote by $E$ the exceptional divisor. Let  
$\Sf{M}$ be the full $\Ss{\Rs}$-module associated to $M$. The first Chern class of $\Sf{M}$ in $Pic(\Rs)$ is the class determined by the determinant 
bundle $\Sf{L}:=det(\Sf{M})$.

\begin{proposition}
\label{prop:extension1st}
The full $\Ss{\Rs}$-module $\Sf{M}$ is an extension of the determinant line bundle $\Sf{L}$ by $\Ss{\Rs}^{r-1}$.
\end{proposition}
\proof
Take $r$ generic sections $(\phi_1,...,\phi_r)$ of $\Sf{M}$ and consider the exact sequence given by the sections 
\begin{equation}
\label{exctseq:directaM55}
0 \to \Ss{\Rs}^r \to \Sf{M} \to \Sf{A}' \to 0.
\end{equation}

Since the resolution is the minimal adapted resolution we have that $\Sf{A}'$ is isomorphic to $\Ss{D}$, where $D$ is a smooth curve meeting
the exceptional divisor $E$ transversely at smooth points by Lemma~\ref{lem:ArtinVerdierminadap}.  

Locally in a trivializing open subset $U$ of the locally free sheaf $\Sf{M}$ we have that the sections can be written as follows
\begin{equation*}Q=
\begin{tikzpicture}[baseline=(m-2-1.base)]
  \matrix (m)[matrix of math nodes,
    left delimiter=(,right delimiter=)]{
    q_{11} &  q_{12} & \dots & q_{1r} \\
		\vdots &  \vdots & \vdots & \vdots \\
		q_{r1} &  q_{12} & \dots & q_{rr} \\};
\end{tikzpicture}
\end{equation*}
where each $q_{ij}$ is an element of $\Ss{\Rs}(U)$.

Therefore $p$ belongs to $D$ if and only if the determinant of $Q(p)$ is equal to zero. 

Since $D$ is smooth, the matrix $Q$ must have at least $r-1$ columns linearly independent. Therefore we can choose $r-1$ sections linear independent 
everywhere. By genericity of the system of sections $(\phi_1,...,\phi_r)$ we may assume that these sections are $(\phi_1,...,\phi_{r-1})$. These sections give us the exact sequence
\begin{equation}
\label{exact:picard}
0 \to \Ss{\Rs}^{r-1} \to \Sf{M} \to \Sf{L} \to 0,
\end{equation}
where $\Sf{L}$ is the line bundle $\text{det}(\Sf{M})$.
\endproof

Now we assume that $M$ is special and $(X,x)$ is Gorenstein and prove stronger properties.

\begin{lemma}\label{lemma:dimL}
The dimension of $R^1 \pi_* \Sf{L}$ is $p_g$.
\end{lemma}
\proof
By Proposition~\ref{prop:minadapspproperty} the resolution $\pi:\Rs\to X$ is small with respect to the Gorenstein form (hence the canonical cycle
$Z_K$ is non-negative), and moreover $D$ does not meet the support of $Z_K$. Therefore we have  
\begin{align*}
\text{Tor}_1^{\Ss{\Rs}}(\Ss{D},\Ss{Z_K}) &= 0, \\
\Ss{D}\otimes\Ss{Z_K} &= 0.
\end{align*}

By these equalities, applying $- \otimes \Ss{Z_K}$  the exact sequence
\begin{equation}
\label{exctseq:detM1}
0 \to \Ss{\Rs} \to \Sf{L} \to \Ss{D} \to 0,
\end{equation}
we get
\begin{equation}
\Ss{Z_K} \cong \Sf{L} \otimes \Ss{Z_K}.
\end{equation}

Now applying the functor $\pi_* -$ to the exact sequence \eqref{exctseq:detM1} and using the last isomorphism we obtain
\begin{equation*}
\begin{tikzpicture}
  \matrix (m)[matrix of math nodes,
    nodes in empty cells,text height=1.5ex, text depth=0.25ex,
    column sep=2.5em,row sep=2em] {
  0 &  \pi_* \Ss{\Rs} & \pi_* \Sf{L} & \pi_* \Ss{D}&  R^1 \pi_* \Ss{\Rs} & R^1 \pi_* \Sf{L} & 0 \\
	 &   &  & 0 &  R^1 \pi_* \Ss{Z_K} & R^1 \pi_* \Sf{L} \otimes \Ss{Z_K} & 0 \\
};

\foreach \x [remember=\x as \lastx (initially 1)] in {2,...,7}
{
\draw[-stealth] (m-1-\lastx) -- (m-1-\x);
}

\foreach \x [remember=\x as \lastx (initially 4)] in {5,...,7}
{
\draw[-stealth] (m-2-\lastx) -- (m-2-\x);
}

\draw[-stealth] (m-1-5) -- (m-2-5);
\draw[-stealth] (m-1-6) -- (m-2-6);
\end{tikzpicture}
\end{equation*}

Since the diagram commutes and $R^1 \pi_* \Ss{\Rs}$ and $R^1 \pi_* \Ss{Z_K}$ are isomorphic we conclude that $R^1 \pi_* \Ss{\Rs}$ and $R^1 \pi_* \Sf{L} $ have the same dimension.
\endproof

\begin{theorem}\label{th:Chernresolucionadapted}
Let $X$ be a Stein normal surface with Gorenstein singularities. Let $M$ be a special $\Ss{X}$-module without free factors. Let $\pi:\Rs\to X$ be the minimal
resolution adapted to $M$, and $\Sf{M}$ the full $\Ss{\Rs}$-module associated to $M$. The module $\Sf{M}$ (and equivalently $M$) is determined by its first Chern class in $\text{Pic}(\Rs)$.
\end{theorem}
\proof

Applying the functor $\pi_* -$ to the exact sequence \eqref{exact:picard} we get
\begin{equation}\label{exac:detMtemp}
\begin{tikzpicture}
  \matrix (m)[matrix of math nodes,
    nodes in empty cells,text height=1.5ex, text depth=0.25ex,
    column sep=2.5em,row sep=2em] {
  0 &  \pi_* \Ss{\Rs}^{r-1} & \pi_* \Sf{M} & \pi_* \Sf{L}&  R^1 \pi_* \Ss{\Rs}^{r-1} & R^1 \pi_* \Sf{M} & R^1 \pi_* \Sf{L} & 0. \\
};
\foreach \x [remember=\x as \lastx (initially 1)] in {2,...,8}
{
\draw[-stealth] (m-1-\lastx) -- (m-1-\x);
}
\end{tikzpicture}
\end{equation}

Since
\begin{align*}
\dimc{R^1 \pi_* \Sf{M}} &= rp_g \quad \text{by Corollary~\ref{cor:dimMadap} and}\\
\dimc{R^1 \pi_* \Sf{L}} &= \dimc{R^1 \pi_* \Ss{\Rs}}=p_g \quad  \text{by Lemma \ref{lemma:dimL}},
\end{align*}
we get that the exact sequence \eqref{exac:detMtemp} split as follows
\begin{equation}
\label{eq:paraelrango}
\begin{tikzpicture}
  \matrix (m)[matrix of math nodes,
    nodes in empty cells,text height=1.5ex, text depth=0.25ex,
    column sep=2.5em,row sep=2em] {
  0 &  \pi_* \Ss{\Rs}^{r-1} & \pi_* \Sf{M} & \pi_* \Sf{L} & 0. \\
};
\foreach \x [remember=\x as \lastx (initially 1)] in {2,...,5}
{
\draw[-stealth] (m-1-\lastx) -- (m-1-\x);
}
\end{tikzpicture}
\end{equation}

Therefore $\pi_* \Sf{M} \in \Ext_{\Ss{X}}^1 \left( \pi_* \Sf{L}, \Ss{X}^{r-1} \right)$. Since the module $\pi_* \Sf{M}$ is reflexive and without free factors we conclude the proof by
Lemma~1.9.ii in \cite{AV} (this Lemma globalizes to the Stein surface situation that we are considering here).
\endproof

\subsection{The classification of special reflexive modules on Gorenstein surface singularities}
\label{sec:clasdef}
Before we state and prove the classification theorem we need the following lemma.

\begin{lemma}
\label{lem:varioD}
Let $(X,x)$ be a normal Gorenstein surface singularity and $\pi:\Rs\to X$ be a resolution which is small with respect to the canonical form  such that for some irreducible component $E_i$ of 
the exceptional divisor $E$ we have $E_i \not \subseteq \text{Supp}(Z_K)$. If $D_1$ and $D_2$ are two irreducible curvettes, 
each one transverse to $E_i$ at regular points of $E$, then we have an isomorphism of line bundles $\Ss{\Rs}(-D_1) \cong \Ss{\Rs}(-D_2)$.
\end{lemma}
\proof
We want to prove that $\Ss{\Rs}(-D_1+D_2)$ is isomorphic to $\Ss{\Rs}$.

Consider the exponential exact sequence
\begin{equation*}
\begin{tikzpicture}
  \matrix (m)[matrix of math nodes,
    nodes in empty cells,text height=1.5ex, text depth=0.25ex,
    column sep=2.5em,row sep=2em] {
  0 &  \mathbb{Z} & \Ss{\Rs} & \Ss{\Rs}^* & 0. \\
};

\draw[-stealth] (m-1-1) -- (m-1-2);
\draw[-stealth] (m-1-2) -- (m-1-3);
\draw[-stealth] (m-1-3) edge node[auto]{$\text{exp}$} (m-1-4);
\draw[-stealth] (m-1-4) -- (m-1-5);
\end{tikzpicture}
\end{equation*}

Applying the functor $\pi_* $ to the previous exact sequence we get
\begin{equation}\label{exact:largaexp}
\begin{tikzpicture}
  \matrix (m)[matrix of math nodes,
    nodes in empty cells,text height=1.5ex, text depth=0.25ex,
    column sep=2.5em,row sep=2em] {
  \dots &  H^1(\Rs, \mathbb{Z}) & H^1(\Rs, \Ss{\Rs}) & H^1(\Rs, \Ss{\Rs}^*) & H^2(\Rs, \mathbb{Z}) & 0. \\
};

\draw[-stealth] (m-1-1) -- (m-1-2);
\draw[-stealth] (m-1-2) -- (m-1-3);
\draw[-stealth] (m-1-3) edge node[auto]{$\text{exp}$} (m-1-4);
\draw[-stealth] (m-1-4) edge node[auto]{$\delta$} (m-1-5);
\draw[-stealth] (m-1-5) -- (m-1-6);
\end{tikzpicture}
\end{equation}

We know that the Picard group of $\Rs$ is $H^1(\Rs, \Ss{\Rs}^*)$ and the morphism $\delta$ is given by taking the first Chern class in cohomology.

By hypothesis we know that $\delta(\Ss{\Rs}(-D_1+D_2)) = 0$. By the exact sequence \eqref{exact:largaexp} we get that there exist an element $f$ in $H^1(\Rs, \Ss{\Rs})$ such that the line bundle given by $\text{exp}(f)$ is isomorphic to $\Ss{\Rs}(-D_1+D_2)$. 

Denote by $E$ the exceptional divisor of $\pi$.
By the location of the curvettes $D_1$ and $D_2$, an easy \v Cech  cohomology computation shows that there exists a finite Stein cover $\Sf{U}=\{U_i\}_{i\in I}$ so that $f\in H^1(\Rs, \Ss{\Rs})=\check{H}^1(\Sf{U}, \Ss{\Rs})$ is represented by a $1$-cocycle $[f_{i,j}]$ with
$f_{i,j}\in \Ss{\Rs}(U_i\cap U_j)$ so that $f_{i,j}=0$ unless $U_i\cap U_j\cap E$ is included in $E_i\setminus Sing(E)$. 

Since the resolution is small with respect to the Gorenstein form, we have the exact sequence
\begin{equation}\tag{\ref{eq:exactseqCs}}
0 \to \Cs{\Rs} \to \Ss{\Rs} \to \Ss{Z_K} \to 0.
\end{equation}
Applying the functor $\pi_* -$ and by Grauert-Riemenschneider Vanishing Theorem we get that the homomorphism 
$$H^1(\Rs, \Ss{\Rs})\to H^1(\Rs, \Ss{Z_K})$$
is an isomorphism. Since the image of the \v Cech cocycle $[f_{i,j}]$ under this isomorphism is obviously zero for having empty support ($Z_k$ does not have support in $E_i$), we deduce that $f=0$ in $H^1(\Rs, \Ss{\Rs})$. This implies that $\Ss{\Rs}(-D_1+D_2)$ is the trivial line bundle. 
\endproof

\begin{definition}
Let $(X,x)$ be a normal surface singularity and $\pi \colon \Rs \to X$ be a resolution. 
 Any irreducible component $E_i$ of the exceptional divisor is called \emph{a divisor over $X$}.
\end{definition}

\begin{remark}
 \label{rem:resoluciondivisores}
Let $E_1, \dots, E_n$ be a collection of divisors over $X$. Then there exists a unique minimal resolution $\pi \colon \Rs \to X$ such that 
$E_1, \dots, E_n$ are irreducible components of the exceptional divisor.
\end{remark}

Now we present the classification theorem: 

\begin{theorem}
\label{Teo:final}
Let $(X,x)$ be a Gorenstein surface singularity. Then there exists a bijection between the following sets: 
\begin{enumerate}
\item The set of special $\Ss{X}$-modules without free factors up to isomorphism.
\item The set of finite pairs $(E_1, n_1), \dots, (E_l, n_l)$ where each $E_i$ is a divisor over $X$ and $n_i$ is a positive integer, 
such the minimal resolution given by Remark~\ref{rem:resoluciondivisores} is small with respect to the Gorenstein form and the
Gorenstein form does not have any pole in the components $E_1, \dots, E_l$.
\end{enumerate}
\end{theorem}
\proof
Let $M$ be a special $\Ss{X}$-module and $\pi \colon \Rs \to X$ be the minimal resolution adapted to $M$ with exceptional divisor 
$E=\bigcup_{i=1}^l E_i$. Denote by $\Sf{M}$ the full sheaf associated to $M$ and by $n_j=c_1(\Sf{M}) \cdot E_j$ for $j=1, \dots, l$. 
We associate to the module $M$ the pairs $(E_1,n_1), \dots, (E_k,n_k)$ such that $n_j$ is different form zero. 

In order to prove the surjectivity of the previous assignment consider $(E_1, n_1), \dots, (E_l, n_l)$ where each $E_i$ is a divisor over $X$ and 
$n_i$ is a positive integer and denote by $\pi \colon \Rs \to X$ be the resolution given by Remark~\ref{rem:resoluciondivisores}. The 
divisors are so that $\pi \colon \Rs \to X$ is small with respect to the Gorenstein form, and the coefficient of the canonical cycle
at each of them vanishes. For each positive integer $n_j$ take a smooth curvette $D_j$ with $n_j$ irreducible components such that $D_j$ intersects only 
the irreducible component $E_j$ and the intersection is transverse. Denote by $D= D_1 \coprod \dots \coprod D_l$. Let $(\psi_1,...,\psi_r)$ be a minimal
set of generators of $\pi_*\Ss{D}$ as a $\Ss{X}$-module. Since the $E_i$'s are not at the support of the canonical cycle we have that the 
minimal conductor $cond(\mathfrak{K}_{(\Ss{D},(\psi_1,...,\psi_r))})$ equals $0$. Let $(\Sf{M},(\phi_1,...,\phi_r))$ be pair associated to
$(\Ss{D},(\psi_1,...,\psi_r))$ by the correspondence of Theorem~\ref{th:corres}. By Proposition~\ref{prop:consecuenciaspracticas}, (1) the module
$\Sf{M}$ is full. Since $(\psi_1,...,\psi_r)$ generate of $\pi_*\Ss{D}$ as a $\Ss{X}$-module, we have that $\Sf{M}$ is special. By 
Theorem~\ref{th:characspecial} the module $M:=\pi_*\Sf{M}$ is special. The equality $cond(\mathfrak{K}_{(\Ss{D},(\psi_1,...,\psi_r))})=0$ implies that
$\pi:\Rs\to X$ is the minimal resolution adapted to $M$ by Proposition~\ref{prop:minadapnumchar}. By construction, the previous assignment applied to
$M$ gives $(E_1, n_1), \dots, (E_l, n_l)$. In order to prove surjectivity we need that $M$ does not have free factors. If $M$ has free factors we write $M=M_0\oplus\Ss{X}^a$, where $M_0$ is without free factors. Then the previous assignment applied to $M_0$ also gives 
 $(E_1, n_1), \dots, (E_l, n_l)$ and surjectivity is proven.

The injectivity follows from Theorem~\ref{th:Chernresolucionadapted} and Lemma~\ref{lem:varioD}.
\endproof

\begin{remark}
\label{rem:Dgenericanofactorlibre}
If the union of curvettes $D$ in the previous proof is chosen generic then the obtained module $M$ does not have free factors. 
\end{remark}
\proof
By the previous Theorem there is a unique reflexive $\Ss{X}$-module $M_0$ without free factors associated with the set of pairs  $(E_1, n_1), \dots, (E_l, n_l)$. Let $r$ be its rank and $(\phi_1,...,\phi_r)$ be a set of generic sections. Let $(\Sf{C},(\psi_1,...,\psi_r))$
be the pair associated with $(M_0,(\phi_1,...,\phi_r))$ by Theorem~\ref{th:corrsing}. Since $M_0$ is without free factors and the sections
are generic $(\psi_1,...,\psi_r)$ is a minimal set of generators of $\Sf{C}$. Taking $\pi:\Rs\to X$ the minimal adapted resolution to $M_0$ and using specialty and Proposition~\ref{prop:invressing}, we have that $\Sf{C}=\pi_*\Ss{D}$ for a curve $D$ as in the previous proof.

Let $D'$ be a union of curvettes as in the proof of the previous Theorem.
Let $(\psi_1,...,\psi_r)$ be a minimal
set of generators of $\pi_*\Ss{D'}$ as a $\Ss{X}$-module. Let $(M',(\phi_1,...,\phi_r))$ be the pair associated with $(\pi_*\Ss{D'},(\psi_1,...,\psi_r))$ by Theorem~\ref{th:corrsing}. According with Proposition~\ref{prop:invressing}, if $M$ is the module associated with 
$D'$ by the previous proof, then we have the equality $M=M'$.  

The minimal number of generators of $\pi_*\Ss{D'}$ as a $\Ss{X}$-module is upper semi-continuous under deformation of $D'$. Then the minimal number of generators among all choices of $D'$ as in the previous proof is $r$: if it were smaller the module $M_0$ would have rank smaller than $r$. 

Now, if $D'$ is chosen generic the module $M$ associated with $D'$ by the previous proof has rank $r$ and contains $M_0$ as a direct factor, hence it is equal to $M_0$.
\endproof

\begin{corollary}\label{Cor:finalprincipal}
Let $(X,x)$ be a Gorenstein surface singularity. Then there exists a bijection between the following sets: 
\begin{enumerate}
\item The set of special, indecomposable $\Ss{X}$-modules up to isomorphism.
\item The set of irreducible divisors $E$ over $x$, such at any resolution of $X$ where $E$ appears, the Gorenstein form has not either 
zeros or poles along $E$.
\end{enumerate}
\end{corollary}
\proof
It follows immediately from Theorem~\ref{Teo:final} and Proposition~\ref{prop:decompesp}.
\endproof

Notice that if $(X,x)$ is a rational double point, then the previous Corollary is the McKay correspondence given by Artin and Verdier~\cite{AV}.

\begin{corollary}[\cite{AV}]
Let $(X,x)$ be a rational double point and denote by $\pi\colon \Rs \to X$ the minimal resolution with exceptional divisor $E=\bigcup_{i=1}^l E_i$. Then there exists a bijection between the following sets: 
\begin{enumerate}
\item The set of reflexive, indecomposable $\Ss{X}$-modules up to isomorphism.
\item The set of irreducible divisors $E_i$ where $E_i$ is an irreducible component of the exceptional divisor $E$.
\end{enumerate}
\end{corollary}
\proof
Since the singularity is a rational double point the following two sets are the same:
\begin{enumerate}
\item The set of irreducible divisors $E_i$ where $E_i$ is an irreducible component of the exceptional divisor $E$.
\item The set of irreducible divisors $E'_i$ where $E'_i$ is a divisor over $X$, such that the minimal resolution given by Lemma~\ref{rem:resoluciondivisores} is small with respect to the Gorenstein form and the Gorenstein form does not have any pole in the components $E'_i$.
\end{enumerate}

Now the Corollary follows immediately from Corollary~\ref{Cor:finalprincipal} and from the fact that any reflexive module on a rational double point singularity is special.
\endproof

\section{Deformations of reflexive modules and full sheaves}
\label{sec:defs}

In the next sections we study deformations of reflexive modules. We treat simultaneously deformations over complex spaces and over complex algebroid germs (spectra of noetherian complete $\CC$-algebras). 

\subsection{The deformation functors}\label{sec:deffunctors}
We assume basic knowledge on Deformation Theory. We follow~\cite{Har4} as a basic reference. 
In order to fix terminology we recall some known definitions. 

\begin{notation}
\label{not:restrictionstalk}
Let $\Sf{Y}\to S$ be flat morphism of two complex spaces, and $y,s$ be points in each of them. Let $\overline{M}$ be a $\Ss{\Sf{Y}}$-module. We will use the
notation $\overline{M}|_s:=\overline{M}\otimes_{\Ss{S}}(\Ss{S}/\mathfrak{m}_s)$,
where $\mathfrak{m}_s$ denotes the maximal ideal at $s$. Clearly $\overline{M}|_s$ is a $\Ss{\Sf{Y}_s}$ module, where $\Sf{Y}_s$ is the fibre of $\Sf{Y}$ over $s$. Furthermore $(\overline{M}|_s)_y$ will denote the stalk of $\overline{M}|_s$ at $y$.
\end{notation}

\begin{definition}
\label{def:deforfam}
Let $Y$ be a either complex space or an algebroid germ and $M$ be a $\Ss{Y}$-module. 
\begin{enumerate}
 \item A {\em deformation} of $(Y,M)$ over a germ of complex space (or an algebroid germ) $(S,s)$ is a
triple $(\Sf{Y},\overline{M},\iota)$ where $\Sf{Y}$ is a flat deformation of $Y$ over $S$,  $\overline{M}$ is a $\Ss{\Sf{Y}}$-module which is flat over $S$, and $\iota$ is an isomorphism from $M$ to the fibre $\overline{M}|_s$.
 \item A {\em deformation fixing} $Y$ of $M$ over a germ of complex space $(S,s)$ is a deformation of $(Y,M)$ over $(S,s)$ such that $\Sf{Y}$ is the trivial deformation $Y\times S$.
\item Given a flat morphism $\Sf{Y}\to S$, a {\em flat family of modules on} $\Sf{Y}$ is a $\Ss{\Sf{Y}}$-module $\overline{M}$ which is flat over $S$. 
\item A {\em flat family of $\Ss{Y}$-modules fixing} $Y$ is a flat family of modules on $Y\times S$.
\end{enumerate}
\end{definition}

Deformations of a pair $(Y,M)$ form a contravariant functor $\mathbf{Def_{Y,M}}$ from the category of germs of complex spaces to the category 
of sets in the usual way: morphisms of germs are transformed into mappings of set via the pull-back of deformation. Likewise deformations of $M$ fixing the base form a contravariant functor $\mathbf{Def_{M}}$. If we restrict to the category of spectra 
of Artinian $\mathbb{C}$-algebras, we can view the functor as a covariant functor from Artinian $\mathbb{C}$-algebras to sets.
It is easy to check that Schlessinger conditions $(H0)-(H2)$ of Theorem 16.2 of~\cite{Har4} are satisfied for these two functors. 

Let us remark that reflexiveness is an open property (see the next Lemma), and hence the deformation notion of Definition~\ref{def:deforfam} is adequate as a deformation notion of reflexive sheaves.

\begin{lemma}
\label{lema:reflexiveopen}
Let $\sigma:\Sf{Y}\to S$ be a flat family of normal surfaces. Let $\overline{M}$ be a family of modules on $\Sf{Y}$. Let $y\in\Sf{Y}$ be a point. Suppose that the $\Ss{\Sf{Y}_{\sigma(y)},y}$-module $(\overline{M}|_{\sigma(y)})_y$ is reflexive. There exists an open neighborhood $U$ of $y$ in $\Sf{Y}$ such that for any $y'\in U$ the $\Ss{\Sf{Y}_{\sigma(y')},y'}$-module $(M|_{\sigma(y')})_{y'}$ is reflexive.
\end{lemma}
\proof
Reflexiveness is equivalent to being Cohen-Macaulay of dimension $2$. An straightforward adaptation of EGA IV~\cite[\S~6.11]{Gr} shows that the locus 
where $(M|_{s'})_{x'}$ is Cohen-Macaulay of dimension $2$ is open.
\endproof

\begin{remark}
Since being Gorenstein is equivalent to ask that the dualizing sheaf is an invertible sheaf, it is easy to prove that being Gorenstein is also an open property. In our work we do not use that property, so we omit the proof.
\end{remark}

On the other hand, if we work at the resolution, fullness is not an open property, so one needs to restrict the deformations of 
Definition~\ref{def:deforfam} in order to get a good notion of deformations and flat families of full sheaves.

\begin{definition}[Laufer~\cite{La}]
\label{def:veryweak}
Let $\Sf{X}\to S$ be a flat family of normal Stein surfaces. A {\em very weak simultaneous resolution} of $\Sf{X}\to S$ is a proper birational morphism $\Pi:\tilde{\Sf{X}}\to\Sf{X}$ satisfying
\begin{enumerate}
 \item $\tilde{\Sf{X}}$ is flat over $S$,
 \item For any closed point $s\in S$ the morphism $\Pi|_s:\tilde{\Sf{X}}_s\to\Sf{X}_s$ is a resolution of singularities. 
\end{enumerate}
We will use the following notation: for any $s\in S$ denote the restriction $\Pi|_{\Rsd_s}$ by $\Pi_s:\Rsd_s\to \Sf{X}_s$.
\end{definition}

\begin{lemma}
\label{lem:h0h1}
Let $\Sf{X}\to S$ be a flat family of normal Stein surfaces and $\Pi:\tilde{\Sf{X}}\to\Sf{X}$ a very weak simultaneous resolution.
Let $\overline{\Sf{M}}$ be a $\Ss{\Rsd}$-module which is flat over $\Ss{S}$. Then the first 2 of the following 3 conditions are equivalent and imply the third:
\begin{enumerate}
 \item $R^1\Pi_*\overline{\Sf{M}}$ is flat as $\Ss{S}$-module in a neighborhood of a point $s\in S$;
 \item the natural map $(\Pi_*\overline{\Sf{M}})|_s\to (\Pi_s)_*(\overline{\Sf{M}}|_s)$ is an isomorphism;
 \item $\Pi_*\overline{\Sf{M}}$ is flat as $\Ss{S}$-module in a neighborhood of a point $s\in S$.
\end{enumerate}
If $S$ is the spectrum of an artinian algebra then the third condition is equivalent to the first two conditions. 

For any morphism $\phi:S'\to S$ let $\Pi': \Rsd\times_S S'\to \Sf{X}\times_S S'$, $\tilde{\psi}:\Rsd\times_S S'\to \Rsd$ and $\psi:\Sf{X}\times_S S'\to \Sf{X}$ be the natural maps. If the $3$ previous conditions are satisfied for any $s\in S$ then the natural map
\begin{equation}
 \label{eq:basechange}
\psi^*(R^1\Pi_*\overline{\Sf{M}})\to R^1(\Pi')_*\left(\tilde{\psi}^*\overline{\Sf{M}}\right),
\end{equation}
is an isomorphism.
\end{lemma}
\proof
The proof is an adaptation of the methods of Section III. 12 of~\cite{Har2}. 
The main difference is that in our case the morphism $\Pi$ is not projective, and that $\Pi_*\overline{\Sf{M}}$ is not coherent over $S$.
Now we explain the changes needed in each of the results from Hartshorne book that we will use; we numerate the results as Hartshorne does. Our base $S$ in Hartshorne's setting is the spectrum of a ring $A$. Easy adaptations of the proofs allow to modify Hartshorne 
statements as follows: Proposition 12.1 is true without modification. 
The complex $L^\bullet$ of Proposition 12.2 is bounded above, $L^i$ is finitely generated and free over $A$ if $i>0$, and 
only flat over $A$ if $i\leq 0$. Proposition 12.4 is true asking $W^i$ to be flat over $A$ instead of projective, if $i=0$, and not asking $Q$ to be
finitely generated if $i=0$. Proposition 12.5 is true as stated. Corollary 12.6 is true if one asks $T^0(A)$ to be flat instead of projective. 
Proposition 12.10 works as stated; the only point of the proof of Proposition 12.10 that needs some care is the following:
the Theorem on Formal Functions (\cite{Har2}, Chapter III, Theorem 11.1) is used only for the $0$-th cohomology. This theorem assumes projectivity
for the morphism $\Pi$, but for $0$-th cohomology the theorem works without this hypothesis.

Now let us proceed to the proof using Hartshorne language. There are only two functors, $T^0$ and $T^1$. Thus $T^0$ is left exact and $T^1$ right 
exact. Condition $(1)$ translates in the flatness of $T^1(A)$, which by the adapted Corollary 12.6 of Hartshorne is equivalent to the exactness of $T^1$. This
is equivalent to the exactness of $T^0$, and by the adapted Corollary 12.6 this implies the flatness of $T^0(A)$, which is exactly Condition $(3)$.

Condition $(2)$ is a particular case of the isomorphism~(\ref{eq:basechange}). If the first condition hold then $T^1$ is exact and hence $T^0$ is 
right exact. Proposition 12.5 of Hartshorne implies the isomorphism~(\ref{eq:basechange}). 

Then a direct application of Proposition 12.10 of~\cite{Har2})
gives that Condition $(2)$ implies the right exactness of $T^0$. This implies the exactness of $T^1$ and, by the adapted Corollary 12.6, Condition $(1)$ 
holds.

Suppose that $S$ is the spectrum of an artinian algebra. The Artinian Principle of Exchange of~\cite{Wa} gives the equivalence between the 
first and third conditions. 
\endproof

\begin{definition}
\label{def:deformationfull}
Let $X$ be a normal Stein surface, $\pi:\Rs\to X$ be a resolution and $\Sf{M}$ be a full $\Ss{\Rs}$-module. 
\begin{enumerate}
 \item A {\em deformation} of $(\Rs,X,\Sf{M})$ over a germ of complex space $(S,s)$ is a cuadruple $(\Rsd,\Sf{X},\overline{\Sf{M}},\iota)$, where $\Sf{X}$ is a flat deformation of $X$ over $(S,s)$, there is a proper birational morphism $\Pi:\Rsd\to\Sf{X}$ which is a very weak simultaneous resolution, $\overline{\Sf{M}}$ is $\Ss{\Rsd}$-module which is flat over $S$, and $\iota$ is an isomorphism from $\Sf{M}$ to $\overline{\Sf{M}}|_s$. A {\em deformation fixing $X$} of $(\Rs,X,\Sf{M})$  is a deformation $(\Rsd,\Sf{X},\overline{\Sf{M}},\iota)$ where $\Sf{X}$ is the trivial deformation. A {\em deformation fixing $(\Rs,X)$} of $(\Rs,X,\Sf{M})$ is a deformation where $\Sf{X}$ and $\Rsd$ are trivial deformations. 
\item A {\em full deformation} of $(\Rs,X,\Sf{M})$ is a deformation $(\Rsd,\Sf{X},\overline{\Sf{M}},\iota)$ such that $R^1\Pi_*\overline{\Sf{M}}$ is flat as $\Ss{S}$-module. 
\item Given a morphism $\Pi:\Rsd\to\Sf{X}$ as above, a family of full modules on $\Rsd$ is a triple $(\Rsd,\Sf{X},\overline{\Sf{M}})$ where
$\overline{\Sf{M}}$ is a $\Ss{\Rsd}$-module which is flat over $S$, the $\Ss{S}$-module $R^1\Pi_*\overline{\Sf{M}}$ is flat, and the $\Ss{\Rsd_s}$-module $\overline{\Sf{M}}|_{s}$ is full for all $s\in S$. A flat family fixing $X$ and/or $(\Rs,X)$ is defined in the obvious way.
\end{enumerate}
\end{definition}

The reader may have noticed that while in the definition of full family we ask that the $\Ss{\Rsd_s}$-module $\overline{\Sf{M}}|_{s}$ is asked to be full for all $s\in S$, we do not ask the same property for full deformations over a germ $(S,s)$. The reason is the following proposition, which shows that fullness is an open property in full deformations.

\begin{proposition}
\label{prop:fullopen}
Let $X$ be a normal Stein surface, $\pi:\Rs\to X$ be a resolution and $\Sf{M}$ be a full $\Ss{\Rs}$-module. Let $(\Rsd,\Sf{X},\overline{\Sf{M}},\iota)$ be a full deformation of $(\Rs,X,\Sf{M})$ over $(S,s)$. There exists an open neighborhood $W$ of $s\in S$ such that
for any $s'\in U$ the $\Ss{\Rsd_s}$-module $\overline{\Sf{M}}|_{s}$ is full. 
\end{proposition}
\proof
We will use the characterization of Proposition~\ref{fullcondiciones}. 

Since $\Sf{M}$ is locally free, by flatness we have that $\overline{\Sf{M}}$ is also locally free. Then $\overline{\Sf{M}}|_{s'}$ is locally 
free for any $s'\in S$. 

Since $R^1\Pi_*\overline{\Sf{M}}$ is flat over $S$, by Lemma~\ref{lem:h0h1} we have that $\Pi_*\overline{\Sf{M}}$ is flat over $S$ and also the equality 
\begin{equation}
\label{eq:igualdadcommutada}
(\Pi_*\overline{\Sf{M}})|_{s'}=(\Pi|_{\Rsd_{s'}})_*\overline{\Sf{M}}|_{s'}.
\end{equation}

This implies that $\Pi_*\overline{\Sf{M}}$ is a flat deformation of the reflexive module $(\Pi|_{\Rsd_{s}})_*\Sf{M}$ (here we use that $\Sf{M}$ is full). By openness of reflexivity (Lemma~\ref{lema:reflexiveopen}), for any $s'$ is a neighborhood of $s$ in $S$ we have that $(\Pi|_{\Rsd_{s'}})_*\overline{\Sf{M}}|_{s'}$ is a reflexive $\Ss{\Sf{X}_{s'}}$-module. 

Let $E_{s'}$ be the exceptional divisor at $\Rsd_{s'}$. We have the local cohomology exact sequence 
$$0\to H^0_{E_{s'}}(\overline{\Sf{M}}|_{s'})\to H^0(\Rsd_{s'},\overline{\Sf{M}}|_{s'})\to H^0(\Rsd_{s'}\setminus E_{s'},\overline{\Sf{M}}|_{s'})\to H^1_{E_{s'}}(\overline{\Sf{M}}|_{s'})\to H^1(\Rsd_{s'},\overline{\Sf{M}}|_{s'}).$$

The morphism $H^0(\Rsd_{s'},\overline{\Sf{M}}|_{s'})\to H^0(\Rsd_{s'}\setminus E_{s'},\overline{\Sf{M}}|_{s'})$ is surjective since $H^0(\Rsd_{s'},\overline{\Sf{M}}|_{s'})$ is a reflexive $\Ss{\Sf{X}_{s'}}$-module and $\Rsd_{s'}\setminus E_{s'}$ is identified with $\Sf{X}_{s'}$ minus a finite set of points (where the modification takes place). Hence $H^1_{E_{s'}}(\overline{\Sf{M}}|_{s'})\to H^1(\Rsd_{s'},\overline{\Sf{M}}|_{s'})$ is injective as needed. 

Let $(\phi_1,...,\phi_m)$ be a collection of global sections of $\Sf{M}$, which almost generate it except at a finite set $Z\subset\Rs$. By Equality~(\ref{eq:igualdadcommutada}) there exist $(\overline{\phi}_1,...,\overline{\phi}_m)$, global sections of $\overline{\Sf{M}}$ which specialize to $(\phi_1,...,\phi_m)$ at the fibre $\Rs$ over $s$. Let $\overline{Z}\subset\Rsd$ denote the locus where the sections $(\overline{\phi}_1,...,\overline{\phi}_m)$ do not generate $\overline{\Sf{M}}$. Then we have the equality $\overline{Z}\cap \Rs=Z$, and as a consequence there is an open neighborhood $U$ of $s$ in $S$ such that $\overline{Z}\cap \Rsd_{s'}$ is finite for any $s'\in S$. Over $U$ we have that $\overline{\Sf{M}}|_{s'}$ is generically generated by global sections. 
\endproof

At the following proposition we introduce the relevant deformation functors. 

\begin{proposition}
\label{propdef:defuntors}
Deformations of $(\Rs,X,\Sf{M})$ form a contravariant functor from the category of germs of complex spaces to the category of sets. We denote it by $\mathbf{Def_{\Rs,X,\Sf{M}}}$. The functors of deformations fixing $X$ and $(\Rs,X)$ are denoted respectively by $\mathbf{Def_{\Rs,\Sf{M}}}$ and $\mathbf{Def_{\Sf{M}}}$.

Full deformations of $(\Rs,X,\Sf{M})$ form a contravariant functor denoted by $\mathbf{FullDef_{\Rs,X,\Sf{M}}}$. The functors of full deformations fixing $X$ and $(\Rs,X)$ are denoted respectively by $\mathbf{FullDef_{\Rs,\Sf{M}}}$ and $\mathbf{FullDef_{\Sf{M}}}$.

These functors, restricted to the artinian basis, may be seen as a covariant functor from the category of artinian algebras to sets. 
\end{proposition}
\proof
The only non-trivial point is to prove that pullback of full deformations are full deformations, but this follows from the 
isomorphism~(\ref{eq:basechange}). 
\endproof
 
It is again easy to check that Schlessinger conditions $(H0)-(H2)$ are satisfied for the functors defined at the previous proposition.

\begin{proposition}
\label{prop:naturaltrans}
Let $\pi:\Rs\to X$ be a resolution of a normal Stein surface, let $\Sf{M}$ be a full $\Ss{\Rs}$-module and $M=\pi_*\Sf{M}$ be its associated 
reflexive module. The push forward operation $\Pi_*$ along the resolution map defines a natural transformation from $\mathbf{FullDef_{\Rs,X,\Sf{M}}}$ to $\mathbf{Def_{X,M}}$. 

Analogous statements holds for the deformation functors fixing $X$ and/or $\Rs$, and for families of full modules. 
\end{proposition}
\proof
For the assertion about deformations let $(\Rsd,\Sf{X},\overline{\Sf{M}},\iota)$ be a full deformation of $(\Rs,X,\Sf{M})$.  
Flatness of the push down $\Pi_*\overline{\Sf{M}}$ holds by Lemma~\ref{lem:h0h1}. 
The isomorphism from $(\Pi_*\overline{\Sf{M}})|_s$ to $M$ is obtained 
composing the natural isomorphism (Lemma~\ref{lem:h0h1}) with the isomorphism $\Pi_*\iota$. 
  
The remaining assertions are proved similarly. 
\endproof

\begin{remark}
Although there is a bijection between full sheaves and reflexive sheaves~\cite{Ka}, as we will see below the transformation 
$\mathbf{FullDef_{\Sf{M}}}\to\mathbf{Def_M}$ is not an isomorphism of functors.
\end{remark}

\begin{proposition}
\label{prop:miniversalexist}
Let $\pi \colon \Rs\to X$ be a resolution of a normal Stein surface. Let $\Sf{M}$ be a full $\Ss{\Rs}$-module and $M=\pi_*\Sf{M}$ be its associated reflexive $\Ss{X}$-module. Then the deformation functors $\mathbf{Def_{\Rs,X,\Sf{M}}}$, $\mathbf{Def_{X,\Sf{M}}}$, $\mathbf{Def_{\Sf{M}}}$, $\mathbf{FullDef_{\Rs,X,\Sf{M}}}$, $\mathbf{FullDef_{X,\Sf{M}}}$, $\mathbf{FullDef_{\Sf{M}}}$, $\mathbf{Def_{X,M}}$ and $\mathbf{Def_M}$ have miniversal deformations.

Let $(\Rsd,\Sf{X},\overline{\Sf{M}},\iota)$ be the miniversal deformation of $\mathbf{Def_{\Rs,X,\Sf{M}}}$, then the miniversal deformation of  
$\mathbf{FullDef_{\Rs,X,\Sf{M}}}$ is the stratum of the flattening stratification of $R^1\Pi_*\overline{\Sf{M}}$ containing the origin. Analogous statements hold for $\mathbf{Def_{X,\Sf{M}}}$ / $\mathbf{FullDef_{X,\Sf{M}}}$ and $\mathbf{Def_{\Sf{M}}}$ / $\mathbf{FullDef_{\Sf{M}}}$.

\end{proposition}
\proof
By Theorem 16.2 of~\cite{Har4} in order to prove the existence of miniversal deformation we only have to prove Schlessinger $(H3)$ condition.  

For $\mathbf{Def_M}$ this amounts to prove that $\Ext^1_{\Ss{X}}(M,M)$ is finite dimensional (see for example~\cite{Tr})
Since $M$ is reflexive, it is locally free at $X\setminus Sing(X)$. Therefore $\Ext^1_{\Ss{X}}(M,M)$ is finite dimensional as needed. 
The functor $\mathbf{Def_{X,M}}$ fibres over the functor of deformations of $X$, with fibre the functor $\mathbf{Def_M}$. Since both the base and fibre functors satisfy $(H3)$, so $\mathbf{Def_{X,M}}$ does it. 

Since $\Sf{M}$ is locally free the local to global spectral sequence shows that $\Ext^1_{\Ss{\Rs}}(\Sf{M},\Sf{M})$ is isomorphic to 
$R^1\pi_*\Homs_{\Ss{\Rs}}(\Sf{M},\Sf{M})$, which is finite dimensional. Hence the Schlessinger condition $(H3)$ is satisfied for the functor
$\mathbf{Def_{\Sf{M}}}$. Since the functor $\mathbf{FullDef_{\Sf{M}}}$ is a sub-functor of $\mathbf{Def_{\Sf{M}}}$, the condition
$(H3)$ also holds for it. For the functors $\mathbf{Def_{\Rs,X,\Sf{M}}}$, $\mathbf{Def_{X,\Sf{M}}}$, $\mathbf{FullDef_{\Rs,X,\Sf{M}}}$, $\mathbf{FullDef_{X,\Sf{M}}}$ we use fibration of functors arguments as before. 

The flattening stratification statement is by versality and definition of the functors.
\endproof

Let us give a basic proposition that we will use later.

\begin{proposition}
\label{prop:dualMdeformado}
Let $\pi\colon \Rs\to X$ be a resolution of a normal Stein surface, let $(\Rsd,\Sf{X},\overline{\Sf{M}},\iota)$ be a full deformation of $(\Rs,X,\Sf{M})$ over $(S,s)$. Then $\Pi_* \left(\overline{\Sf{M}}^{\smvee}\right) = \left( \Pi_* \overline{\Sf{M}} \right)^{\smvee}$.
\end{proposition}
\proof
The proof is an adaptation of Lemma~\ref{lema:ceroggsg} and Lemma~\ref{lema:dualM}. Let $\Cs{\Rsd|S}$ be relative the canonical sheaf over $\Rsd$. The sheaf $\overline{\Sf{M}} \otimes \Cs{\Rsd|S}$ is locally free, hence it is flat over $S$ and $(\overline{\Sf{M}} \otimes \Cs{\Rs|S})|_s \cong \Sf{M}\otimes \Cs{\Rs}$, then by Lemma~\ref{lema:ceroggsg} and the Cohomology and Base Change Theorem~(\cite{Har2}, Chapter III, Theorem~12.11) we have $R^1 \Pi_* (\overline{\Sf{M}}  \otimes \Cs{\Rs})=0$.

Now the proof is parallel to the proof of Lemma~\ref{lema:dualM}.
\endproof

\subsection{The correspondence for deformations at families of normal Stein surfaces with Gorenstein singularities}
\label{sec:singfam}

\begin{definition}
\label{def:enhanceddef1}
 Let $X$ be a normal Stein surface, let $\Sf{C}$  be a rank $1$ generically reduced Cohen-Macaulay $\Ss{X}$-module of dimension $1$,
 and a system of generators $(\psi_1,...,\psi_r)$ of $\Sf{C}$ as $\Ss{X}$ module. Let $M$ be a reflexive $\Ss{X}$-module of rank $r$ and 
 $(\phi_1,...,\phi_r)$ be $r$ sections.
 \begin{enumerate}
  \item A {\em deformation of} $(X,\Sf{C},(\psi_1,...,\psi_r))$ over a germ $(S,s)$ is a cuadruple 
  $(\Sf{X},\overline{\Sf{C}},(\overline{\psi}_1,...,\overline{\psi}_r),\rho)$, where $(\Sf{X},\overline{\Sf{C}},\rho)$ is a deformation of $(X,\Sf{C})$ and $(\overline{\psi}_1,...,\overline{\psi}_r)$ are sections of $\overline{\Sf{C}}$, which via the isomorphism $\rho$ restrict to $(\psi_1,...,\psi_r)$ over $s$. A deformation {\em fixing} $X$ is a deformation such that $\Sf{X}$ is the trivial deformation of $X$. 
  \item A {\em deformation of} $(X,M,(\phi_1,...,\phi_r))$ over a germ $(S,s)$ is a cuadruple 
  $(\Sf{X},\overline{M},(\overline{\phi}_1,...,\overline{\phi}_r),\iota)$, where $(\Sf{X},\overline{M},\iota)$ is a deformation of $M$ and $(\overline{\phi}_1,...,\overline{\phi}_r)$ are sections of $\overline{M}$, which via the isomorphism $\iota$ restrict to $(\phi_1,...,\phi_r)$ over $s$. A deformation {\em fixing} $X$ is a deformation such that $\Sf{X}$ is the trivial deformation of $X$. 
 \end{enumerate}
 The deformations defined above, together with the pullback operation, form two contravariant functors 
$\mathbf{Def_{X,\Sf{C}}^{(\psi_1,...,\psi_r)}}$ and $\mathbf{Def_{X,M}^{(\phi_1,...,\phi_r)}}$ 
from the category of germs of complex spaces to the category of sets. Deformations fixing the base form two contravariant functors called $\mathbf{Def_{\Sf{C}}^{(\psi_1,...,\psi_r)}}$ and $\mathbf{Def_{M}^{(\phi_1,...,\phi_r)}}$ .
 \end{definition}
 
Now we extend the correspondences at the Stein space (Theorem~\ref{th:corrsing}) to get an isomorphism of functors. First we need the following lemma, which can be found in Proposition 0.1 of~\cite{dJvS}, and also in \cite{SchE},~Proposition~2.2.

\begin{lemma}
\label{lem:extbasechange}
Let $Y$ and $(S,s)$ be a complex space and a germ of complex space. Let $\Sf{Y}$ be a flat deformation of $Y$ over $(S,s)$. Let $\overline{\Sf{F}}$ be a $\Ss{\Sf{Y}}$-module. Let $\mathfrak{m}_s$ be the 
maximal ideal at $s$. For any $\Ss{S}$-module $L$ and any index $i$ there is a natural morphism
\begin{equation}
\label{eq:basechangeext}
\phi^i(L):\Exts_{\Ss{\Sf{Y}}}^i(\overline{\Sf{F}},\Ss{\Sf{Y}})\otimes_{\Ss{S}} L\to \Exts_{\Ss{\Sf{Y}}}^i(\overline{\Sf{F}},\Ss{\Sf{Y}}\otimes_{\Ss{S}} L).
\end{equation}
The following assertions hold:
\begin{enumerate}
 \item If $\phi^i(\Ss{S,s}/\mathfrak{m}_s)$ is surjective, then $\phi^i(L)$ is an isomorphism for any $L$.
 \item Assume that $\phi^i(\Ss{S,s}/\mathfrak{m}_s)$ is surjective. Then $\phi^{i-1}(\Ss{S,s}/\mathfrak{m}_s)$ is surjective if and only if 
 $\Exts_{\Ss{\Sf{Y}}}^i(\overline{\Sf{F}},\Ss{\Sf{Y}})$ is flat over $S$.
 \item If $\Exts_{\Ss{\Sf{Y}}}^i(\Sf{F},\Ss{\Sf{Y}}\otimes \Ss{S,s}/\mathfrak{m}_s)=0$ then 
 $\Exts_{\Ss{\Sf{Y}}}^i(\Sf{F},\Ss{\Sf{Y}}\otimes_{\Ss{S}} L)=0$ for all $L$. 
\end{enumerate}
\end{lemma}

\begin{lemma}
\label{lem:isoext}
In the situation of the preceeding Lemma, if $\overline{\Sf{F}}$ is flat over $S$ then we have the isomorphism
\begin{equation}
\label{eq:isoext}
\Exts_{\Ss{\Sf{Y}}}^i(\overline{\Sf{F}},\Ss{\Sf{Y}}\otimes \Ss{S,s}/\mathfrak{m}_s)\cong \Exts_{\Ss{Y}}^i(\overline{\Sf{F}}\otimes \Ss{S,s}/\mathfrak{m}_s,\Ss{\Sf{Y}}\otimes \Ss{S,s}/\mathfrak{m}_s).
\end{equation}
\end{lemma}
\proof
It is straightforward if one compute $\Exts$ using a free resolution of $\Sf{F}$.
\endproof

Now we are able to extend the correspondence given by Theorem~\ref{th:corrsing}. From now and for this and the following section we will always assume that $X$ is a normal Stein surface with Gorenstein singularities.

\begin{theorem}
\label{th:dirXdef}
Let $X$ be a normal Stein surface with Gorenstein singularities. Let $(M,(\phi_1,..,\phi_r))$ be a reflexive $\Ss{X}$-module of rank $r$ and a collection $r$ nearly generic sections. Let  
$(\Sf{C},(\psi_1,...,\psi_r))$  be the  rank 1 generically reduced  Cohen-Macaulay $\Ss{X}$-module of dimension $1$ with the collection of generators 
obtained from $(M,(\phi_1,..,\phi_r))$ by the direct correspondence at the $X$ (see Theorem~\ref{th:corrsing}). The correspondence defined at Theorem~\ref{th:corrsing} 
extends to define an isomorphism between the functors  $\mathbf{Def_{X,M}^{(\phi_1,...,\phi_r)}}$ and $\mathbf{Def_{X,\Sf{C}}^{(\psi_1,...,\psi_r)}}$. The isomorphism restricts to an isomorphism between $\mathbf{Def_{M}^{(\phi_1,...,\phi_r)}}$ and $\mathbf{Def_{\Sf{C}}^{(\psi_1,...,\psi_r)}}$.
\end{theorem}
\proof
Let $(S,s)$ be a germ of complex
space. Let $(\Sf{X},\overline{M},(\overline{\phi}_1,...,\overline{\phi}_r),\iota)$ be a deformation of $(X,M,(\phi_1,...,\phi_r))$ over $(S,s)$. Let  $\Ss{\Sf{X}}^r\to\overline{M}$ be the morphism induced by the sections, denote its cokernel by $(\overline{\Sf{C}})'$. Then we have the exact sequence:
\begin{equation}
\label{eq:dirsingfam1}
0\to \Ss{\Sf{X}}^r\to\overline{M}\to (\overline{\Sf{C}})'\to 0.
\end{equation}
The flatness of $\overline{M}$ over $S$, and the fact that the first mapping specializes over $s$ to an injection, 
implies the flatness of $(\overline{\Sf{C}})'$ over $S$, by using the Local Criterion of Flatness.  
The specialization of the sequence to the fibre over $s$ is the exact sequence of $\Ss{X}$-modules
\begin{equation}
\label{eq:dirsingfam2}
0\to\Ss{X}^r\to\overline{M}|_s\to (\overline{\Sf{C}})'|_s\to 0,
\end{equation}
and $\iota$ induces an isomorphism between this sequence and the exact sequence 
\begin{equation}
\label{eq:dirsingfam3}
0\to \Ss{X}^r\to M\to \Sf{C}'\to 0,
\end{equation}
induced by the sections $(\phi_1,...,\phi_r)$. 

The dual of the last sequence is the sequence
\begin{equation}
\label{eq:dirsingfam4}
0\to N\to \Ss{X}^r\to \Sf{C}\to 0,
\end{equation}
where the last morphism of the sequence gives rise to the generators $(\psi_1,...,\psi_r)$ of $\Sf{C}$ (see the proof of Theorem~\ref{th:corrsing}).

Dualize the sequence (\ref{eq:dirsingfam1}) with respect to $\Ss{\Sf{X}}$ 
and obtain the exact sequence
$$0\to \Homs_{\Ss{\Sf{X}}}(\overline{N},\Ss{\Sf{X}})\to \Ss{\Sf{X}}^r\to  \Exts^1_{\Ss{\Sf{X}}}((\overline{\Sf{C}})',\Ss{\Sf{X}})\to \Exts^1_{\Ss{\Sf{X}}}(\overline{M},\Ss{\Sf{X}}).$$

Define $\overline{N}:=\Homs_{\Ss{\Sf{X}}}(\overline{M},\Ss{\Sf{X}})$.
We claim that $\Exts^1_{\Ss{\Sf{X}}}(\overline{M},\Ss{\Sf{X}})$ vanishes. Indeed, since $M$ is Cohen-Macaulay of dimension $2$ we have the vanishing $\Exts^1_{\Ss{X}}(M,\Ss{X})=0$ by Theorem~\ref{Th:Herzog}. By Lemma~\ref{lem:isoext} we have the isomorphism 
$\Exts^1_{\Ss{\Sf{X}}}(\overline{M},\Ss{X})\cong \Exts^1_{\Ss{X}}(M,\Ss{X})=0$. This vanishing, together with Lemma~\ref{lem:extbasechange} proves the claim.

As a consequence of the claim we have the exact sequence
\begin{equation}
\label{eq:dirsingfam5}
0\to \overline{N}\to \Ss{\Sf{X}}^r\to  \Exts^1_{\Ss{\Sf{X}}}((\overline{\Sf{C}})',\Ss{\Sf{X}})\to 0.
\end{equation}

We define $\overline{\Sf{C}}:=\Exts^1_{\Ss{\Sf{X}}}((\overline{\Sf{C}})',\Ss{\Sf{X}})$. We claim that the following assertions hold:
\begin{enumerate}
 \item the $\Ss{\Sf{X}}$-module $\overline{\Sf{C}}$ is flat over $S$.
 \item The specialization $\overline{N}|_s\to \Ss{X}^r$ of the first morphism of the sequence coincides with $N\to\Ss{X}^r$. 
\end{enumerate}

Assume the claim. The second assertion induces an identification of $\Ss{X}^r\to\Sf{C}$ with $\Ss{X}^r\to \overline{\Sf{C}}|_s$. Let $\rho$ denote the isomorphism 
$\Sf{C}\to\overline{\Sf{C}}|_s$. The second morphism of Sequence~(\ref{eq:dirsingfam5}) induces a collection of sections 
$(\bar\psi_1,...,\bar\psi_r)$. The first assertion shows that $(\Sf{X},\overline{\Sf{C}},(\overline{\psi}_1,...,\overline{\psi}_r),\rho)$ is a 
deformation of $(X,\Sf{C},(\psi_1,...,\psi_r))$ over $(S,s)$. 

Let us prove the claim. Since $\Sf{C}'$ is Cohen-Macaulay of dimension $1$ and $X$ has Gorenstein singularities, Theorem~\ref{Th:Herzog} implies the vanishing  $\Exts^i_{\Ss{X}}(\Sf{C}',\Ss{X})=0$ if $i\geq 2$. Then, since $(\overline{\Sf{C}})'$ is flat over $S$, using Lemma~\ref{lem:isoext} we obtain the vanishing 
$\Exts^i_{\Ss{\Sf{X}}}((\overline{\Sf{C}})',\Ss{X})=0$ if $i\geq 2$. Now we will apply Lemma~\ref{lem:extbasechange} repeatedly for 
$\Sf{F}=(\overline{\Sf{C}})'$: by the vanishing $\Exts^2_{\Ss{\Sf{X}}}((\overline{\Sf{C}})',\Ss{X})=0$ we deduce that 
$\phi^2(\Ss{S,s}/\mathfrak{m}_s)$ is surjective; the Lemma~\ref{lem:extbasechange} (1) shows that $\Exts_{\Ss{\Sf{X}}}^2((\overline{\Sf{C}})',\Ss{\Sf{X}})$ vanishes (hence it is  
flat over $S$). By Lemma~\ref{lem:extbasechange}, (2) we have that $\phi^1(\Ss{S,s}/\mathfrak{m}_s)$ is surjective. 
Lemma~\ref{lem:extbasechange}, (2) shows now that $\overline{\Sf{C}}=\Exts^1_{\Ss{\Sf{X}}}((\overline{\Sf{C}})',\Ss{\Sf{X}})$ is flat over $S$ 
if and only if $\phi^0(\Ss{S,s}/\mathfrak{m}_s)$ is surjective. This surjectivity holds since the target of this map vanishes because 
$(\overline{\Sf{C}})'$ has proper support. This shows Assertion $(1)$ of the claim. 

For Assertion $(2)$ we apply Lemma~\ref{lem:extbasechange} for $\Sf{F}=\overline{M}$. We need to show the isomorphism $\overline{N}|_s\cong N$, but 
this follows from Lemma~\ref{lem:extbasechange}, (1), if we show that $\phi^0(\Ss{S,s}/\mathfrak{m}_s)$ is surjective. By 
Lemma~\ref{lem:extbasechange}, (2), this is reduced to prove the vanishing of 
$\Exts_{\Ss{\Sf{X}}}^1(\overline{M},\Ss{\Sf{X}}\otimes_{\Ss{S}} \Ss{S,s}/\mathfrak{m}_s)$ and the flatness over $S$ of 
$\Exts_{\Ss{\Sf{X}}}^1(\overline{M},\Ss{\Sf{X}})$. The vanishing holds because $\Exts_{\Ss{\Sf{X}}}^1(\overline{M},\Ss{\Sf{X}}\otimes_{\Ss{S}} \Ss{S,s}/\mathfrak{m}_s)$ is isomorphic to $\Exts_{\Ss{X}}^1(M,\Ss{X})$ by flatness
of $\overline{M}$ and Lemma~\ref{lem:isoext}, and the second module vanishes by Theorem~\ref{Th:Herzog}. The flatness of 
$\Exts_{\Ss{\Sf{X}}}^1(\overline{M},\Ss{\Sf{X}})$ holds because this module also vanishes (apply Lemma~\ref{lem:isoext} and 
Lemma~\ref{lem:extbasechange}, (3) and (1) as before, starting from the vanishing of $\Exts^1_{\Ss{X}}(M,\Ss{X})$). 

In order to have a natural transformation from $\mathbf{Def_{X,M}^{(\phi_1,...,\phi_r)}}$ to $\mathbf{Def_{X,\Sf{C}}^{(\psi_1,...,\psi_r)}}$ we have to show
that the construction commutes with pullbacks. This follows from Lemma~\ref{lem:extbasechange}, (1), if we show the isomorphism
$\overline{\Sf{C}}|_s\cong\Exts_{\Ss{\Sf{X}}}^1((\overline{\Sf{C}})',\Ss{\Sf{X}}\otimes_{\Ss{S}} \Ss{S,s}/\mathfrak{m}_s)$. By the flatness
of $(\overline{\Sf{C}})'$ over $S$ and Lemma~\ref{lem:isoext} the second module is isomorphic to $\Sf{C}$, and then the desired isomorphism becomes
the already proven identity $\overline{\Sf{C}}|_s\cong \Sf{C}$.

Now we define the inverse natural transformation from $\mathbf{Def_{X,\Sf{C}}^{(\psi_1,...,\psi_r)}}$ to $\mathbf{Def_{X,M}^{(\phi_1,...,\phi_r)}}$.

Let $(\Sf{X},\overline{\Sf{C}},(\overline{\psi}_1,...,\overline{\psi}_r),\rho)$ be a deformation of $(X,\Sf{C},(\psi_1,...,\psi_r))$ over $(S,s)$.
Consider the exact sequence induced by the sections:
\begin{equation}
\label{eq:invsingfam1}
0\to \overline{N}\to\Ss{\Sf{X}}^r\to \overline{\Sf{C}}\to 0.
\end{equation}
The flatness of $\overline{\Sf{C}}$ over $S$ implies the flatness of  $\overline{N}$ over $S$. The specialization of the sequence to the fibre 
over $s$ is the exact sequence of $\Ss{X}$-modules
\begin{equation}
\label{eq:invsingfam2}
0\to \overline{N}|_s\to\Ss{X}^r\to \overline{\Sf{C}}|_s\to 0,
\end{equation}
and $\rho$ induces an isomorphism between this sequence and the exact sequence 
\begin{equation}
\label{eq:invsingfam3}
0\to N\to\Ss{X}^r\to \Sf{C}\to 0,
\end{equation}
induced by the generators $(\psi_1,...,\psi_r)$. The dual of the last sequence is the sequence
\begin{equation}
\label{eq:invsingfam4}
0\to\Ss{X}^r\to M\to\Exts^1_{\Ss{X}}(\Sf{C},\Ss{X})\to 0,
\end{equation}
where the sections $(\phi_1,...,\phi_r)$ are induced by the first map of the sequence (see the proof of Theorem~\ref{th:corrsing}).
 
Dualize the sequence (\ref{eq:invsingfam1}) with respect to $\Ss{\Sf{X}}$ 
and obtain the exact sequence
\begin{equation}
\label{eq:invsingfam5}
0\to \Ss{\Sf{X}}^r\to \overline{M}\to \Exts^1_{\Ss{\Sf{X}}}(\overline{\Sf{C}},\Ss{\Sf{X}})\to 0,
\end{equation}
where we define $\overline{M}:=\Homs_{\Ss{\Sf{X}}}(\overline{N},\Ss{\Sf{X}})$.  

We claim that the following assertions hold:
\begin{enumerate}
 \item the $\Ss{\Sf{X}}$-module $\overline{M}$ is flat over $S$.
 \item The specialization $ \Ss{X}^r\to \overline{M}|_s$ of the first morphism of the sequence is isomorphic to $\Ss{X}^r\to M$. 
\end{enumerate}

Assume the claim. The second assertion induces an isomorphism from $\Ss{X}^r\to M$ to $\Ss{X}^r\to \overline{M}|_s$. Let $\iota$ denote the isomorphism 
$M\to\overline{M}|_s$. The first morphism of Sequence~(\ref{eq:invsingfam5}) induces a collection of sections 
$(\bar\phi_1,...,\bar\phi_r)$. The first assertion shows that $(\Sf{X},\overline{M},(\overline{\phi}_1,...,\overline{\phi}_r),\iota)$ is a 
deformation of $(X,\Sf{C},(\phi_1,...,\phi_r))$ over $(S,s)$. 

The proof of the claim is an application of Lemmata~\ref{lem:isoext} and~\ref{lem:extbasechange} competely analogous to the proof of 
the previous two claim in this proof. We omit it here. 

In order to have a natural transformation from $\mathbf{Def_{X,\Sf{C}}^{(\psi_1,...,\psi_r)}}$ to $\mathbf{Def_{X,M}^{(\phi_1,...,\phi_r)}}$ 
we have to show that the construction commutes with pullbacks, but this is an application of Lemmata~\ref{lem:isoext} and~\ref{lem:extbasechange}, 
similar to the proof that the transformation from $\mathbf{Def_{X,M}^{(\phi_1,...,\phi_r)}}$ to $\mathbf{Def_{X,\Sf{C}}^{(\psi_1,...,\psi_r)}}$ commutes with pullbacks.

Finally we need to check that the correspondences that we have just defined are inverse to each other, but this is clear by construction.

It is obvious that the isomorphisms that we have defined restrict to an isomorphism between $\mathbf{Def_{M}^{(\phi_1,...,\phi_r)}}$ and $\mathbf{Def_{\Sf{C}}^{(\psi_1,...,\psi_r)}}$.
\endproof

\subsection{The correspondence for deformations at simultaneous resolutions of families of normal Stein Surfaces with Gorenstein singularities}
\label{sec:resfam}

\begin{definition}
\label{def:enhanceddef2}
 Let $X$ be a normal Stein surface and $\pi:\Rs\to X$ be a resolution. Let $(\Sf{A},(\psi_1,...,\psi_r))$ be a rank $1$ generically reduced Cohen-Macaulay $\Ss{\Rs}$-module of dimension $1$, and a system of generators as $\Ss{\Rs}$-module satisfying the Containment Condition. Let $\Sf{M}$ be a full $\Ss{\Rs}$-module of rank $r$ and 
 $(\phi_1,...,\phi_r)$ be $r$ nearly generic sections.
 \begin{enumerate}
 \item A {\em deformation of} $(\Rs,X,\Sf{A},(\psi_1,...,\psi_r))$ over a germ $(S,s)$ is a quintuple $(\Rsd,\Sf{X},\overline{\Sf{A}},(\overline{\psi}_1,...,\overline{\psi}_r),\rho)$, where $(\Rsd,\Sf{X},\overline{\Sf{A}},\rho)$ is a deformation of $(\Rs,X,\Sf{A})$, 
  and $(\overline{\psi}_1,...,\overline{\psi}_r)$ are sections of $\overline{\Sf{A}}$, which via the isomorphism $\rho$ restrict to 
  $(\psi_1,...,\psi_r)$ over $s$.
 \item A {\em specialty defect constant deformation of} $(\Sf{A},(\psi_1,...,\psi_r))$ over a germ $(S,s)$ is a deformation $(\Rsd,\Sf{X},\overline{\Sf{A}},(\overline{\psi}_1,...,\overline{\psi}_r),\rho)$ such that the cokernel $\overline{\Sf{D}}$ of the natural mapping $$\Pi_*\Ss{\Rsd}^r\to \Pi_*\overline{\Sf{A}}$$ induced by the sections $(\overline{\psi}_1,...,\overline{\psi}_r)$ is flat over $S$.
 \item  A {\em (full) deformation of} $(\Rs,X,\Sf{M},(\phi_1,...,\phi_r))$ over a germ $(S,s)$ is a quintuple 
  $$(\Rsd,\Sf{X},\overline{\Sf{M}},(\overline{\phi}_1,...,\overline{\phi}_r),\iota),$$ where $(\Rsd,\Sf{X},\overline{\Sf{M}},\iota)$ is a (full) deformation of $(\Rs,X,\Sf{M})$ and $(\overline{\phi}_1,...,\overline{\phi}_r)$ are sections of $\overline{\Sf{M}}$, which via the isomorphism $\iota$ restrict to $(\phi_1,...,\phi_r)$ over $s$. 
 \end{enumerate}
 The deformations defined above give rise to functors $\mathbf{Def_{\Rs,X,\Sf{A}}^{(\psi_1,...,\psi_r)}}$,  $\mathbf{SDCDef_{\Rs,X,\Sf{A}}^{(\psi_1,...,\psi_r)}}$, $\mathbf{Def_{\Rs,X,\Sf{M}}^{(\phi_1,...,\phi_r)}}$ and $\mathbf{FullDef_{\Rs,X,\Sf{M}}^{(\phi_1,...,\phi_r)}}$. The obvious restricted deformation functors fixing $\Rs$ or $\Rs$ and $X$ are denoted by $\mathbf{Def_{X,\Sf{A}}^{(\psi_1,...,\psi_r)}}$,  $\mathbf{SDCDef_{X,\Sf{A}}^{(\psi_1,...,\psi_r)}}$, $\mathbf{Def_{X,\Sf{M}}^{(\phi_1,...,\phi_r)}}$, $\mathbf{FullDef_{X,\Sf{M}}^{(\phi_1,...,\phi_r)}}$, $\mathbf{Def_{\Sf{A}}^{(\psi_1,...,\psi_r)}}$,  $\mathbf{SDCDef_{\Sf{A}}^{(\psi_1,...,\psi_r)}}$, $\mathbf{Def_{\Sf{M}}^{(\phi_1,...,\phi_r)}}$ and $\mathbf{FullDef_{\Sf{M}}^{(\phi_1,...,\phi_r)}}$.
 \end{definition}

\begin{lemma}
The assignements defined in the previous definition are in fact contravariant functors.
\end{lemma}
\proof
The only non-trivial assertion is for $\mathbf{SDCDef_{\Rs,X,\Sf{A}}^{(\psi_1,...,\psi_r)}}$: we need to prove the preservation of the flatness of the cokernel $\Sf{D}$ under pullback. This holds because the formation of cokernels commutes with pullbacks and flatness is preserved by pullbacks. 
\endproof
 
\begin{lemma}
\label{lem:defsarcos}
In the setting of the previous definition, suppose that $(\psi_1,...,\psi_r)$ generate $\Sf{A}$ as a $\Ss{X}$-module (that is, they generate $\pi_*\Sf{A}$). Then the functors $\mathbf{Def_{\Rs,X,\Sf{A}}^{(\psi_1,...,\psi_r)}}$ and $\mathbf{SDCDef_{\Rs,X,\Sf{A}}^{(\psi_1,...,\psi_r)}}$ coincide.
\end{lemma}
\proof
The specialization of the cokernel of $\Pi_*:\Ss{\Rsd}^r\to \pi_*\overline{\Sf{X}}$ to the fibre over $s$ vanishes, and hence the cokernel vanishes as well by Nakayama Lemma.
\endproof

\begin{theorem}
\label{th:dirresdef}
Let $X$ be a normal Stein surface with Gorenstein singularities. Let $(\Sf{A},(\psi_1,...,\psi_r))$ be a Cohen-Macaulay $\Ss{\Rs}$-module of dimension $1$, and a system of generators as $\Ss{\Rs}$-module satisfying the Containment Condition, let $(\Sf{M},(\phi_1,...,\phi_r))$ be a full $\Ss{\Rs}$-module of rank $r$ with $r$ nearly generic sections. Suppose that the pairs $(\Sf{A},(\psi_1,...,\psi_r))$ and $(\Sf{M},(\phi_1,...,\phi_r))$ are related by the correspondence of Theorem~\ref{th:corres}. There is an
isomorphism of functors between $\mathbf{FullDef_{\Rs,X,\Sf{M}}^{(\phi_1,...,\phi_r)}}$ and $\mathbf{SDCDef_{\Rs,X,\Sf{A}}^{(\psi_1,...,\psi_r)}}$. The isomorphism restricts to isomorphisms between $\mathbf{FullDef_{X,\Sf{M}}^{(\phi_1,...,\phi_r)}}$ and $\mathbf{SDCDef_{X,\Sf{A}}^{(\psi_1,...,\psi_r)}}$, and between $\mathbf{FullDef_{\Sf{M}}^{(\phi_1,...,\phi_r)}}$ and $\mathbf{SDCDef_{\Sf{A}}^{(\psi_1,...,\psi_r)}}$.
\end{theorem}
\proof
The restriction statements are obvious after finding an isomorphism between $\mathbf{FullDef_{\Rs,X,\Sf{M}}^{(\phi_1,...,\phi_r)}}$ and $\mathbf{SDCDef_{\Rs,X,\Sf{A}}^{(\psi_1,...,\psi_r)}}$. 

The proof of this isomorphism has two parts. In the first we find an isomorphism 
between $\mathbf{Def_{\Rs,X,\Sf{M}}^{(\phi_1,...,\phi_r)}}$ and $\mathbf{Def_{\Rs,X,\Sf{A}}^{(\psi_1,...,\psi_r)}}$. In the second part we prove that, under the defined isomorphism full deformations
correspond to specialty defect constant deformations.

\textsc{Part 1}. The first part runs parallel to the proof of Theorem~\ref{th:dirXdef}, so we skip many details and include only what is needed to define the isomorphism and to set up the notation for the proof of Part 2. 

Let $(S,s)$ be a germ of complex space. Let $(\Rsd,\Sf{X},\overline{\Sf{M}},(\overline{\phi}_1,...,\overline{\phi}_r),\iota)$ be an element of $\mathbf{Def_{\Rs,X,\Sf{M}}^{(\phi_1,...,\phi_r)}}(S,s)$. 

Consider the exact sequence induced by the sections:
\begin{equation}
\label{eq:dirresfam1}
0\to \Ss{\Rsd}^r\to\overline{\Sf{M}}\to (\overline{\Sf{A}})'\to 0.
\end{equation}

The flatness of $\overline{\Sf{M}}$ over $S$, and the fact that the first mapping specializes over $s$ to an injection, 
implies the flatness of $(\overline{\Sf{A}})'$ over $S$, by using the local criterion of flatness.  
The specialization of the sequence to the fibre 
over $s$ is the exact sequence of $\Ss{\Rs}$-modules
\begin{equation}
\label{eq:dirresfam2}
0\to\Ss{\Rs}^r\to\overline{\Sf{M}}|_s\to (\overline{\Sf{A}})'|_s\to 0,
\end{equation}
and $\iota$ induces an isomorphism between this sequence and the exact sequence 
\begin{equation}
\label{eq:dirresfam3}
0\to \Ss{\Rs}^r\to \Sf{M}\to (\Sf{A})'\to 0,
\end{equation}
induced by the sections $(\phi_1,...,\phi_r)$. 

The dual of this last sequence is the sequence
\begin{equation}
\label{eq:dirresfam4}
0\to \Sf{N}\to \Ss{\Rs}^r\to \Sf{A}\to 0,
\end{equation}
where the last morphism of the sequence gives rise to the generators $(\psi_1,...,\psi_r)$ of $\Sf{A}$ as a $\Ss{\Rs}$-module (see the proof 
of Theorem~\ref{th:corres}).

Dualize the sequence (\ref{eq:dirresfam1}) with respect to $\Ss{\Rsd}$ 
and obtain the exact sequence
\begin{equation}
\label{eq:dirresfam5}
0\to \overline{\Sf{N}}\to \Ss{\Rsd}^r\to \overline{\Sf{A}}\to 0,
\end{equation}
where $\overline{\Sf{N}}:=\Homs_{\Ss{\Rsd}}(\overline{\Sf{M}},\Ss{\Rsd})$ and $\overline{\Sf{A}}:=\Exts^1_{\Ss{\Rsd}}((\overline{\Sf{A}})',\Ss{\Rsd})$. The following two assertions are proved like
the corresponding ones in the proof of Theorem~\ref{th:dirXdef}. 
\begin{enumerate}
 \item the $\Ss{\Rsd}$-module $\overline{\Sf{A}}$ is flat over $S$.
 \item The specialization $\overline{\Sf{N}}|_s\to \Ss{X}^r$ of the first morphism of the sequence is isomorphic to $\Sf{N}\to\Ss{X}^r$. 
\end{enumerate}

The second assertion induces an isomorphism from $\Ss{\Rs}^r\to\Sf{A}$ to $\Ss{\Rs}^r\to \overline{\Sf{A}}|_s$. Let $\rho$ denote the isomorphism 
$\Sf{A}\to\overline{\Sf{A}}|_s$. The second morphism of Sequence~(\ref{eq:dirresfam5}) induces a collection of sections 
$(\bar\psi_1,...,\bar\psi_r)$. We have that $(\Rsd,\Sf{X},\overline{\Sf{A}},(\overline{\psi}_1,...,\overline{\psi}_r),\rho)$ is a deformation of $(\Rs,X,\Sf{A},(\psi_1,...,\psi_r))$ over $(S,s)$ by the first assertion. So, we have defined map 
$\mathbf{Def_{\Rs,X,\Sf{M}}^{(\phi_1,...,\phi_r)}}(S,s)\to\mathbf{Def_{\Rs,X,\Sf{A}}^{(\psi_1,...,\psi_r)}}(S,s)$.
In order to have a natural transformation of functors from 
$\mathbf{Def_{\Rs,X,\Sf{M}}^{(\phi_1,...,\phi_r)}}$ to $\mathbf{Def_{\Rs,X,\Sf{A}}^{(\psi_1,...,\psi_r)}}$ we need to show that the transformation
which we have defined commutes with pullbacks. This is analogous to the corresponding statement in the proof of Theorem~\ref{th:dirXdef}.

Now we define the inverse natural transformation from $\mathbf{Def_{\Rs,X,\Sf{A}}^{(\psi_1,...,\psi_r)}}$ to 
$\mathbf{Def_{\Rs,X,\Sf{M}}^{(\phi_1,...,\phi_r)}}$. Let $(\Rsd,\Sf{X},\overline{\Sf{A}},(\overline{\psi}_1,...,\overline{\psi}_r),\rho)$ be a deformation of $(\Rs,X,\Sf{A},(\psi_1,...,\psi_r))$ over $(S,s)$.
Consider the exact sequence induced by the sections:
\begin{equation}
\label{eq:invresfam1}
0\to \overline{\Sf{N}}\to\Ss{\Rsd}^r\to \overline{\Sf{A}}\to 0.
\end{equation}
The flatness of $\overline{\Sf{A}}$ over $S$ implies the flatness of  $\overline{\Sf{N}}$ over $S$. 
The specialization of the sequence to the fibre 
over $s$ is the exact sequence of $\Ss{\Rs}$-modules
\begin{equation}
\label{eq:invresfam2}
0\to \overline{\Sf{N}}|_s\to\Ss{\Rs}^r\to \overline{\Sf{A}}|_s\to 0,
\end{equation}
and $\rho$ induces an isomorphism between this sequence and the exact sequence 
\begin{equation}
\label{eq:invresfam3}
0\to \Sf{N}\to\Ss{X}^r\to \Sf{A}\to 0.
\end{equation}
induced by the generators $(\psi_1,...,\psi_r)$. The dual of this last sequence is the sequence
\begin{equation}
\label{eq:invresfam4}
0\to\Ss{\Rs}^r\to \Sf{M}\to\Exts^1_{\Ss{\Rs}}(\Sf{A},\Ss{\Rs})\to 0,
\end{equation}
where the sections $(\phi_1,...,\phi_r)$ are induced by the first map of the sequence (see the proof of Theorem~\ref{th:corres}).
 
Dualize the sequence (\ref{eq:invresfam1}) with respect to $\Ss{\Rsd}$ 
and obtain the exact sequence
\begin{equation}
\label{eq:invresfam5}
0\to \Ss{\Rsd}^r\to \overline{\Sf{M}}\to \Exts^1_{\Ss{\Rsd}}(\overline{\Sf{A}},\Ss{\Rsd})\to 0,
\end{equation}
where $\overline{\Sf{M}}:=\Homs_{\Ss{\Rsd}}(\overline{\Sf{N}},\Ss{\Rsd})$.

The following two assertions are proved like the corresponding ones in the proof of Theorem~\ref{th:dirXdef}:

\begin{enumerate}
 \item the $\Ss{\Rsd}$-module $\overline{\Sf{M}}$ is flat over $S$.
 \item The specialization $ \Ss{\Rs}^r\to \overline{\Sf{M}}|_s$ of the first morphism of the sequence is isomorphic to $\Ss{\Rs}^r\to\Sf{M}$. 
\end{enumerate}

The second assertion induces an isomorphism from $\Ss{\Rs}^r\to \Sf{M}$ to $\Ss{\Rs}^r\to \overline{\Sf{M}}|_s$. Let $\iota$ denote the isomorphism 
$\Sf{M}\to\overline{\Sf{M}}|_s$. The first morphism of Sequence~(\ref{eq:invresfam5}) induces a collection of sections 
$(\bar\phi_1,...,\bar\phi_r)$. The first assertion shows that $(\Rsd,\Sf{X},\overline{\Sf{M}},(\overline{\phi}_1,...,\overline{\phi}_r),\iota)$ is a deformation of $(\Rs,X,\Sf{M},(\phi_1,...,\phi_r))$ over $(S,s)$.

\textsc{Part 2}. Let $(\Rsd,\Sf{X},\overline{\Sf{M}},(\overline{\phi}_1,...,\overline{\phi}_r),\iota)$ be an element of $\mathbf{FullDef_{\Rs,X,\Sf{M}}^{(\phi_1,...,\phi_r)}}(S,s)$. We have to show that the quintuple $(\Rsd,\Sf{X},\overline{\Sf{A}},(\overline{\psi}_1,...,\overline{\psi}_r),\rho)$ associated to it in Part 1 is a specialty constant deformation. In this part the Gorenstein condition plays an important role.

For this we have to show that the cokernel $\overline{\Sf{D}}$ of the map                     
$\Pi_*\Ss{\Rsd}^r\to \Pi_*\overline{\Sf{A}}$ is flat over $S$. Denote by $\overline{\Sf{C}}$ the image of the same map. 
We have the exact sequences
\begin{equation}
\label{eq:Cbarra}
0\to\Pi_*\overline{\Sf{N}}\to \Ss{\Sf{X}}^r\to \overline{\Sf{C}}\to 0,
\end{equation}
\begin{equation}
\label{eq:Dbarra}
0\to \overline{\Sf{C}}\to\Pi_*\overline{\Sf{A}}\to\overline{\Sf{D}}\to 0.
\end{equation}

Now we specialize Sequence~(\ref{eq:Dbarra}) at $s$. If we denote the maximal ideal of $s$ by $\mathfrak{m}_s$, and use the flatness of 
$\overline{\Sf{A}}$ over $S$, we have the exact sequence
$$0\to Tor_1^{\Ss{S}}(\overline{\Sf{D}},\Ss{S}/\mathfrak{m}_s)\to (\overline{\Sf{C}})|_s\to(\Pi_*\overline{\Sf{A}})|_s\to (\overline{\Sf{D}})|_s\to 0.$$ 
We claim that the second morphism of the sequence can be identified with the injective morphism $\Sf{C}\to\pi_*\Sf{A}$, where $\Sf{C}$ is the $\Ss{X}$-module spanned by $(\psi_1,...,\psi_r)$. If the claim is true we deduce the vanishing of $Tor_1^{\Ss{S}}(\overline{\Sf{D}},\Ss{S}/\mathfrak{m}_s)$. 
This implies the flatness of $\overline{\Sf{D}}$ using the local criterion of flatness. 

In order to prove the claim we produce natural identifications 
$(\Pi_*\overline{\Sf{A}})|_s\cong \pi_*\Sf{A}$ and $(\overline{\Sf{C}})|_s\cong \Sf{C}$, which show that the morphism 
$(\overline{\Sf{C}})|_s\to(\Pi_*\overline{\Sf{A}})|_s$ is injective, yielding the desired vanishing.

For the first identification notice that $\Pi_*\overline{\Sf{A}}$ coincides with $\overline{\Sf{A}}$ with the $\Ss{\Sf{X}}$-module structure obtained by restriction of scalars; similarly $\pi_*\Sf{A}$ coincides with $\Sf{A}$ with the $\Ss{X}$-module structure obtained by restriction of scalars. Then the needed identification is the isomorphism $\rho$ defined above.

The second identification is a bit more involved (at the moment we do not even know the flatness of $\overline{\Sf{C}}$ over $S$). We proceed like in Section~\ref{sec:cohomology} to obtain a diagram similar to~(\ref{diagram:comparacionM}).  
Dualizing sequence~(\ref{eq:Cbarra}) we obtain the sequence   
$$0 \to \Ss{\Sf{X}}^r \to \Homs_{\Ss{\Sf{X}}}(\Pi_*\overline{\Sf{N}},\Ss{\Sf{X}}) \to \Exts_{\Ss{\Sf{X}}}^1 \left(\overline{\Sf{C}}, \Ss{\Sf{X}} \right) \to 0.$$
Applying $\Pi_*$ to Sequence~(\ref{eq:dirresfam1}) and using the isomorphism $(\overline{\Sf{A}})'\cong\Exts^1_{\Ss{\Rsd}}(\overline{\Sf{A}},\Ss{\Rsd})$, which comes from the fact that the natural transformations are inverse to each other, we obtain the exact sequence
$$0 \to \Ss{\Sf{X}}^r \to \Pi_*\overline{\Sf{M}} \to \Pi_* \Exts_{\Ss{\Rsd}}^1 \left(\overline{\Sf{A}}, \Ss{\Rsd} \right) \to R^1 \Pi_* \Ss{\Rsd}^r \to R^1 \Pi_* \overline{\Sf{M}} \to 0.$$
We have the chain of equalities
$$\Homs_{\Ss{\Sf{X}}}(\Pi_*\overline{\Sf{N}},\Ss{\Sf{X}})=\Homs_{\Ss{\Sf{X}}}(\Pi_*\Homs_{\Ss{\Rsd}}(\overline{\Sf{M}},\Ss{\Rsd}),\Ss{\Sf{X}})=$$
$$=\Homs_{\Ss{\Sf{X}}}(\Homs_{\Ss{\Sf{X}}}(\Pi_*\overline{\Sf{M}},\Ss{\Sf{X}}),\Ss{\Sf{X}})=\Pi_*\overline{\Sf{M}}.$$
The first is because $\Homs_{\Ss{\Rsd}}(\overline{\Sf{M}},\Ss{\Rsd}) \cong \overline{\Sf{N}}$, the second is a consequence of Proposition~\ref{prop:dualMdeformado}, and the third is a consequence of  Proposition~\ref{prop:naturaltrans} and the fact that flat deformations of reflexive modules are reflexive.  

The last two exact sequences, together with the previous identifications yields the following commutative diagram:

\begin{equation}\label{diagram:comparacionMbarra}
\begin{tikzpicture}
  \matrix (m)[matrix of math nodes,
    nodes in empty cells,text height=1.5ex, text depth=0.25ex,
    column sep=1.5em,row sep=2em] {
		0 & \Ss{\Sf{X}}^r & \Homs_{\Ss{\Sf{X}}}(\Pi_*\overline{\Sf{N}},\Ss{\Sf{X}}) & \Exts_{\Ss{\Sf{X}}}^1 \left(\overline{\Sf{C}}, \Ss{\Sf{X}} \right) & 0\\
		0 & \Ss{\Sf{X}}^r & \Pi_*\overline{\Sf{M}} & \Pi_* \Exts_{\Ss{\Rsd}}^1 \left(\overline{\Sf{A}}, \Ss{\Rsd} \right) & R^1 \Pi_* \Ss{\Rsd}^r & R^1 \Pi_* \overline{\Sf{M}} & 0\\};
\draw[-stealth] (m-1-1) -- (m-1-2);
\draw[-stealth] (m-1-2) -- (m-1-3);
\draw[-stealth] (m-1-3) edge node [auto] {$h$} (m-1-4);
\draw[-stealth] (m-1-4) -- (m-1-5);
\draw[-stealth] (m-2-1) -- (m-2-2);
\draw[-stealth] (m-2-2) -- (m-2-3);
\draw[-stealth] (m-2-3) edge node [auto] {$h$} (m-2-4);
\draw[-stealth] (m-2-4) -- (m-2-5);
\draw[-stealth] (m-2-5) -- (m-2-6);
\draw[-stealth] (m-2-6) -- (m-2-7);
\draw[-stealth] (m-1-2) edge node [auto] {$Id$} (m-2-2);
\draw[-stealth] (m-1-3) edge node [auto] {$Id$} (m-2-3);
\draw[-stealth] (m-1-4)  edge node [auto] {$\theta$} (m-2-4);
\end{tikzpicture}
\end{equation}
where $\theta$ is the only map making the diagram commutative.

Since we have the isomorphism $\Homs_{\Ss{\Rsd}}(\overline{\Sf{N}},\Ss{\Rsd}) \cong \overline{\Sf{M}}$, dualizing the sequence (\ref{eq:dirresfam5}), and using the exact sequence (\ref{eq:dirresfam1}) we obtain that $\Exts_{\Ss{\Rsd}}^1 \left(\overline{\Sf{A}}, \Ss{\Rsd} \right)$ is isomorphic to $(\overline{\Sf{A}}')$, and, as 
$\Ss{S}$-module, $\Pi_*(\overline{\Sf{A}}')$ is equal to $(\overline{\Sf{A}}')$, which is flat over $S$. Then, 
the exactness of the lower row of the diagram and the flatness of $R^1 \Pi_* \overline{\Sf{M}}$ imply that 
$\Exts_{\Ss{\Sf{X}}}^1 \left(\overline{\Sf{C}}, \Ss{\Sf{X}} \right)$ is flat over $S$.

As a consequence, specializing the first row of diagram~(\ref{diagram:comparacionMbarra}) over $s$ we obtain the exact sequence
$$0 \to \Ss{X}^r \to M \to \left(\Exts_{\Ss{\Sf{X}}}^1 \left(\overline{\Sf{C}}, \Ss{\Sf{X}} \right)\right)|_s \to 0.$$
Comparing with the first row of diagram~(\ref{diagram:comparacionM}) we obtain the isomorphism 
\begin{equation}
\label{eq:espext}
\left(\Exts_{\Ss{\Sf{X}}}^1 \left(\overline{\Sf{C}}, \Ss{\Sf{X}} \right)\right)|_s\cong \Exts_{\Ss{X}}^1(\Sf{C},\Ss{X}).
\end{equation}

We have the chain of isomorphisms:
$$\Exts_{\Ss{\Sf{X}}}^2(\Exts_{\Ss{\Sf{X}}}^1 \left(\overline{\Sf{C}}, \Ss{\Sf{X}} \right), \Ss{X})\cong 
\Exts_{\Ss{X}}^2(\left(\Exts_{\Ss{\Sf{X}}}^1 \left(\overline{\Sf{C}}, \Ss{\Sf{X}} \right)\right)|_{s=0}, \Ss{X})\cong$$ 
$$\cong\Exts_{\Ss{X}}^2(\Exts_{\Ss{X}}^1 \left(\Sf{C}, \Ss{X} \right),\Ss{X})=0.$$
The first isomorphism is due to the flatness of $\Exts_{\Ss{\Sf{X}}}^1 \left(\overline{\Sf{C}}, \Ss{\Sf{X}} \right)$ and Lemma~\ref{lem:isoext},
the second isomorphism is because of Equation~(\ref{eq:espext}), and the vanishing is due to the fact that $\Sf{C}$ is Cohen-Macaulay of dimension $1$, Theorem~\ref{Th:Herzog} and the Gorenstein condition.  

The last vanishing shows, applying Lemma~\ref{lem:extbasechange} (3), (2), and (1) for 
$\Sf{F}=\Exts_{\Ss{\Sf{X}}}^1 \left(\overline{\Sf{C}}, \Ss{\Sf{X}} \right)$, the isomorphism 
\begin{equation}
\label{eq:auxi1}
\Exts_{\Ss{\Sf{X}}}^1(\Exts_{\Ss{\Sf{X}}}^1 \left(\overline{\Sf{C}}, \Ss{\Sf{X}} \right), \Ss{\Sf{X}})|_s\cong
\Exts_{\Ss{\Sf{X}}}^1(\Exts_{\Ss{\Sf{X}}}^1 \left(\overline{\Sf{C}}, \Ss{\Sf{X}} \right), \Ss{X}).
\end{equation}

Dualizing the first row of diagram~(\ref{diagram:comparacionMbarra}) we obtain the exact sequence 
$$0\to  \Homs_{\Ss{\Sf{X}}}(\Homs_{\Ss{\Sf{X}}}(\Pi_*\overline{\Sf{N}},\Ss{\Sf{X}}),\Ss{\Sf{X}})\to\Ss{\Sf{X}}^r\to \Exts_{\Ss{\Sf{X}}}^1(\Exts_{\Ss{\Sf{X}}}^1 \left(\overline{\Sf{C}}, \Ss{\Sf{X}} \right), \Ss{\Sf{X}})\to$$
$$\to \Exts^1_{\Ss{\Sf{X}}}(\Homs_{\Ss{\Sf{X}}}(\Pi_*\overline{\Sf{N}},\Ss{\Sf{X}}),\Ss{\Sf{X}}).$$
We have the chain of equalities 
$$\Homs_{\Ss{\Sf{X}}}(\Homs_{\Ss{\Sf{X}}}(\Pi_*\overline{\Sf{N}},\Ss{\Sf{X}}),\Ss{\Sf{X}})=$$
$$=\Homs_{\Ss{\Sf{X}}}(\Homs_{\Ss{\Sf{X}}}(\Homs_{\Ss{\Sf{X}}}(\Pi_*\overline{\Sf{M}},\Ss{\Sf{X}}),\Ss{\Sf{X}}),\Ss{\Sf{X}})=$$
$$=\Homs_{\Ss{\Sf{X}}}(\Pi_*\overline{\Sf{M}},\Ss{\Sf{X}})=\Pi_*\overline{\Sf{N}}.$$
The first and third equalities are applications of Proposition~\ref{prop:dualMdeformado} and the second is because a triple dual coincides with a single dual. We also have 
$$\Exts^1_{\Ss{\Sf{X}}}(\Homs_{\Ss{\Sf{X}}}(\Pi_*\overline{\Sf{N}},\Ss{\Sf{X}}),\Ss{\Sf{X}})=
\Exts^1_{\Ss{\Sf{X}}}(\Pi_*\overline{\Sf{M}},\Ss{\Sf{X}})=0.$$
The first equality has been shown in the chain of equalities prior to diagram~\eqref{diagram:comparacionMbarra}, and the vanishing follows by an application of Lemmata~\ref{lem:isoext} and~\ref{lem:extbasechange}, the fact that $\Pi_*\overline{\Sf{M}}$ is a flat deformation of a reflexive $\Ss{X}$-module and the Gorenstein condition. After these identifications the last exact sequence becomes:
$$0\to \Pi_*\overline{\Sf{N}}\to\Ss{\Sf{X}}^r\to \Exts_{\Ss{\Sf{X}}}^1(\Exts_{\Ss{\Sf{X}}}^1 \left(\overline{\Sf{C}}, \Ss{\Sf{X}} \right), \Ss{\Sf{X}})\to 0,$$
and this gives, by comparison with Exact Sequence~(\ref{eq:Cbarra}) the isomorphism 
\begin{equation}
\label{eq:nuevoC}
\overline{\Sf{C}}\cong \Exts_{\Ss{\Sf{X}}}^1(\Exts_{\Ss{\Sf{X}}}^1 \left(\overline{\Sf{C}}, \Ss{\Sf{X}} \right), \Ss{\Sf{X}}).
\end{equation}

The following chain of isomorphisms gives the needed identification:
$$(\overline{\Sf{C}})|_{s}\cong 
\Exts_{\Ss{\Sf{X}}}^1(\Exts_{\Ss{\Sf{X}}}^1 \left(\overline{\Sf{C}}, \Ss{\Sf{X}} \right), \Ss{\Sf{X}})|_s\cong
\Exts_{\Ss{\Sf{X}}}^1(\Exts_{\Ss{\Sf{X}}}^1 \left(\overline{\Sf{C}}, \Ss{\Sf{X}} \right), \Ss{X})\cong$$
$$\cong\Exts_{\Ss{X}}^1((\Exts_{\Ss{\Sf{X}}}^1 \left(\overline{\Sf{C}}, \Ss{\Sf{X}} \right))|_s, \Ss{X})\cong
\Exts_{\Ss{X}}^1(\Exts_{\Ss{X}}^1 \left(\Sf{C}, \Ss{X} \right), \Ss{X})\cong
\Sf{C}.$$ 
The first isomorphism is by (\ref{eq:nuevoC}); the second by (\ref{eq:auxi1}); the third by flatness of 
$\Exts_{\Ss{\Sf{X}}}^1 \left(\overline{\Sf{C}}, \Ss{\Sf{X}} \right)$ and Lemma~\ref{lem:isoext}; the fourth by (\ref{eq:espext}); the fifth is 
by Theorem~\ref{Th:Herzog}, using that $\Sf{C}$ is Cohen-Macaulay of dimension $1$. 

We have proven that the pair $(\Rsd,\Sf{X},\overline{\Sf{A}},(\overline{\psi}_1,...,\overline{\psi}_r),\rho)$ is a specialty defect constant
deformation of the pair $(\Rs,X,\Sf{A},(\psi_1,...,\psi_r))$ over $(S,s)$. 

In order to finish Part 2 of the proof we let $(\Rsd,\Sf{X},\overline{\Sf{A}},(\overline{\psi}_1,...,\overline{\psi}_r),\rho)$ be a specialty constant deformation of $(\Rs,X,\Sf{A},(\psi_1,...,\psi_r))$ over $(S,s)$, and consider $(\Rsd,\Sf{X},\overline{\Sf{M}},(\overline{\phi}_1,...,\overline{\phi}_r),\iota)$, the deformation assigned by the isomorphism of functors from $\mathbf{Def_{\Rs,X,\Sf{A}}^{(\psi_1,...,\psi_r)}}$ to 
$\mathbf{Def_{\Rs,X,\Sf{M}}^{(\phi_1,...,\phi_r)}}$. We have to prove that $(\Rsd,\Sf{X},\overline{\Sf{M}},(\overline{\phi}_1,...,\overline{\phi}_r),\iota)$ is an element of $\mathbf{FullDef_{\Rs,X,\Sf{M}}^{(\phi_1,...,\phi_r)}}(S,s)$. For this we need to show 
that $R^1\Pi_*\overline{\Sf{M}}$ is flat over $S$. 

As before denote by $\overline{\Sf{C}}$ the image of $\Pi_*\Ss{\Sf{X}}\to \Pi_*\overline{\Sf{A}}$ and consider the exact sequences~(\ref{eq:Cbarra})
and~(\ref{eq:Dbarra}). Denote by $\Sf{D}$ the image of $\pi_*\Ss{X}\to \pi_*\Sf{A}$, and consider the exact sequences
\begin{equation}
\label{eq:C}
0\to\pi_*\Sf{N}\to \Ss{X}^r\to \Sf{C}\to 0,
\end{equation}
\begin{equation}
\label{eq:D}
0\to \Sf{C}\to\pi_*\Sf{A}\to\Sf{D}\to 0.
\end{equation}

Sequence~(\ref{eq:Dbarra}), and the flatness of $\Pi_*\overline{\Sf{A}}$ and $\overline{\Sf{D}}$ implies the flatness of $\overline{\Sf{C}}$ over $(S,s)$. 
Observe that $\Pi_*\overline{\Sf{A}}$ and $\overline{\Sf{D}}$ specialize over $s$ to $\pi_*\Sf{A}$ and $\Sf{D}$. Consequently, specializing 
Sequence~(\ref{eq:Dbarra}) over $s$ and comparing with Sequence~(\ref{eq:D}) we conclude the isomorphism 
\begin{equation}
\label{eq:Cspbien}
\overline{\Sf{C}}|_s\cong\Sf{C}.
\end{equation}

We have the chain of isomorphisms
$$\Exts^2_{\Ss{\Sf{X}}}(\overline{\Sf{C}},\Ss{X})\cong\Exts^2_{\Ss{X}}(\overline{\Sf{C}}|_s,\Ss{X})\cong \Exts^2_{\Ss{X}}(\Sf{C},\Ss{X})=0.$$

The first isomorphism is by flatness of $\overline{\Sf{C}}$ and Lemma~\ref{lem:isoext}, the second follows from~(\ref{eq:Cspbien}), and the vanishing
by Theorem~\ref{Th:Herzog}, the fact that $\Sf{C}$ is Cohen-Macaulay of dimension $1$ and the Gorenstein condition. Then, by Lemma~\ref{lem:extbasechange} we have the isomorphism 
\begin{equation}
\label{eq:identcruc}
\Exts_{\Ss{\Sf{X}}}^1 \left(\overline{\Sf{C}}, \Ss{\Sf{X}} \right)|_s\cong \Exts_{\Ss{X}}^1 \left(\Sf{C}, \Ss{X} \right).
\end{equation}

Observe also the vanishing $\Homs_{\Ss{\Sf{X}}}(\overline{\Sf{C}},\Ss{X})=0.$
Using the last isomorphism and the vanishing, Lemma~\ref{lem:extbasechange} (2) implies that 
$\Exts_{\Ss{\Sf{X}}}^1 \left(\overline{\Sf{C}}, \Ss{\Sf{X}} \right)$ is flat over $S$. 

Specializing the first row of Diagram~(\ref{diagram:comparacionMbarra}) over $s$, we obtain the exact sequence
$$0\to \Ss{X}^r\to \Pi_*\overline{\Sf{M}}|_s\to \Exts_{\Ss{\Sf{X}}}^1 \left(\overline{\Sf{C}}, \Ss{\Sf{X}} \right)|_s\to 0.$$
Using the identification~(\ref{eq:identcruc}) and comparing with the sequence
$$0\to\Ss{X}^r\to\pi_*\Sf{M}\to  \Exts_{\Ss{X}}^1 \left(\Sf{C}, \Ss{X} \right)       \to 0,$$
obtained by dualizing Sequence~(\ref{eq:C}), we deduce the isomorphism $\Pi_*\overline{\Sf{M}}|_s\cong \pi_*\Sf{M}$. Hence Condition 2 of 
Lemma~\ref{lem:h0h1} holds and this concludes the 
proof.
\endproof

The previous theorem, together with Theorem~\ref{th:corres} gives the following set of corollaries:

\begin{corollary}
Assume that $(\Sf{A},(\psi_1,...,\psi_r))$ is a Cohen-Macaulay $\Ss{\Rs}$-module of dimension $1$ with a collection of generators as a 
$\Ss{\Rs}$-module satisfying the Containment Condition.
Let $(\Rsd,\Sf{X},\overline{\Sf{A}},(\overline{\psi}_1,...,\overline{\psi}_r),\rho)$ be a specialty defect constant deformation of
$(\Sf{A},(\psi_1,...,\psi_r))$ over a germ $(S,s)$.  Then for any $s'\in S$ the Cohen-Macaulay $\Ss{\Rsd_{s'}}$-module with generators 
$(\overline{\Sf{A}}|_{s'},(\overline{\psi}_1|_{s'},...,\overline{\psi}_r|_{s'}))$ satisfies the Containment Condition.  
\end{corollary}

\begin{remark}
Assume that $(\Sf{A},(\psi_1,...,\psi_r))$ is a Cohen-Macaulay $\Ss{\Rs}$-module of dimension $1$ with a collection of generators as a 
$\Ss{\Rs}$-module satisfying the Containment Condition. Let $(\Rsd,\Sf{X},\overline{\Sf{A}},(\overline{\psi}_1,...,\overline{\psi}_r),\rho)$ be a specialty defect constant deformation of
$(\Sf{A},(\psi_1,...,\psi_r))$ over a germ $(S,s)$. Applying the correspondence defined in Theorem~\ref{th:dirresdef} we obtain a full deformation
of full $\Ss{\Rs}$-modules whose specialty defect is constant.
\end{remark}

\begin{corollary}
\label{cor:specialtydefectconstant}
The specialty defect is constant in a full deformation of a full sheaf.
\end{corollary}
\proof
Given a full sheaf $\Sf{M}$, let $(\phi_1,...,\phi_r)$ be generic sections. For any deformation $(\overline{\Sf{M}},\iota)$ of $\Sf{M}$ over 
$(S,s)$ let $(\overline{\phi}_1,...,\overline{\phi}_r)$ be an extension of the sections over $(S,s)$. Applying the correspondence of 
Theorem~\ref{th:dirresdef} to $(\overline{\Sf{M}},(\overline{\phi}_1,...,\overline{\phi}_r),\iota)$ we obtain a speciality defect constant 

deformation. Apply the inverse correspondence to get back $(\overline{\Sf{M}},(\overline{\phi}_1,...,\overline{\phi}_r),\iota)$ and use the previous 
Remark.
\endproof

Finally we need to compare the isomorphism between the functors $\mathbf{Def_{X,M}^{(\phi_1,...,\phi_r)}}$ and $\mathbf{Def_{X,\Sf{C}}^{(\psi_1,...,\psi_r)}}$ with the isomorphism between the functors $\mathbf{FullDef_{\Rs,X,\Sf{M}}^{(\phi_1,...,\phi_r)}}$ and $\mathbf{SDCDef_{\Rs,X,\Sf{A}}^{(\psi_1,...,\psi_r)}}$. 

The results we need are the following:

\begin{proposition}
\label{prop:defdirressing}
Let $X$ be a normal Stein surface with Gorenstein singularities, let $(M,(\phi_1,...,\phi_r))$ be a reflexive $\Ss{X}$-module of rank $r$ together with $r$ generic sections, $\pi:\Rs\to X$ be a resolution and $\Sf{M}$ be the associated full $\Ss{\Rs}$-module. 
Let $(\Sf{A},(\psi_1,...,\psi_r))$ be the result of applying the correspondence of Theorem~\ref{th:corres} to $(\Sf{M},(\phi_1,...,\phi_r))$, and $(\Sf{C},(\psi'_1,...,\psi'_r))$ the result of applying the correspondence of Theorem~\ref{th:corrsing} to $(M,(\phi_1,...,\phi_r))$.

Consider a full deformation $(\Rsd,\Sf{X},\overline{\Sf{M}},(\overline{\phi}_1,...,\overline{\phi}_r),\iota)$ of $(\Rs,X,\Sf{M},(\phi_1,...,\phi_r))$ over a base $S$. Let $(\Sf{X},\overline{M},(\overline{\phi}_1,...,\overline{\phi}_r),\iota)$ be the deformation of $(X,M,(\phi_1,...,\phi_r))$ obtained applying $\Pi_*$ (see Proposition~\ref{prop:naturaltrans}). 

Let $(\Rsd,\Sf{X},\overline{\Sf{A}},(\overline{\psi}_1,...,\overline{\psi}_r),\rho)$ be the result of applying the correspondence of Theorem~\ref{th:dirresdef} to $(\Rsd,\Sf{X},\overline{\Sf{M}},(\overline{\phi}_1,...,\overline{\phi}_r),\iota)$, and let 
$(\Sf{X},\overline{\Sf{C}},(\overline{\psi}'_1,...,\overline{\psi}'_r),\rho')$ be the result of applying the correspondence of Theorem~\ref{th:dirXdef} to $(\Sf{X},\overline{M},(\overline{\phi}_1,...,\overline{\phi}_r),\iota)$.

Then there is a natural inclusion $\overline{\Sf{C}}\subset \Pi_*\overline{\Sf{A}}$ which extends the inclusion of $\Sf{C}$ into $\pi_*\Sf{A}$ predicted in Proposition~\ref{prop:dirressing}. Under this inclusion the sections $(\overline{\psi}_1,...,\overline{\psi}_r)$ are identified with $(\overline{\psi}'_1,...,\overline{\psi}'_r)$.   
\end{proposition}
\proof
The proof is a straightforward adaptation of the proof of Proposition~\ref{prop:dirressing} in which one should quote Proposition~\ref{prop:dualMdeformado} instead of Lemma~\ref{lema:dualM}.
\endproof

\begin{proposition}
\label{prop:comparecorrdef}
 Let $\pi:\Rs\to X$ be a resolution of a normal Stein surface with Gorenstein singularities. Let $(\Sf{A},(\psi_1,...,\psi_r))$ be a pair formed by a  rank 1 generically reduced  $1$-dimensional Cohen-Macaulay $\Ss{\Rs}$-module, 
whose support meets the exceptional divisor $E$ in finitely many points, and a set of $r$ global sections spanning $\Sf{A}$ as $\Ss{\Rs}$-module and satisfying the Containment Condition. Let $\Sf{C}$ be the $\Ss{X}$-module spanned by $\psi_1,...,\psi_r$ (then $\Sf{C}$ is  a  rank 1 generically reduced  $1$-dimensional Cohen-Macaulay $\Ss{X}$-module). Let $(\Rsd,\Sf{X},\overline{\Sf{A}},(\overline{\psi}_1,...,\overline{\psi}_r),\rho)$ be a specialty constant deformation of $(\Rs,X,\Sf{A},(\psi_1,...,\psi_r))$, with $\Pi:\Rsd\to\Sf{X}$ the simultaneous resolution. Denote by $\overline{\Sf{C}}$ the $\Ss{\Sf{X}}$-submodule of $\Pi_*\overline{\Sf{A}}$ spanned by $(\overline{\psi}_1,...,\overline{\psi}_r)$. Then $(\Sf{X},\overline{\Sf{C}},(\overline{\psi}_1,...,\overline{\psi}_r),\rho)$ is a deformation of $(X,\Sf{C},(\psi_1,...,\psi_r))$. 

Moreover, let $(\Rsd,\Sf{X},\overline{\Sf{M}},(\overline{\phi}_1,...,\overline{\phi}_r),\rho)$ be the result of applying the correspondence of Theorem~\ref{th:dirresdef} to  $(\Rsd,\Sf{X},\overline{\Sf{A}},(\overline{\psi}_1,...,\overline{\psi}_r),\rho)$ and $(\Sf{X},\overline{M},(\overline{\phi}'_1,...,\overline{\phi}'_r),\rho')$ be the result of applying the correspondence of Theorem~\ref{th:dirXdef} to 
$(\Sf{X},\overline{\Sf{C}},(\overline{\psi}_1,...,\overline{\psi}_r),\rho)$. Then we have the equality of deformations
$$(\Sf{X},\Pi_*\overline{\Sf{M}},(\overline{\phi}_1,...,\overline{\phi}_r),\rho)=(\Sf{X},\overline{M},(\overline{\phi}'_1,...,\overline{\phi}'_r),\rho').$$
\end{proposition}
\proof
In order to prove that $(\Sf{X},\overline{\Sf{C}},(\overline{\psi}_1,...,\overline{\psi}_r),\iota)$ is a deformation of $(X,\Sf{C},(\psi_1,...,\psi_r))$ we only need to prove the flatness of $\overline{\Sf{C}}$ over the base $S$ of the deformation. This follows because we have the exact sequence 
$$0\to\overline{\Sf{C}}\to\Pi_*\overline{\Sf{A}}\to\overline{\Sf{D}}\to 0,$$
the module $\overline{\Sf{D}}$ is flat over $S$ because $(\Rsd,\Sf{X},\overline{\Sf{A}},(\overline{\psi}_1,...,\overline{\psi}_r),\rho)$ is a specialty constant deformation and the module $\Pi_*\overline{\Sf{A}}$ is also flat over $S$ because it coincides with $\overline{\Sf{A}}$ as a $\Ss{S}$-module.

The remaining assertion runs parallel to the proof of Proposition~\ref{prop:invressing}: according with the proof of Theorem~\ref{th:dirresdef} and its proof the module $\overline{\Sf{N}}$ in the sequence 
$0\to\overline{\Sf{N}}\to\Ss{\Rsd}^r\to\overline{\Sf{A}}\to 0$
is the dual of $\overline{\Sf{M}}$. 
Pushing down by $\Pi_*$ we obtain
$$0\to\Pi_*\overline{\Sf{N}}\to\Ss{\Sf{X}}^r\to\Pi_*\overline{\Sf{A}}\to R^1\Pi_*\overline{\Sf{N}}\to R^1\Pi_*\Ss{\Rsd}^r\to 0,$$
and the image of the map $\Ss{\Sf{X}}^r\to\Pi_*\overline{\Sf{A}}$ is the $\Ss{\Sf{X}}$-module spanned by $\overline{\psi}_1,...,\overline{\psi}_r$, that is, the module $\overline{\Sf{C}}$.
So we obtain the sequence 
$$0\to\Pi_*\overline{\Sf{N}}\to\Ss{\Sf{X}}^r\to\overline{\Sf{C}}\to 0.$$
According with Theorem~\ref{th:dirXdef} and its proof the module $\Pi_*\overline{\Sf{N}}$ is isomorphic to the dual of $\overline{M}$. By 
Proposition~\ref{prop:dualMdeformado}  the module $\Pi_*\overline{\Sf{N}}$ is isomorphic to the dual of $\Pi_*\overline{\Sf{M}}$. This concludes the proof of the 
equality $\Pi_*\overline{\Sf{M}}=\overline{M}$. 
\endproof

\section{Classification of Gorenstein normal surface singularities in Cohen-Macaulay representation types}
\label{sec:fintamewild}

In this section we prove that non log-canonical Gorenstein surface singularities are of wild Cohen-Macaulay representation type. This confirms a conjecture by Drodz, Greuel and Kashuba~\cite{DrGrKa}, and completes the classification in Cohen-Macaulay representation types of Gorenstein normal surface singularities. The reader may consult~\cite{DrGrKa} and the references quoted there for a full definition of finite, tame and wild Cohen-Macaulay representation types. For our purposes it is enough to know that if a singularity admits essential families of indecomposable reflexive modules then it is of wild Cohen-Macaulay representation type.

\begin{definition}
\label{def:reptype}
Let $(X,x)$ be a normal surface singularity. A family $\overline{M}$ of $\Ss{X}$-modules over a variety $S$ is called {\em essential} 
for any reflexive module $M$, if the set of points $s\in S$ such that $\overline{M}|_s$ is isomorphic to 
$M$ is a $0$-dimensional analytic subvariety.
\end{definition}

\begin{proposition}
\label{prop:unboundedfam}
Let $(X,x)$ be a normal Gorenstein surface singularity. If the minimal canonical order of prime divisors over $x$ (see Definition~\ref{def:mincanord}) is 
unbounded from below, then there are essential families of indecomposable special reflexive $\Ss{X}$-modules with arbitrarily high dimensional base. Thus $(X,x)$ is of wild Cohen-Macaulay representation type. 
\end{proposition}
\proof
Let $\pi:\Rs\to X$ be a resolution which contains an irreducible component $F$ of the exceptional divisor $E$ of $\pi$ such that the minimal 
canonical order at $F$ equals $-d$. We are going to construct a $d$-dimensional essential family of indecomposable special reflexive $\Ss{X}$-modules. 

Let $V$ be the set of sequences $(x_1,...,x_d)$ of infinitely near points to $x$ having the following inductive properties: the point
$x_1$ is a smooth point of the exceptional divisor $E$ belonging to $F$. For $2\leq i\leq d-1$ let $\pi_{i-1}^{x_1,...,x_{i-1}}:\Rs^{i-1}\to X$ be the 
composition of the blow-ups of $\Rs$ at $x_1,...,x_{i-i}$, and the map $\pi$, let $E^{i-1}$ be the exceptional divisor of 
$\pi_{i-1}^{x_1,...,x_{i-1}}$ and $F^{i-1}$ be the exceptional divisor of the blow up at $x_{i-1}$; the point $x_i$ is a smooth point of the exceptional
divisor $E^{i-1}$ belonging to $F^{i-1}$. According with the methods of~\cite{Bo}, Chapter 3, there is a variety $V$ parametrizing such sequences of 
infinitely near points, and a universal sequence of proper birational maps
$$Z^{d-1}\stackrel{\Pi^{d-1}}{\longrightarrow} Z^{d-1}\to...\to Z^{1}\stackrel{\Pi^{1}}{\longrightarrow} \Rs\times V\stackrel{\Pi^{0}}{\longrightarrow} X\times V,$$
so that $\Pi^0$ is the map $(\pi,Id_V)$ and for any $(x_1,...,x_d)\in V$ and any $k\leq d-1$, the fibre morphism 
$$\Pi^0_{(x_1,...,x_d)}\comp...\comp\Pi^{k}_{(x_1,...,x_d)}:Z^k_{(x_1,...,x_d)}\to X\times\{(x_1,...,x_d)\},$$
is equal to $\pi\comp\pi^{x_1}....\comp\pi_{k}^{x_1,...,x_{k}}$. We denote the composition 
$\Pi^0\comp...\comp\Pi^{d-1}$ by $\rho$. 

Let $\overline{D}\subset Z^{d-1}$ be a divisor such that its fibre over $(x_1,...,x_d)\in V$ is a smooth curvette meeting the exceptional divisor 
$E^{d-1}$ transversely through $x_d$. It is clear that $\overline{D}$ exists at least over a Zariski open subset of $V$. Let $r(x_1,...,x_d)$ be the minimal
number of generators of $(\rho_*\Ss{\overline{D}})_{(x_1,...,x_d)}$ as a $\Ss{X}$-module. It is easy to prove that the function $r(x_1,...,x_d)$ is 
upper semi-continuous in $V$. By shrinking $V$ to a Zariski open subset we may assume the existence of sections 
$(\overline{\psi}_1,....,\overline{\psi}_r)$ of
$\rho_*\Ss{\overline{D}}$ that specialized over each point $(x_1,...,x_d)\in V$ gives a minimal set of generators of 
$(\rho_*\Ss{\overline{D}})_{(x_1,...,x_d)}$. 

Applying the correspondence of Theorem~\ref{th:dirXdef} to $((\rho_*\Ss{\overline{D}})_{(x_1,...,x_d)},(\overline{\psi}_1,....,\overline{\psi}_r))$ 
we obtain pair $(\overline{M},(\overline{\phi}_1,...,\overline{\phi}_r))$, where $\overline{M}$ is a family of reflexive $\Ss{X}$-modules. 
Since $R^1\rho_*\Ss{\overline{D}}$ vanishes, Lemma~\ref{lem:h0h1} implies the equality 
$(\rho_*\Ss{\overline{D}})_{(x_1,...,x_d)}=(\rho_{(x_1,...,x_d)})_*\Ss{\overline{D}|_{(x_1,...,x_d)}}$ for any $(x_1,...,x_d)\in V$. Consequently
the module $M|_{(x_1,...,x_d)}$ is the result of applying the correspodence of Theorem~\ref{th:corrsing} to the pair 
$((\rho_{(x_1,...,x_d)})_*\Ss{\overline{D}|_{(x_1,...,x_d)}},(\overline{\psi}_1|_{(x_1,...,x_d)},....,\overline{\psi}_r|_{(x_1,...,x_d)}))$. Then, having chosen $\overline{D}$ generic, by
Proposition~\ref{prop:Aeslanormalizacion}, Remark~\ref{rem:canonicalcondset}  and Proposition~\ref{prop:decompesp} we conclude that $M|_{(x_1,...,x_d)}$ is a special indecomposable module. 

The dimension of $V$ equals $d$, since each of the infinitely near points is free. Therefore, in order to finish the proof we only have to show that 
the family is essential. We claim that $(\rho_{(x_1,...,x_d)})_*$ is the minimal resolution adapted to $M|_{(x_1,...,x_d)}$. If the claim is true
the modules $M|_{(x_1,...,x_d)}$ are pairwise non-isomorphic by Theorem~\ref{Teo:final} and we are done. 

The claim follows from Proposition~\ref{prop:minadapnumchar}, noticing the facts that 
$$cond\left( \mathfrak{K}_{((\rho_{(x_1,...,x_d)})_*\Ss{\overline{D}|_{(x_1,...,x_d)}},(\overline{\psi}_1|_{(x_1,...,x_d)},....,\overline{\psi}_r|_{(x_1,...,x_d)}))} \right)=0,$$
and that the minimal canonical order of $F^{i}$ vanishes if and only if $i=d$.
\endproof

\begin{theorem}
\label{th:reptype}
A Gorenstein surface singularity is of finite Cohen-Macaulay representation type if and only if it is a rational double point. Gorenstein surface singularities 
of tame Cohen-Macaulay representation type are precisely the log-canonical ones. The remaining Gorenstein surface singularities are of wild Cohen-Macaulay representation type.
\end{theorem}
\proof
In~\cite{Es} it is proved that a normal surface singularity has finitely many indecomposable reflexive modules if and only if it is a quotient
singularity. This implies that the Gorenstein surface singularities of finite Cohen-Macaulay representation type are exactly the rational double points.
In~\cite{DrGrKa} it is proved that log-canonical surface singularities are of tame Cohen-Macaulay representation type. By the classification of Example 3.27 
of~\cite{Ko}, if $X$ is a Gorenstein
singularity which is not log-canonical, and $\sum q_i E_i$ is the divisor associated with its Gorenstein form at the minimal resolution, then 
either there is a $q_i<-1$, or the exceptional divisor have singularities of Milnor number at least $3$. In this case it is possible to obtain
resolutions which are small with respect to the canonical cicle with arbitrarily negative coefficients for the divisor of the Gorenstein form. 
Proposition~\ref{prop:unboundedfam} shows now that non log-canonical singularities are of wild Cohen-Macaulay representation type. 
\endproof

\section{Lifting deformations}
\label{sec:liftingdefs}
Let $X$ be a normal Stein surface. Let $\pi:\Rs\to X$ a resolution with exceptional divisor $E$.
Let $M$ be a reflexive $\Ss{X}$-module of rank $r$ and $\Sf{M}$ be the associated full $\Ss{\Rs}$-module. We study when deformations of $M$ lift to full deformations of $\Sf{M}$. 

\begin{definition}
\label{def:liftlocus}

Let $X$ be a normal Stein surface. Let $\Sf{X}$ be a deformation of $X$ over a base $(S,s)$. Let $\Pi:\Rsd\to \Sf{X}$ be a very weak simultaneous resolution with exceptional divisor $\Sf{E}$. The {\em strict transform} $\hat{A}$ of a subscheme $A\subset \Sf{X}$ is the scheme theoretic closure of $(\Pi|_{\Rsd\setminus \Sf{E}})^{-1}(A\setminus\Pi(\Sf{E})$ in $\Rsd$. A subscheme $A\subset \Sf{X}$ {\em lifs to} $\Rsd$ if the strict transform $\hat{A}$ to $\Rsd$ is finite over $A$. 

Let $\Sf{C}$ be a rank $1$ generically reduced $1$-dimensional Cohen-Macaulay $\Ss{X}$-module, $(\psi_1,...,\psi_r)$ be a system of generators. A deformation
$(\Sf{X},\overline{\Sf{C}},(\psi_1,...,\psi_r),\iota)$ of $(\Sf{C},(\psi_1,...,\psi_r))$ {\em lifs in a specialty defect constant way to} $\Rsd$ if 
\begin{itemize}
 \item there is a rank $1$ generically reduced $1$-dimensional Cohen-Macaulay $\Ss{\Rs}$-module $\Sf{A}$ meeting the exceptional divisor in finitely many points, and a set of generators $(\psi'_1,...,\psi'_r)$ as a $\Ss{\Rs}$ module such that the $\Ss{X}$-submodule of $\pi_*\Sf{A}$ generated by the sections $(\psi'_1,...,\psi'_r)$ is isomorphic to $\Sf{C}$. Then, by abuse of notation we denote the sections $(\psi'_1,...,\psi'_r)$ by $(\psi_1,...,\psi_r)$.
 \item There is a specialty defect constant deformation 
 $(\Rsd,\Sf{X},\overline{\Sf{A}},(\overline{\psi}_1,...,\overline{\psi}_r),\iota')$ of $(\Sf{A},(\psi_1,...,\psi_r))$ whose image by $\Pi_*$ equals $(\Sf{X},\overline{\Sf{C}},(\psi_1,...,\psi_r),\iota)$.
\end{itemize} 
\end{definition}

\begin{remark}
\label{rem:liftstrictcommuta}
In the setting of the previous definition, it is clear that $A$ lifts to $\Rsd$ if and only if the fibre over $s$ of the strict transform of $A$ coincides with the strict transform of the fibre of $A$ over $s$. 
\end{remark}

\begin{lemma}
\label{lem:curvecriterionliftability}
Let $X$ be a normal Stein surface. Let $\Sf{X}$ be a deformation of $X$ over a reduced base $(S,s)$. Let $\Pi:\Rsd\to \Sf{X}$ be a very weak simultaneous resolution with exceptional divisor $\Sf{E}$. Let $A\subset \Sf{X}$ be a closed subscheme such that the fibre $A_s$ of $A$ over $s$ is of dimension $1$, and such that the Zariski open subset $A\cap (\Sf{X}\setminus\Pi(\Sf{E}))$ is flat over $S$. Then $A$ is liftable if and only if for any for any arc $\gamma:Spec(\mathbb{C}[[t]])\to (S,s)$ the subscheme 
$A\times_S Spec(\mathbb{C}[[t]])\subset \Sf{X}\times_{S} Spec(\mathbb{C}[[t]])$ is liftable for the very weak simultaneous resolution obtained by pullback.
\end{lemma}
\proof
If $A$ is liftable, then, by the previous Remark it is obvious that for any arc $\gamma:Spec(\mathbb{C}[[t]])\to (S,s)$ the subscheme 
$A\times_S Spec(\mathbb{C}[[t]])\subset \Sf{X}\times_{S} Spec(\mathbb{C}[[t]])$ is liftable.

It is clear that a subscheme $A$ is liftable if and only if each of the irreducible components of the corresponding reduced subscheme $A^{red}$ are liftable. Hence we may assume $A$ to be reduced and irreducible.

Conversely, assume that $A$ is not liftable. Then, by the previous Remark the strict transform of the fibre $A_s$ to $\Rsd_s$ is strictly contained in the fibre $\hat{A}|_s$ over $s$ of the strict transform of $A$ to $\Rsd$. Let $a\in\hat{A}|_s$ a point not contained in the 
strict transform of $A_s$. On the other hand, there is a Zariski open subset $W$ of $S$ such that for any $s'\in W$ the fibre $\hat{A}|_{s'}$ coincides with the strict transform of the fibre over $s'$. 

Since $A$ is irreducible, so it is its strict transform $\hat{A}$. Then the strict transform $\hat{A}|_W$ is Zariski-dense in $\hat{A}$, and consequently $a$ is at the closure of $\hat{A}|_W$. By Curve Selection Lemma there exists an arc $\hat{\gamma}:Spec(\mathbb{C}[[t]])\to\hat{A}$ such that $\hat{\gamma}(0)=a$ and such that the generic point of $Spec(\mathbb{C}[[t]])$ is mapped to $\hat{A}|_W$. Let $\gamma:Spec(\mathbb{C}[[t]])\to S$ be the composition of $\hat{\gamma}$ with $\Pi$ and the projection to $S$. Then $\gamma$ is an arc so that $A\times_S Spec(\mathbb{C}[[t]])\subset \Sf{X}\times_{S} Spec(\mathbb{C}[[t]])$ is not liftable for the very weak simultaneous resolution obtained by pullback.
\endproof

\begin{proposition}
\label{prop:necessarylifting}
Let $X$ be a normal Stein surface with Gorenstein singularities. Let $\Sf{X}$ be a deformation of $X$ over a reduced base $(S,s)$. Let $\Pi:\Rsd\to \Sf{X}$ be a very weak simultaneous resolution with exceptional divisor $\Sf{E}$. Denote by $\pi:\Rs\to X$ the fibre of $\Pi$ over $s$. Let $M$ be a reflexive $\Ss{X}$-module of rank $r$. Let $(\Sf{X},\overline{M},\iota)$ be a deformation of $M$ over $(S,s)$. Let $\Sf{M}$ be the full $\Ss{\Rs}$-module associated to $M$. 
The first 3 of the following conditions are equivalent and imply the fourth and fifth.
\begin{enumerate}
 \item There is a deformation $(\Rsd,\Sf{X},(\overline{\Sf{M}}),\iota')$ of $\Sf{M}$ which transforms under $\Pi_*$ to $(\Sf{X},\overline{M},\iota)$.
 \item For any collection $(\phi_1,....,\phi_r)$ of nearly generic global sections of $\Sf{M}$ and any extension 
 $(\overline{\phi}_1,...,\overline{\phi}_r)$ as sections of $\overline{M}$, the deformation $(\Sf{X},\Sf{C},(\psi_1,...,\psi_r),\rho)$ obtained applying the correspondence of Theorem~\ref{th:dirXdef} to $(\Sf{X},\overline{M},(\overline{\phi}_1,...,\overline{\phi}_r),\iota)$ lifts in a specialty defect constant way to $\Rsd$.
 \item There exists a collection $(\phi_1,....,\phi_r)$ of nearly generic global sections of $\Sf{M}$ and an extension 
 $(\overline{\phi}_1,...,\overline{\phi}_r)$ as sections of $\overline{M}$, such that the deformation $(\Sf{X},\Sf{C},(\psi_1,...,\psi_r),\rho)$ obtained applying the correspondence of Theorem~\ref{th:dirXdef} to $(\Sf{X},\overline{M},(\overline{\phi}_1,...,\overline{\phi}_r),\iota)$ lifts in a specialty defect constant way to $\Rsd$.
 \item For any collection $(\phi_1,....,\phi_r)$ of nearly generic global sections of $\Sf{M}$ and any extension $(\overline{\phi}_1,...,\overline{\phi}_r)$ as sections of $\overline{M}$, the support of the degeneracy module of $(\overline{M},(\overline{\phi}_1,...,\overline{\phi}_r))$ is liftable.
 \item There exists a collection $(\phi_1,....,\phi_r)$ of nearly generic global sections of $\Sf{M}$ and an extension 
 $(\overline{\phi}_1,...,\overline{\phi}_r)$ as sections of $\overline{M}$, such that the support of the degeneracy module of 
 $(\overline{M},(\overline{\phi}_1,...,\overline{\phi}_r))$ is liftable.
\end{enumerate}
\end{proposition}
\proof
Suppose Condition $(1)$ holds. Consider a collection $(\phi_1,....,\phi_r)$ of nearly generic global sections of $\Sf{M}$ and an extension 
 $(\overline{\phi}_1,...,\overline{\phi}_r)$ as sections of $\overline{M}$. Applying the correspondence of Theorem~\ref{th:dirresdef} to 
 $(\Rsd,\Sf{X},\overline{\Sf{M}},(\overline{\phi}_1,...,\overline{\phi}_r),\iota')$ we obtain the desired lifting. This proves Condition $(2)$. 
 
Condition $(2)$ implies Condition $(3)$ trivially. 
 
If Condition $(3)$ holds let $(\Rsd,\Sf{X},\Sf{A},(\overline{\psi}_1,...,\overline{\psi}_r),\rho)$ be the specialty defect constant lifting. Applying 
the correspondence of Theorem~\ref{th:dirresdef} to it we obtain a deformation $(\Rsd,\Sf{X},\overline{\Sf{M}},(\overline{\phi}_1,...,\overline{\phi}_r),\iota')$.
The cuadruple $(\Rsd,\Sf{X},\overline{\Sf{M}}),\iota')$ obtained by forgetting the sections is the deformation that we need to obtain to show that Condition $(1)$ 
holds.

If Condition $(2)$ holds let $(\Rsd,\Sf{X},\Sf{A},(\overline{\psi}_1,...,\overline{\psi}_r),\rho)$ be the specialty defect constant lifting. The support of $\Sf{A}$ is the strict transform of the support of the degeneracy module of $(\overline{M},(\overline{\phi}_1,...,\overline{\phi}_r))$. Since for any $s'\in S$ the support of $\Sf{A}|_{s'}$ meets the exceptional divisor at finitely many points, Condition $(4)$ holds.

Condition $(4)$ implies condition $(5)$ obviously.
\endproof

In the next example we use the non-liftability of the support of the degeneracy module to prove that a deformation of a reflexive module does not lift to a full deformation.

 \begin{example}
 \label{ex:nonlifting}
 Let $X=V(xz-y^2)\subset\mathbb{C}^3$. Let $S=Spec(\mathbb{C}[[s]])$. Define $\alpha:Spec(\mathbb{C}[[t,s]])\to X\times S$ by
 $(t,s)\to (t^2,t^3+st,(t^2+s)^2)$. Define $\Sf{C}:=\alpha_*\mathbb{C}[[t,s]]$, and let $(\overline{\psi}_1,...,\overline{\psi}_r)$ be a system 
 of generators of $\Sf{C}$ as $\Ss{X\times S}$-module. The correspondence of Theorem~\ref{th:dirresdef} defines a deformation $\overline{M}$ of 
 reflexive modules which does not lift to a full deformation of full sheaves, since the support of the degeneracy module is not liftable.
 \end{example}

Here we show an example where the support of the degeneracy module is liftable, but there is no full deformation.

\begin{example}
\label{ex:liftsnotlifts}
Let $X=V(x^3+y^3+z^3)$. This is a minimally elliptic singularity as studied in~\cite{Ka}. The blowing up at the origin $\pi:\Rs\to X$ produces its minimal resolution. Its exceptional divisor is a smooth elliptic curve $E$. Let $\tilde{C}$ be a smooth curvette embedded in 
$\Rs$, which meets $E$ at a single point $p$ with intersection multiplicity equal to $2$. Let $t$ be a uniformizing parameter for the germ $(\tilde{C},p)$. The curve $C:=\pi(\tilde{C})$ has an ordinary cusp singularity at the origin of $\CC^3$ (that is $C=Spec\CC[[t^2,t^3]]$). Let $S:=Spec\CC[[s]]$. We have the isomorphism $\Ss{\tilde{C}\times S}\cong\CC[[t,s]]$; this endows $\CC[[t,s]]$ with structures of $\Ss{X\times S}$-module and of $\Ss{\Rs\times S}$-module. Denote by $\Sf{C}$ the $\Ss{X\times S}$-submodule of $\CC[[t,s]]$ spanned by $(s+t,t^2)$. We let $\Sf{D}$ be its cokernel as $\Ss{X\times S}$ module. It is easy to check that 
$\Sf{D}$ is $\CC[[s]]$-flat, and hence $\Ss{C}$ is $\CC[[s]]$-flat as well. Apply the correspondence of Theorem~\ref{th:dirXdef} to
$(\Sf{C},(s+t,t^2))$ and obtain $(\overline{M},(\overline{\phi}_1,\overline{\phi}_2))$, where $\overline{M}$ is a family of reflexive 
$\Ss{X}$-modules.

The support of $\Sf{C}$ equals $C\times S$, and its strict transform equals $\tilde{C}\times S$. So $C\times S$ lifts. On the other hand, if there is a family $\overline{\Sf{M}}$ of full $\Rs$-modules lifting $\overline{M}$, then applying the correspondence of Theorem~\ref{th:dirresdef} to $(\overline{\Sf{M}},(\overline{\phi}_1,\overline{\phi}_2))$ we would obtain $(\overline{\Sf{A}},(s+t,t^2))$,
where $\overline{\Sf{A}}$ is the $\Ss{\tilde{C}\times S}$ module spanned by $(s+t,t^2)$. Since this module is not Cohen-Macaulay of dimension $2$ (it is not a free module over $\CC[[t,s]]$), it can not be a flat family over $\CC[[s]]$ of $1$-dimensional Cohen-Macaulay $\Ss{\Rs}$-modules. This is a contradiction which implies that the there is no family of full $\Rs$-modules lifting $\overline{M}$.

Another way of proving the non-existence of lifting is observing that the specialty defect of the full $\Ss{\Rs}$-module lifting $\overline{M}|_0$ is zero, that the specialty defect of the full $\Ss{\Rs}$-module lifting $\overline{M}|_{s}$ is not zero if $s\neq 0$, and using Corollary~\ref{cor:specialtydefectconstant}.
\end{example}

\begin{proposition}
\label{prop:genericlifting}
Let $X$ be a normal Stein surface with Gorenstein singularities. Let $\Sf{X}$ be a deformation of $X$ over a reduced base $(S,s)$. Let $\Pi:\Rsd\to \Sf{X}$ be a very weak simultaneous resolution. Let $M$ be a reflexive $\Ss{X}$-module and $(\Sf{X},\overline{M},\iota)$ be a deformation of $(X,M)$ over $(S,s)$. Then there exists a dense Zariski open subset of $S$ where the deformation lifts as a full family of $\Ss{\Rsd}$-modules.
\end{proposition}
 \proof
Denote by 
\begin{align*}
\overline{\Sf{N}}&:= \left( \Pi^* \overline{M} \right)^{\smvee},\\
\overline{\Sf{M}}&:= \left( \Pi^* \overline{M} \right)^{\smvee \smvee},
\end{align*}
and also denote by $\Sf{M}_{s'}:=\left( \Pi^*\overline{M}|_{s'}\right)^{\smvee \smvee}$ the full sheaf associated to $\overline{M}|_{s'}$ and $\Sf{N}_{s'}:=\left( \Pi^*\overline{M}|_{s'}\right)^{\smvee}= \left( \Sf{M}_{'s} \right)^{\smvee}$ for any $s'\in S$.

By genericity of flatness over a reduced base, there exists a Zariski dense open subset $U\subset S$ such that the following two properties hold:
\begin{enumerate}
\item The sheaves $\Pi^* \overline{M}$, $\overline{\Sf{N}}$, $\Exts_{\Ss{\Rsd}}^1\left(\Pi^* \overline{M}, \Ss{\Rsd} \right)$ and $\Exts_{\Ss{\Rsd}}^2\left(\Pi^* \overline{M}, \Ss{\Rsd} \right)$ are flat over $U$.
\item The sheaf $R^1 \Pi_* \overline{\Sf{M}}$ is flat over $U$.
\end{enumerate}

For any $s'\in S$ we know that $\Rsd|_{s'}$ is a smooth surface of dimension two, hence the Auslander-Buchsbaum Formula~\cite[Theorem~1.3.3]{BrHe} implies that $\Exts_{\Ss{\Rsd}|_s}^j\left(\Pi^* \overline{M}|_{s'}, \Ss{\Rsd}|_{s'}\right)=0$ for any $j\geq 3$.

The last vanishing, Assertion~(1) of the previous list and a repeated application of Lemma~\ref{lem:extbasechange} and Lemma~\ref{lem:isoext} to the sheaf $\Pi^* \overline{M}$ implies the isomorphism $\overline{\Sf{N}}|_{s'} \cong \Sf{N}_{s'}$ for any $s' \in U$. So we have that $\overline{\Sf{N}}|_U$ is locally free, and dualizing we obtain that $\overline{\Sf{M}}|_U$ is flat over $U$ (locally free on $\Rsd$) and the isomorphism $\overline{\Sf{M}}|_{s'} \cong \Sf{M}_{s'}$ for any $s' \in U$.
Finally Assertion (2) of the previous list implies that $(\Rsd|_U, \Sf{X}|_U, \overline{\Sf{M}}|_U)$ is a family of full modules over $U$. 

In order to finish the proof we need to verify that the natural morphism of coherent $\Ss{\Sf{X}}$-modules 
$$\overline{M}|_U\to (\Pi|_U)_* \overline{\Sf{M}}|_U$$
obtained as the composition
$$\overline{M}|_U\to (\Pi|_U)_*(\Pi|_U)^*\overline{M}|_U\to (\Pi|_U)_*((\Pi|_U)^*\overline{M}|_U)^{\smvee \smvee}=(\Pi|_U)_* \overline{\Sf{M}}|_U$$
is an isomorphism.

By Lemma~\ref{lem:h0h1} the sheaf $(\Pi|_U)_* \overline{\Sf{M}}|_U$ is flat over $U$. Using this and Nakayama's Lemma we are reduced to prove that and for any $s'\in U$ we have the specialization 
$$\overline{M}|_{s'}\to ((\Pi|_U)_* \overline{\Sf{M}}|_U)|_{s'}$$
is an isomorphism. Lemma~\ref{lem:h0h1} implies that this morphism coincides with
$$\overline{M}|_{s'}\to (\Pi|_{s'})_* (\overline{\Sf{M}}|_{s'}),$$
but the second sheaf in the module has been proved to be isomorphic to $(\Pi|_{s'})_* \Sf{M}_{s'}$, where $\Sf{M}_{s'}$ is the full
$\Ss{\Rsd_{s'}}$-module associated with $\overline{M}|_{s'}$. Then the morphism is an isomorphism as needed. 
\endproof

\subsection{Sufficient conditions for liftability to full deformations}
\label{sec:suffcondlift}

In this section we show sufficient conditions ensuring the liftability of a family of reflexive sheaves to a full family on a very weak simultaneous resolution. The result we prove probably is not the best that one can hope for, but is more than sufficient for the applications we have in mind. In particular, at some point we simplify things by working on a normal surface singularity rather than on a normal Stein surface.

\begin{definition}
\label{def:deltaconstant}
Let $X$ be a normal Stein surface. Let $\Sf{X}$ be a deformation of $X$ over a reduced base $(S,s)$. Let $M$ be a reflexive $\Ss{X}$-module of rank $r$. A deformation $(\Sf{X},\overline{M},\iota)$ of $(X,M)$ over a reduced base $(S,s)$ is said to be {\em simultaneously normalizable} if the degeneracy locus $\overline{C}$ of $\overline{M}$ for a generic system of $r$ sections admits a simultaneous normalization over $S$. That is, there exists a smooth family of curves $\overline{D}$ over $S$, and a morphism $n:\overline{D}\to\overline{C}$ such that for any 
$s'\in S$ the restriction $n|_{\overline{D}_{s'}}\colon \overline{D}_{s'}\to\overline{C}_{s'}^{red}$ is the normalization.
\end{definition}

\begin{remark}
\label{rem:simnor}
\begin{enumerate}
 \item The existence of simultaneous normalization of $\overline{C}$ in the previous definition does not depend on the choice of the system of generic sections. 
 \item The reader may consult~\cite{GreNor} and~\cite{Ko0} for recent accounts on simultaneous normalization.
 \item There is a way to define a functor of simultaneously normalizable deformations of reflexive modules, considering also non-reduced bases $(S,s)$. The definition has some subtlety since the family of supports $\overline{C}$ need not be a flat family of reduced curves. Since the applications presented in this paper do not need such a definition we avoid it.
 \item If at the previous definition we assume $X$ to be a normal surface singularity (that is, a germ), we may assume that the normalization of $\overline{C}_s$ is a disjoint union of discs $\coprod_i D_i$. As a consequence if the family is simultaneously normalizable, then the predicted simultaneous normalization $\overline{D}$ is equal to $\coprod_i D_i\times S$.
\end{enumerate}
\end{remark}

\begin{theorem}
\label{th:sufficientlifting}
Let $X$ be a normal Gorenstein surface singularity. Let $\Sf{X}$ be a deformation of $X$ over a normal base $(S,s)$. Let $\Pi:\Rsd\to \Sf{X}$ be a very weak simultaneous resolution. Let $M$ be a reflexive $\Ss{X}$-module and $(\Sf{X},\overline{M},\iota)$ be a simultaneously normalizable  deformation of $(X,M)$ over the base $(S,s)$, so that for each $s'\in S$ the module $\overline{M}|_{s'}$ is special. If the support of the degeneracy 
module of $\overline{M}$ for a generic system of sections is liftable to $\Rsd$, then the family $(\Sf{X},\overline{M},\iota)$ lifts to a full family on $\Rsd$.
\end{theorem}
\proof
Let $(\overline{\phi}_1,...,\overline{\phi}_r)$ be a system of generic global sections of $\overline{M}$. Let $(\Sf{X},\overline{\Sf{C}},(\overline{\psi}_1,...,\overline{\psi}_r),\rho)$ be the result of applying the correspondence of Theorem~\ref{th:dirXdef} to $(\Sf{X},\overline{M},(\overline{\phi}_1,...,\overline{\phi}_r),\iota)$. Let $\overline{C}$ be the support of $\overline{\Sf{C}}$. 
Let 
$$n:\overline{D}\to\overline{C}$$
be the simultaneous normalization that exists because $(\Sf{X},\overline{M},\iota)$ is simultaneously normalizable. By Remark~\ref{rem:simnor} we may assume that $\overline{D}$ is the product of $S$ with a disjoint union of discs. 

Let $K(\overline{C})$ be the total fraction ring of $\overline{C}$. If $\overline{C}=\cup_{i=1}^m\overline{C}_i$ is the decomposition in irreducible components then the total fraction ring equals the direct product $K(\overline{C})=\prod_{i=1}^mK(\overline{C}_i)$ of the function fields of the components. 
Denote by $T$ the set of non-zero divisors of $\Ss{\overline{C}}$. The localization $T^{-1}\overline{\Sf{C}}$ is a $K(\overline{C})$-module, which expresses as $T^{-1}\overline{\Sf{C}}=\prod_{i=1}^m\overline{\Sf{C}}_i$, where $\overline{\Sf{C}}_i$ is a $K(\overline{C}_i)$-vector space. Since $\overline{\Sf{C}}$ is a flat family of $1$-dimensional rank 1 generically reduced Cohen-Macaulay modules over the normal base $S$, the natural map to the localization 
$\overline{\Sf{C}}\to T^{-1}\overline{\Sf{C}}$ is injective, and each $\overline{\Sf{C}}_i$ is a $1$-dimensional vector space. Hence the localization $T^{-1}\overline{\Sf{C}}$ is isomorphic to $\prod_{i=1}^mK(\overline{C}_i)=K(\overline{C})$. We have found a $\Ss{\overline{C}}$-module monomorphism $\iota:\overline{\Sf{C}}\hookrightarrow K(\overline{C}).$
Noticing that $\Ss{\overline{D}}$ is a sub-ring of $K(\overline{C})$, it makes sense to define $\overline{\Sf{B}}$ to be the $\Ss{\overline{D}}$-submodule of $K(\overline{C})$ spanned by $\overline{\Sf{C}}$. Since $\overline{\Sf{C}}$ is generated as $\Ss{\overline{C}}$ module by $(\overline{\psi}_1,...,\overline{\psi}_r)$, multiplying by the common denominator of $\iota(\overline{\psi}_1),...,\iota(\overline{\psi}_r)$ we may assume that the image of $\iota$ lies in 
$\Ss{\overline{D}}$. We have got the chain of inclusions of $\Ss{\overline{C}}$-modules
\begin{equation}
 \label{eq:inc1111}
\overline{\Sf{C}}\hookrightarrow\overline{\Sf{B}}\hookrightarrow\Ss{\overline{D}}.
\end{equation}

The second inclusion is an inclusion of $\Ss{\overline{D}}$-modules; thus $\overline{\Sf{B}}$ is an ideal in $\Ss{\overline{D}}$. Suppose that the zero set $V(\overline{\Sf{B}})$ contains an irreducible component $Z$ that is dominant over $S$ by the natural projection map. Such a $Z$ is contained in a unique connected component of $\overline{D}$, and each of these connected components is isomorphic to the product of $S$ times a disc. Let $t$ be a coordinate of the disc. By the normality of $(S,s_0)$ there is a Weierstrass polynomial $P$ in $\Ss{S}[[t]]$ whose zero locus defines $Z$. Then, dividing by the appropriate Weierstrass polynomials we may assume that the embedding $\iota$ is so that the zero set of the ideal $\overline{\Sf{B}}$ does not dominate $S$. Hence there exists a Zariski dense open subset $U$ in $S$ where, under the embedding $\iota$ we have the equality 
\begin{equation}
\label{eq:inc1122} 
\overline{\Sf{B}}|_U=\Ss{\overline{D}|_U}.
\end{equation}

For any $s'\in S$, by specialty of $\overline{M}|_{s'}$ and Proposition~\ref{prop:Aeslanormalizacion} we have the isomorphism 
$\overline{\Sf{C}}|_{s'}\cong\Ss{\overline{D}_{s'}}$. This implies that $\overline{\Sf{B}}|_{s'}$ is a monic ideal in  
$\Ss{\overline{D}_{s'}}$ and the equality $\overline{\Sf{C}}|_{s'}=\overline{\Sf{B}}|_{s'}$. Then Equation~(\ref{eq:inc1122}) implies the equality
\begin{equation}
\label{eq:2222}
\overline{\Sf{C}}|_{U}=\Ss{\overline{D}|_{U}}.
\end{equation}

Define $\overline{\Sf{D}}:=\Ss{\overline{D}}/\overline{\Sf{C}}$. Denote by $\mathfrak{m}_S$ the maximal ideal of $\Ss{S,s}$. Applying $\centerdot\otimes_{\Ss{S}}\Ss{S}/\mathfrak{m}_S$ to the exact sequence of $\Ss{\overline{C}}$-modules
$$0\to \overline{\Sf{C}}\to\Ss{\overline{D}} \to \overline{\Sf{D}}\to 0,$$
we obtain the exact sequence
$$0\to Tor_1^{\Ss{S}}(\overline{\Sf{D}},\Ss{S}/\mathfrak{m}_S)\to \overline{\Sf{C}}\otimes_{\Ss{S}}\Ss{S}/\mathfrak{m}_S\to\Ss{\overline{D}}\otimes_{\Ss{S}}\Ss{S}/\mathfrak{m}_S\to \overline{\Sf{D}}\otimes_{\Ss{S}}\Ss{S}/\mathfrak{m}_S\to 0.$$
The second group in the sequence is isomorphic to $\Sf{C}|_{s_0}$, the third group is $\Ss{\overline{D}|_{s_0}}$ and the morphism connecting them is injective at the generic points of the support of $\Sf{C}|_{s_0}$. Then, since $\Sf{C}|_{s_0}$ is a rank $1$ generically reduced Cohen-Macaulay $\Ss{X}$-module of dimension $1$ the morphism $$\overline{\Sf{C}}\otimes_{\Ss{S}}\Ss{S}/\mathfrak{m}_S\to\Ss{\overline{D}}\otimes_{\Ss{S}}\Ss{S}/\mathfrak{m}_S$$
is injective and then $Tor_1^{\Ss{S}}(\overline{\Sf{D}},\Ss{S}/\mathfrak{m}_S)$ vanishes. Then $\overline{\Sf{D}}$ is
flat over $S$ by the Local Criterion of Flatness. Since by Equality~(\ref{eq:2222}) the $\Ss{S}$-module $\overline{\Sf{D}}$ has proper support we conclude that $\overline{\Sf{D}}$ vanishes. This proves that the inclusion~(\ref{eq:inc1111}) becomes the equality
\begin{equation}
\label{eq:eq3333}
\overline{\Sf{C}}=\Ss{\overline{D}}.
\end{equation}

Since $\overline{C}$ has been assumed to liftable to $\Rsd$ the restriction of $\Pi$ to the strict transform $\hat{\overline{C}}$ of $\overline{C}$ to $\Rsd$ is finite and birational over $\overline{C}$. This implies that $\overline{D}$ dominates $\hat{\overline{C}}$, and 
then we have a ring monomorphism $\Ss{\hat{\overline{C}}}\hookrightarrow\Ss{\overline{D}}$. As a consequence $\Ss{\overline{D}}$ inherits a structure of 
$\Ss{\Rsd}$-module. Summing up we have shown that
$(\Rsd,\Sf{X},\Ss{\overline{D}},(\psi_1,...,\psi_r),Id|_{\Ss{\overline{D}_s}})$ is a specialty defect constant deformation of $(\Rs,X,\Ss{\overline{D}_s},(\psi_1|_s,...,\psi_r|_s))$.

Applying the correspondence of Theorem~\ref{th:dirresdef} to $(\Rsd,\Sf{X},\Ss{\overline{D}},(\psi_1,...,\psi_r),Id|_{\Ss{\overline{D}_s}})$ we obtain a full deformation $(\Rsd,\Sf{X},\overline{\Sf{M}},(\phi_1,...,\phi_r),\iota)$. An application of Proposition~\ref{prop:comparecorrdef} concludes the proof.
\endproof

\section{Moduli spaces of special reflexive sheaves on Gorenstein surface singularities}
\label{sec:moduli}

Let $(X,x)$ be a Gorenstein normal surface singularity. In this section we consider deformations and families of $\Ss{X}$-modules fixing the space $X$. Our aim is to construct moduli spaces of special reflexive modules with prescribed combinatorial type.

\begin{definition}
\label{def:modulifunctor}
Let $(X.x)$ be a Gorenstein surface singularity.
Let $\Sf{G}$ be the graph of a special reflexive $\Ss{X}$-module and $r$ a positive integer. 
A {\em family of special modules with graph} $\Sf{G}$ {\em and rank} $r$ over a complex space $S$ is a $\Ss{X\times S}$-module $\overline{M}$
which is flat over $S$ and such that for any $s\in S$, the module $\overline{M}|_s$ is a special reflexive module of rank $r$ and graph 
$\Sf{G}$.

Define a moduli functor $\mathbf{Mod^r_{\Sf{G}}}$ from the category of Normal Complex Spaces to the category of Sets, assigning to a normal complex space $S$ the set of families of special modules without free factors with graph $\Sf{G}$ and rank $r$, and to a morphisms of complex spaces the corresponding pullback of families.  
\end{definition}

In this section we prove that the previous functors are representable by a complex algebraic variety. 

\begin{remark}
Our moduli functor is somewhat restricted: we only consider families over normal complex spaces. It is an open problem to show that our moduli spaces represent the usual moduli functors. One should notice that even the definition of the moduli functor has some subtleties: restricting Example~\ref{ex:liftsnotlifts} to the base $Spec(\CC[[s]]/(s^2))$ one sees a flat deformation of a special reflexive $\Ss{X}$-module, such that over each point of the base the corresponding reflexive module is special, but that should not be considered as a family of special modules.
\end{remark}

Let $\Sf{G}$ be the graph of a special reflexive module (the possible graphs are classified in Theorem~\ref{th:charresgraphsp}).
Deleting the arrows of $\Sf{G}$ we obtain a resolution graph $\Sf{G}^{o}$ of the singularity $X$. Two different resolutions with the 
same resolution graph only differ in the positions of the infinitely near points which are the centers of the blow ups occuring after the minimal 
resolution. 

Let $\mathfrak{M}'_{\Sf{G}^{o}}$ be the set of pairs $(\pi,\varphi)$, where $\pi:\Rs\to X$ is a resolution of singularities and 
$\varphi$ is a bijection from the vertices of $\Sf{G}^{o}$
to the irreducible components of the exceptional divisor $E$ of $\pi$ inducing an isomorphism from $\Sf{G}^{o}$ to the dual graph of the 
resolution. The group $Aut(\Sf{G}^{o})$ of automorphims of the graph $\Sf{G}^{o}$ acts on $\mathfrak{M}'_{\Sf{G}^{o}}$ 
by composition on the left at the second coordinate. The group $Aut(\Sf{G})$ is a subgroup of $Aut(\Sf{G}^{o})$. 
Define $\mathfrak{M}_{\Sf{G}^{o}}$ and $\mathfrak{M}_{\Sf{G}}$ to be the quotient of $\mathfrak{M}'_{\mathfrak{G}^{o}}$ by
$Aut(\Sf{G}^{o})$ and $Aut(\Sf{G})$ respectively. The points of $\mathfrak{M}_{\Sf{G}^{o}}$ and $\mathfrak{M}_{\Sf{G}}$ are equivalence classes which will be denoted by $(\pi,\varphi)$ for simplicity.

\begin{lemma}
\label{lem:varietyresolutions}
The sets $\mathfrak{M}_{\Sf{G}}$ and  $\mathfrak{M}_{\Sf{G}^o}$ have a natural structure of algebraic variety, and the natural map from the 
first to the second is an etale covering.

The variety $\mathfrak{M}_{\Sf{G}}$ is a  moduli space of resolutions of $X$, and has a universal family in the following sense: there is a birational morphism
$$\Pi:\tilde{\Sf{X}}\to X\times \mathfrak{M}_{\Sf{G}},$$
such that for any $(\pi,\varphi)\in \mathfrak{M}_{\Sf{G}}$ the pullback morphism 
$$\Pi|_{[\tilde{\Sf{X}}_{(\pi,\varphi)}}:\tilde{\Sf{X}}_{(\pi,\varphi)}\to X\times\{(\pi,\varphi)\}$$
coincides with $\pi$, where $\tilde{\Sf{X}}_{(\pi,\varphi)}$ denotes the fibre of $\tilde{\Sf{X}}$ over $(\pi,\varphi)$ by the composition of $\Pi$ with the projection to the second factor.
\end{lemma}
\proof
See~\cite{Bo}, Chapter 3.
\endproof 

\begin{notation}
\label{not:restrU}
Let $U$ be an open subset of $\mathfrak{M}_{\Sf{G}}$, the restriction of the universal family over $U$ is denoted by 
$$\Pi|_U:\tilde{\Sf{X}}|_U\to X\times U.$$
\end{notation}

Now we associate a special reflexive module to each point of $\mathfrak{M}_{\Sf{G}}$. Let $(\pi,\varphi)\in\mathfrak{M}_{\mathfrak{G}}$.

A curve $D\subset\tilde{\Sf{X}}|_{(\pi,\varphi)}$ is $(\pi,\varphi)$-{\em appropriate} if 
\begin{enumerate}
 \item it is a disjoint union of smooth curvettes which meet the exceptional divisor of $\pi$ transversely, and for any vertex of $\mathfrak{M}_{\Sf{G}}$ the number of curvettes meeting the divisor that $\varphi$ assigns to this vertex is exactly the number of arrows attached to this vertex,
 \item the minimal number of generators of $\pi_*\Ss{D}$ as a $\Ss{X}$ module is minimal among the curves with the previous property.
\end{enumerate}
In Remark~\ref{rem:Dgenericanofactorlibre} and its proof it is shown that a generic curve having the first property also satisfies the second. 

Let $D$ be a $(\pi,\varphi)$-appropriate curve.
Let $(\psi_1,...,\psi_r)$ be a minimal set of generators of the $\Ss{X}$-module $\pi_*\Ss{D}$. 
Applying the correspondence of Theorem~\ref{th:corrsing} to $(\pi_*\Ss{D},(\psi_1,...,\psi_r))$ we obtain  
$(M_{(\pi,\varphi)},(\phi_1,...,\phi_r))$, where $M_{(\pi,\varphi)}$ is special reflexive $\Ss{X}$-module of rank $r$ (specialty is by 
Proposition~\ref{prop:Aeslanormalizacion}). In Sections~\ref{sec:1stchernspecial} and~\ref{sec:clasdef} it is proved that $M_{(\pi,\phi)}$ does not 
depend on the  $(\pi,\varphi)$-appropriate curve. 

However, there is no reason for which the rank $r$ should be independent on the combinatorial type. In fact the rank of $M_{(\pi,\varphi)}$ induces a 
stratification in $\mathfrak{M}_{\Sf{G}}$ because of the following lemma. 

\begin{lemma}
\label{lem:rankuppersem}
The function $rank(M_{(\pi,\varphi)})$ is upper-semicontinuous in $\mathfrak{M}_{\Sf{G}}$ for the Zariski topology.
\end{lemma}
\proof
The assertion is local. Consider $(\pi,\varphi)\in\mathfrak{M}_{\Sf{G}}$. Given a small neighborhood $U$ of $(\pi,\varphi)$ 
it is easy to construct a subscheme $\overline{D}\subset \tilde{\Sf{X}}|_U$ such that for any $(\pi',\varphi')\in U$ 
we have that the fibre $\overline{D}|_{(\pi',\varphi')}$ is a $(\pi',\varphi')$-appropriate curve and the $\Ss{\tilde{\Sf{X}}|_U}$-module 
$\Ss{\overline{D}}$ is flat over $U$ (this is easy by the genericity statement given in Remark~\ref{rem:Dgenericanofactorlibre}). Since $R^1\Pi_*\Ss{\overline{D}}$ vanishes, Lemma~\ref{lem:h0h1} implies that $\Pi_*\Ss{\overline{D}}$ is
flat over $U$ and that $(\Pi_*\Ss{\overline{D}})|_{(\pi',\varphi')}$ equals $(\pi')_*(\Ss{\overline{D}}|_{(\pi',\varphi')})$. Then, the minimal number of
generators of $(\Pi_*\Ss{\overline{D}})|_{(\pi',\varphi')}$ as $\Ss{X}$-module is upper-semicontinuous in the Zariski topology, and coincides
with the rank of $M_{(\pi',\varphi')}$. 
\endproof

Denote by $\mathfrak{M}^r_{\Sf{G}}$ to be the locally closed subset corresponding to modules of rank $r$. By Theorem~\ref{Teo:final}, the closed 
points of $\mathfrak{M}^r_{\Sf{G}}$ are in a bijection with the set of reflexive modules without free factors of graph $\Sf{G}$ and rank $r$.
Our next step is to construct a universal family over each $\mathfrak{M}^r_{\Sf{G}}$.

\begin{lemma}
\label{lem:uniqueunivfamily}
Let $U$ be an open subset of $\mathfrak{M}^r_{\Sf{G}}$. For $i=1,2$ let $\overline{D}^i\subset \tilde{\Sf{X}}|_U$ such that for any $(\pi',\varphi')\in U$ 
we have that the fibre $\overline{D}^i|_{(\pi',\varphi')}$ is a $(\pi',\varphi')$-appropriate curve and the $\Ss{\tilde{\Sf{X}}|_U}$-module 
$\Ss{\overline{D}^i}$ is flat over $U$. Let $(\overline{\psi}^i_1,...,\overline{\psi}^i_r)$ be a system of generators of $\Ss{\overline{D}^i}$ as a $\Ss{X\times S}$-module. Let $(\overline{M}^i,(\overline{\phi}^i_1,...,\overline{\phi}^i_r))$ be the result of applying the correspondence of Theorem~\ref{th:dirXdef} to $(\Pi_*\Ss{\overline{D}^i},(\overline{\psi}^i_1,...,\overline{\psi}^i_r))$. Then we have the 
equality $\overline{M}^1=\overline{M}^2$.
\end{lemma}
\proof
Let $(\overline{\Sf{M}}^i,(\overline{\phi}^i_1,...,\overline{\phi}^i_r))$ be the result of applying the correspondence 
of Theorem~\ref{th:dirresdef} to $(\Ss{\overline{D}^i},(\overline{\psi}^i_1,...,\overline{\psi}^i_r))$. By Proposition~\ref{prop:comparecorrdef} we have the equality $\Pi_*\overline{\Sf{M}}^i=\overline{M}^i$. Repeating in family the arguments of Section~\ref{sec:1stchernspecial} we obtain that $\overline{M}^i$ is determined by the line bundle $\Ss{\tilde{\Sf{X}}|_U}(-\overline{D}^i)$. 

An argument like in Lemma~\ref{lem:varioD} show the isomorphism  $\Ss{\tilde{\Sf{X}}|_U}(-\overline{D}^1)\cong \Ss{\tilde{\Sf{X}}|_U}(-\overline{D}^2)$. 
\endproof

\begin{lemma}
\label{lem:univfamily}
There exists a unique family of special modules $\overline{M}^r_{\Sf{G}}$ with graph $\Sf{G}$ and rank $r$ over
$\mathfrak{M}^r_{\Sf{G}}$ such that for any $(\pi,\varphi)\in\mathfrak{M}^r_{\Sf{G}}$ we have the isomorphism 
$\overline{M}^r_{\Sf{G}}|_{(\pi,\varphi)}=M_{(\pi,\varphi)}$.
\end{lemma}
\proof
The previous Lemma shows how to construct the family locally, and it shows, using Theorem~\ref{th:dirresdef} for suitable generic sections, that the result is unique up to the choices made. So the procedure glues well to a global universal family. 
\endproof

\begin{theorem}
\label{theo:moduli}
The variety $\mathfrak{M}^{r}_{\Sf{G}}$ represents the functor $\mathbf{Mod^r_{\Sf{G}}}$.
\end{theorem}
\proof
Let $\overline{M}$ be a family of special modules without free factor with graph $\Sf{G}$ and rank $r$ over a normal base $S$. Consider the mapping 
$$\theta:S\to \mathfrak{M}^r_{\Sf{G}}$$ 
sending $s$ to the unique point of $(\pi,\varphi)\in\mathfrak{M}^r_{\Sf{G}}$ such that we have the 
isomorphism $\overline{M}|_s\cong \overline{M}^r_{\Sf{G}}|_{(\pi,\varphi)}$. We have to prove that the map is a complex analytic morphism, and that the pullback of the the universal family over $\mathfrak{M}^{r}_{\Sf{G}}$ gives back the family $\overline{M}$.

The infinitely near points that one need to blow up to get $\pi$ from the minimal resolution in $X$ are partially ordered as follows: the first 
generation points is the set of points appearing in the minimal resolution,  the second generation set of points are those appearing in 
the resolution obtained by blow up the first generation points, and succesively. Let $k$ be the maximal generation order 
of the infinitely near points needed to obtain $\pi$.
Let $\Sf{G}^{o}_i$ be the dual graph of the result of blowing up the $i$-th generation points. We have a natural sequence of morphisms
\begin{equation}
\label{eq:sequencegeneration}
\mathfrak{M}^r_{\Sf{G}}\hookrightarrow\mathfrak{M}_{\Sf{G}}\to \mathfrak{M}_{\Sf{G}^{o}}=\mathfrak{M}_{\Sf{G}^{o}_k}\to \mathfrak{M}_{\Sf{G}^{o}_{k-1}}\to
...\to \mathfrak{M}_{\Sf{G}^{o}_{1}},
\end{equation}
where the first morphism is a locally closed inclusion and the second morphism is an etale covering.

We will prove by induction that the composition 
$$\theta_i:S\to \mathfrak{M}_{\Sf{G}^{o}_i}$$
is a complex analytic  morphism for any $i$. 

The initial step of the induction runs as follows. Let $\pi:\Rs_{min}\to X$ be the minimal resolution. Let 

and $\Pi:\Rs_{min}\times S\to X\times S$ be the product of the map $\pi$ with the identity at $S$. Denote by $E_{\Pi}$ the exceptional divisor of $\Pi$. For any $s\in S$ denote by $\Sf{M}_s$  
the full $\Ss{\Rs_{min}}$-module associated with $\overline{M}|_s$. By Proposition~\ref{prop:genericlifting} there is a Zariski open subset $V$ in
$S$ and a flat $\Ss{\Rs_{min}\times S}$-module $\overline{\Sf{M}}_V$ such that $R^1\Pi_*\overline{\Sf{M}}_V$ is flat, for any $s\in V$ the restriction $\overline{\Sf{M}}_V|_s$ is isomorphic to $\Sf{M}_s$, and we have $\Pi_*\overline{\Sf{M}}_V=\overline{M}|_V$.

Fix $s_0\in S$. Choose a neighborhood $U$ of $s_0$ in $S$ and sections $(\overline{\phi}_1,...,\overline{\phi}_r)$ 
of $\overline{M}|_U$ so that the sections $(\overline{\phi}_1|_{s},...,\overline{\phi}_r|_{s})$ are nearly generic for 
\begin{itemize}
 \item the $\Ss{\Rs_{min}}$-module $\Sf{M}_{s}$ for any $s$ in $U\cap V$ or $s=s_0$,
 \item the $\Ss{X}$-module $\overline{M}_{s}$ for any $s\in U$. 
\end{itemize}
Choosing $(\overline{\phi}_1,...,\overline{\phi}_r)$ generic and perhaps shrinking $U$ suffices to guarantee the properties above. Since the assertion we want to prove is local in $S$ we may assume $U=S$ and $V\subset U$ to save some notation.

Applying the correspondence of Theorem~\ref{th:dirXdef} to 
$(\overline{M},(\overline{\phi}_1,...,\overline{\phi}_r))$ we obtain $(\overline{\Sf{C}},(\overline{\psi}_1,...,\overline{\psi}_r))$. Since Theorem~\ref{th:dirXdef} establishes an isomorphism of deformation functors, applying the correspondence of Theorem~\ref{th:corrsing} to $(\overline{M}|_{s},(\overline{\phi}_1|_{s},...,\overline{\phi}_r|_{s}))$, we obtain the pair $(\overline{\Sf{C}}|_{s},(\overline{\psi}_1|_s,...,\overline{\psi}_r)|_s))$.

The correspondence of Theorem~\ref{th:corres} applied to $(\Sf{M}_{s},(\overline{\phi}_1|_{s},...,\overline{\phi}_r|_{s}))$ gives a pair $(\Sf{A}_{s},(\psi_{1,s},...,\psi_{r,s}))$, where 
$\Sf{A}_{s}$ is $1$-dimensional Cohen-Macaulay $\Ss{\Rs_{min}}$-module generated by the sections $(\psi_{1,s},...,\psi_{r,s})$. By Proposition~\ref{prop:dirressing} and specialty we have that $\pi_*\Sf{A}_{s}$ equals $\overline{\Sf{C}}|_{s}$, and we have the identification of sections $(\psi_{1,s},...,\psi_{r,s})=(\overline{\psi}_{1}|_s,...,\overline{\psi}_{r}|_s)$.

The correspondence of Theorem~\ref{th:dirresdef} applied to $(\overline{\Sf{M}}_V,(\overline{\phi}_1,...,\overline{\phi}_r))$ gives a pair $(\overline{\Sf{A}}_V,(\overline{\psi}_{1,V},...,\overline{\psi}_{r,V}))$. By Proposition~\ref{prop:defdirressing} and specialty we have that $\Pi_*(\overline{\Sf{A}}_V)$ equals $\overline{\Sf{C}}|_V$, and the sections $(\overline{\psi}_{1,V},...,\overline{\psi}_{r,V})$ are identified with the restriction $(\overline{\psi}_1|_V,...,\overline{\psi}_r|_V)$ of the sections $(\overline{\psi}_1,...,\overline{\psi}_r)$ over $V$. Since Theorem~\ref{th:dirresdef} establishes an isomorphism of deformation functors, for any $s\in V$ we have the equality of pairs  
$$(\Sf{A}_{s},(\psi_{1,s},...,\psi_{r,s}))=(\overline{\Sf{A}}_V|_s,(\overline{\psi}_1|_s,...,\overline{\psi}_r)|_s).$$

For any $s\in V$, let $A_s$ be the support of $\Sf{A}_{s}$. By the construction of the minimal adapted resolution (see the proof of 
Proposition~\ref{prop:minadap}), the intersection of $A_s\cap E_{\Pi}$ is a finite set containing the set of 
infinitely near points of first generation of $\Pi|_s$. If $A_V$ is the support of $\overline{\Sf{A}}_V$ we have that the fibre 
$A_V|_s$ of $A_V$ over $s$ equals $A_s$. This implies that the graph of the restriction 
$$\theta_1|_V:V\to\mathfrak{M}_{\Sf{G}^{o}_1}$$
is a complex analytic subvariety of $V\times \mathfrak{G}^{o}_1$. Using Zariski's Main Theorem this implies that $\theta_1|_V$ is a complex analytic morphism since $V$ is normal. 

\textsc{Claim I}. Let $\overline{C}\subset X\times S$ be the support of $\overline{\Sf{C}}$. Then $\overline{C}$ lifts to $\Rs_{min}\times S$.

Assume that the claim is true. Denote by $\tilde{C}$ the strict transform $\tilde{C}$ of $C$. Then, for any $s\in S$  the intersection of $\tilde{C}|_s\cap E_{\Pi}$ is a finite set and moreover the fibre $\tilde{C}|_s$ is the support of $\Sf{A}_s$ for $s\in S$. Consequently
$\tilde{C}|_s\cap E_{\Pi}$ contains the set of infinitely near points of first generation of $\Pi|_s$. This implies that the graph of
$\theta_1$ is a complex analytic subvariety of $U\times \mathfrak{G}^{o}_1$, and using Zariski's Main Theorem this implies that $\theta_1|_U$ is a complex analytic morphism since $U$ is normal.

Let us prove the claim. By Lemma~\ref{lem:curvecriterionliftability} we may assume
that $(S,s_0)$ is a germ of smooth curve. 
The fibre of $\hat{C}_s$ over a generic point of $S$ equals $A_s$. The fibre $\hat{C}|_{s_0}$
equals $A_{s_0}+F$ where $F$ is a divisor in $\Rs_{min}$ with support in $E$, which is non-negative.
Now, since the modules $\overline{M}|_s$ have the same combinatorial type for any $s\in S$, we have the equality of intersection numbers $E_i\centerdot A_s=E_i\centerdot A_{s_0}$ for any component $E_i$ of $E$. 
This implies the vanishing $E_i\centerdot F=0$, and hence $F=0$ by the non-degeneracy of the intersection form. This shows that $F$ vanishes, and shows the liftability. 

This sets up the initial step of the induction. The inductive step is completely similar. So, we have shown that $\theta_k$ is an analytic morphism.

Now, since the first morphism of the sequence~(\ref{eq:sequencegeneration})
is a locally closed inclusion, and the second is an etale map, we conclude that $\theta$ is a complex analytic morphism. 

Observe that our induction procedure shows that the support $\overline{C}$ of $\overline{\Sf{C}}$ lifts to the universal family of minimal adapted resolutions $\Rsd$. Then an application of Theorem~\ref{th:sufficientlifting} shows that $\overline{M}$ lifts to a full family 
$\overline{\Sf{M}}$ on $\Rsd$. In order to prove that $\overline{M}$ coincides with the pullback by $\theta$ of the universal family it is 
enough to show it locally, but this is a consequence of Lemma~\ref{lem:uniqueunivfamily} and the procedure in which $\overline{\Sf{M}}$ is constructed in the proof of Theorem~\ref{th:sufficientlifting}.
\endproof

\end{document}